\documentclass[12pt]{amsart}
\usepackage[svgnames]{xcolor}
\usepackage[utf8]{inputenc}
\usepackage{amsmath}
\usepackage{amssymb}
\usepackage{amsthm}
\usepackage{hyperref}
\hypersetup{
  colorlinks   = true, 
  urlcolor     = blue, 
  linkcolor    = blue, 
  citecolor   = red 
}
\usepackage{amsfonts}
\usepackage{amsrefs}
\usepackage{tikz}
\usepackage{tikz-cd}
\usepackage{mathtools}
\usepackage{url}
\usepackage{makecell}
\usepackage{tikz-3dplot}
\usepackage{newtxtext, newtxmath}
\definecolor{lightgray}{rgb}{.9,.9,.9}
\definecolor{gray}{rgb}{.7,.7,.7}
\definecolor{darkgray}{rgb}{.5,.5,.5}
\definecolor{darkestgray}{rgb}{.3,.3,.3}

\tikzcdset{scale cd/.style={every label/.append style={scale=#1},
    cells={nodes={scale=#1}}}}

\makeatletter
\newtheorem*{rep@theorem}{\rep@title}
\newcommand{\newreptheorem}[2]{%
\newenvironment{rep#1}[1]{%
 \def\rep@title{#2 \ref{##1}}%
 \begin{rep@theorem}}%
 {\end{rep@theorem}}}
\makeatother

\usepackage[ left = 1.2in, right = 1.2in, 
             top = 1.2in, bottom = 1.2in ]{geometry}

\DeclareMathOperator{\sqdeg}{\mathop{sqdeg}}
\DeclareMathOperator{\Gr}{\mathop{Gr}}
\DeclareMathOperator{\Tor}{\mathop{Tor}}
\DeclareMathOperator{\rk}{\mathop{rk}}
\DeclareMathOperator{\ea}{\mathop{ea}}
\DeclareMathOperator{\ia}{\mathop{ia}}
\DeclareMathOperator{\ep}{\mathop{ep}}
\DeclareMathOperator{\EP}{\mathop{EP}}
\DeclareMathOperator{\IN}{\mathop{IN}}

\DeclareMathOperator{\Sym}{\mathop{Sym}}

\DeclareMathOperator{\spann}{\mathop{span}}

\DeclareMathOperator{\Star}{\mathop{St}}

\DeclareMathOperator{\supp}{\mathop{supp}}
\DeclareMathOperator{\coker}{\mathop{coker}}
\DeclareMathOperator{\lex}{\mathrm{lex}}
\DeclareMathOperator{\HPoin}{\mathrm{HPoin}}
\DeclareMathOperator{\pos}{\mathop{pos}}
\DeclareMathOperator{\NBC}{\mathop{NBC}}
\DeclareMathOperator{\closure}{\mathop{cl}}

\newtheorem{thm}{Theorem}[section]
\newreptheorem{theorem}{Theorem}
\newreptheorem{theorem*}{Theorem*}
\newtheorem*{thm*}{Theorem}
\newtheorem{cor}[thm]{Corollary}
\newreptheorem{cor}{Corollary}
\newtheorem{lem}[thm]{Lemma}

\newtheorem*{claim*}{Claim}
\newtheorem{prop}[thm]{Proposition}

\theoremstyle{definition}
\newtheorem{definition}[thm]{Definition}
\theoremstyle{remark}
\newtheorem{rmk}[thm]{Remark}
\newtheorem{example}[thm]{Example}

\DeclareEmphSequence{\bfseries}

\title{The Singular 
		Cohomology Ring of a Matroid}
\author{Kyle Binder}
\address{Department of Mathematics, Louisiana State University}
\email{kbinde1@lsu.edu}
\begin{document}

\begin{abstract}
	We introduce the singular cohomology ring of a matroid which extends the 
	Chow ring of a matroid. This is defined as the singular cohomology ring
	of a certain quasi-projective toric variety associated to the matroid.
	Using the matroidal flips of Adiprasito,
	Huh, and Katz, we prove sharp vanishing results for the cohomology ring and
	compute the dimension of the top-weight cohomology in terms of the M\"{o}bius
	invariant of the matroid. In the case of uniform matroids,
	these techniques give a recursive formula for the 
	Hodge numbers. 
	Finally, we generalize the singular cohomology ring to arbitrary 
	building sets on the lattice of flats, and we show how
	the cohomology depends on the building set.
\end{abstract}

\maketitle

 	\section{Introduction}
		The Chow ring $ A^{\bullet}(M) $ of a loopless matroid $M$ 
		is defined to be the Chow ring of the toric variety $ X_{\Sigma_{M}} $ 
		associated to the \textit{Bergman fan} $ \Sigma_{M} $. This ring has several
		surprising properties.
		Even though $ X_{\Sigma_{M}} $ is often not projective, only quasi-projective,
		Adiprasito, Huh, and Katz \cite{AHK} show that
		$ A^{\bullet}(M) $ acts like the Chow ring of a projective variety, satisfying
		the \textit{K\"{a}hler package} of Poincar\'{e} duality, Hard Lefschetz, 
		and the Hodge--Riemann relations. These properties of the Chow ring,
		as well as the combinatorial information encoded in various 
		intersection numbers, were foundational to the resolution of
		Mason's Conjecture and the Heron--Rota--Welsh Conjecture.
		Since their introduction, Chow rings remain a central object in the geometric
		theory of matroids.

		Due to the prevalence of the Chow ring of the toric variety $ X_{\Sigma_{M}} $,
		it is natural to ask what other algebro-geometric invariants of 
		$ X_{\Sigma_{M}} $ look like. In this paper we define the 
		\textit{singular cohomology ring of a matroid $M$}  to be the
		singular cohomology ring $ H^{\bullet}(X_{\Sigma_{M}}) $ of
		$ X_{\Sigma_{M}} $. (All Chow and cohomology groups here are taken with rational coefficients.)
		This extends the Chow ring of a matroid, as the cycle class map
		$ A^{\bullet}(M) \hookrightarrow H^{2 \bullet}(X_{\Sigma}) $ is injective.
		Moreover, this extension is often interesting, as
		quasi-projective varieties like $ X_{\Sigma_{M}} $
		frequently have cohomology not seen by the Chow ring.

		Like the Chow ring of a matroid, the singular cohomology ring
		has an explicit combinatorial description. 
		Recall Feichtner--Yuzvinsky's presentation \cite{FY} for
		the Chow ring of a matroid $M$ on the ground set $ [n] := 
		\left\{ 1, \dots, n \right\} $: 
		\begin{equation}\label{eq:FYPresentation}
						A^{\bullet}(M) \cong \frac{\mathbb{Q}[x_{F} : \textrm{$F$ proper,
						non-empty flat}]}{I + J}
						\cong \frac{\mathbb{Q}[\Sigma_{M}]}
							{J},
		\end{equation}
		where $ I := (x_{F} x_{G} : \textrm{$F$, $G$ incomparable} ) $
		is the \textit{Stanley--Reisner ideal},
		$ \mathbb{Q}[\Sigma_{M}] := \mathbb{Q}[x_{F}]_{F} /I $ is
		the \textit{Stanley--Reisner ring}, and
		$ J = (\sum_{i \in F} x_{F} - \sum_{j \in G} x_{G}: i,j \in [n] ) $.
		The ideal $J$ is naturally the image of 
		the map
		\begin{equation*}
						\mathbb{Q}[\ell_{i} - \ell_{j} : i,j \in [n]] \subseteq 
						\mathbb{Q}[\ell_{1}, \dots, \ell_{n}] \to \mathbb{Q}[\Sigma_{M}],\;\;\;
						\ell_{i} - \ell_{j} \longmapsto \sum_{i \in F} x_{F} - \sum_{j \in G } x_{G}.
		\end{equation*}
		For the singular cohomology ring, Franz \cite{franzRing}*{Theorem 1.2} gives a ring isomorphism
		\begin{equation}\label{franzIsomorphism}
				H^{\bullet}(X_{\Sigma_{M}}) \cong
				\Tor_{\bullet}^{\mathbb{Q}[\ell_{i} - \ell_{j}]_{i,j}}
				(\mathbb{Q}[\Sigma_{M}],
				\mathbb{Q}),
		\end{equation}
		and these Tor groups have a description in terms of 
		Koszul homology $H_{\bullet}(K(\mathbb{Q}[\Sigma_{M}])) $. 
		This is defined to be the homology of the \textit{Koszul complex}
		$ K_{\bullet}(\mathbb{Q}[\Sigma_{M}]) $ whose differential is defined combinatorially
		from $M$.
		This presentation allows us to explicitly
		compute the singular cohomology of small matroids (for instance, see
		Examples \ref{ex:U23} and \ref{ex:U34}), and bases for Koszul 
		homology appear to exhibit a strong combinatorial structure (for example,
		see Lemma \ref{lem:koszulH1}). Additionally, the isomorphism 
		\eqref{franzIsomorphism} constructs another pathway between 
		algebraic geometry and combinatorial commutative algebra, and both
		perspectives are indispensible in our study of the singular cohomology ring
		of a matroid.

		\subsection{Main results}
		Our goal for this paper is to describe the Betti numbers, more
		specifically the \textit{Hodge numbers}, of $ X_{\Sigma_{M}} $ in terms of
		the combinatorics of the matroid $M$. Our main theorem is a sharp vanishing
		result for $H^{\bullet}(X_{\Sigma_{M}}) $ and a computation of the top-weight
		cohomology. 

		\begin{reptheorem}{thm:vanishingCohomology}
			Let $ M$ be a loopless matroid of rank $ r $ on the ground set $[n]$. Then,
			\begin{enumerate}
				\item $ \Gr_{j}^{W} H^{k}(X_{\Sigma_{M}}) = 0 $ whenever
								$ j > n -r+k $ or $ j < 2k -2r + 2 $.
				\item $ H^{k}(X_{\Sigma_{M}}) = 0 $ whenever $ k > n + r - 2 $.
			\end{enumerate}
			Moreover, $ \Gr^{W}_{2n-2} H^{n+r-2}(X_{\Sigma_{M}}) $ is the top-weight
			cohomology, and its dimension is $ \mu(M) $.
		\end{reptheorem}
		Here, $ \mu(M) $ is the \textit{M\"{o}bius invariant} of $M$ which counts
		the number of \textit{no-broken-circuit bases}. It is non-zero for
		any loopless matroid. In fact, we prove a stronger result in Theorem 
		\ref{thm:vanishingCohomology} and show the same vanishing for toric
		varieties $ X_{\Sigma_{M, \mathcal{P}}} $ associated to the \textit{Bergman
		fan of an order filter} $ \mathcal{P} $.

		For uniform matroids, we give a recursive 
		formula for all of the Hodge numbers of $ X_{\Sigma_{U_{r,n}}} $. This is 
		best packaged in terms of the \textit{Hodge--Poincar\'{e} polynomial},
		$\HPoin(X_{\Sigma_{U_{r,n}}}) $, which is the generating
		function for the Hodge numbers (Definition \ref{def:hpPolynomial}).

		\begin{reptheorem}{thm:uniform}
			Let $ U_{r,n} $ be the uniform matroid of rank $ r $ on 
			the ground set $[n]$. Then,
			\[
			\HPoin\left( X_{\Sigma_{U_{r,n}}} \right)
			= 	\sum_{i=0}^{r-1} x^{i} + \sum_{k=0}^{n-r-1} \binom{r+k}{r-1} \cdot 
				wx^{r}(1+wx)^{k} + \sum_{i=1}^{r-1} \binom{n}{i} \frac{x-x^{i}}{1-x}
			\HPoin\left( X_{\Sigma_{U_{r-i,n-i}}} \right).
			\]
		\end{reptheorem}

		Bergman fans and Chow rings of matroids have been generalized with respect to
		arbitrary building sets on the lattice of flats $ \mathcal{L}(M) $. We generalize the
		singular cohomology ring for an arbitrary building set $ \mathcal{G} $
		as the singular cohomology of	$ X_{\Sigma_{M, \mathcal{G}}} $,
		where $ \Sigma_{M, \mathcal{G}} $ is the Bergman fan associated to $ \mathcal{G} $.
		It is well known that for a given matroid $M$, the Bergman fans associated to
		any two building sets on the lattice of flats are related by
		a sequence of stellar subdivisions.  We describe how the Hodge--Poincar\'{e}
		polynomial changes under one of these stellar subdivisions.

		\begin{reptheorem}{thm:changeBuildingSet}
				Let $M$ be a loopless matroid on the ground set $[n]$, and let 
				$ \mathcal{H} $ and $ \mathcal{G} $ be two building sets on
				$ \mathcal{L}(M) $ which both contain $[n]$ 
				and whose difference is a single flat $ Z =
				\mathcal{H} \setminus \mathcal{G} $. 
				Write $ \left\{ F_{1}, \dots, F_{s} \right\} = f(\mathcal{G}\big|_{Z}) $
				for the factors of $ Z $ in $ \mathcal{G} $.
				Then,
				\[ 
						\HPoin\left(X_{\Sigma_{M, \mathcal{H}}}\right) 
							= \HPoin\left(X_{\Sigma_{M, \mathcal{G}}} \right)
								+ \frac{1 - x^{s}}{1-x}
								\cdot \prod_{i=1}^{s} 
								\HPoin\left(X_{\Sigma_{M^{F_{i}}, \mathcal{G}^{F_{i}}}}
								\right) \cdot \HPoin\left( X_{\Sigma_{M_{Z}},
									\mathcal{G}_{Z}} \right). 				\]
		\end{reptheorem}

		This result generalizes \cite{EFMPV}*{Theorem 5.1} which shows that the Hilbert
		series of the Chow ring behaves similarly when changing the building
		set.

		\subsection{Outline} 
			Let us describe the structure of this paper. 

			In Section \ref{section:background},
			we recall the background and notation we need on matroids, Bergman fans,
			toric varieties, and Koszul homology. Section \ref{section:singCohomology}
			defines the \textit{singular cohomology ring of a matroid}, and we compute
			some small examples to illustrate Theorem \ref{thm:vanishingCohomology}.
			We also show that the singular cohomology ring is sensitive to 
			parallel elements in the matroid, unlike the Chow ring of the matroid.

			Sections \ref{section:liftEmpty}--\ref{section:flips} form the heart 
			of the paper in proving Theorem \ref{thm:vanishingCohomology}.
			There are three main stages to the proof. In the first two stages,
			we define toric varieties, $ X_{\widehat{\Sigma}_{M, \emptyset}
			\setminus \rho_{[n]}} $
			and $ X_{\Sigma_{M, \emptyset}} $ which are related to $ X_{\Sigma_{M}} $ and
			whose cohomology we compute through combinatorial commutative algebra.
			The Hodge--Poincar\'{e} polynomilas of these varieties are interesting in their own right.
			In Section \ref{section:liftEmpty}, we compute $ \HPoin(X_{\widehat{\Sigma}_{M, \emptyset}
			\setminus \rho_{[n]}}) $ through an application of Hochster's formula to toric subvarieties
			of affine space. Surprisingly, this Hodge--Poincar\'{e} polynomial is a specialization of 
			the Tutte polynomial of $M$. 
			In Section \ref{section:empty}, we use the weight filtration on $H^{\bullet}(X_{\Sigma_{M, 
			\emptyset}}) $ to then deduce
			$ \HPoin(X_{\Sigma_{M, \emptyset}}) $. 
			The last stage of the proof is Section \ref{section:flips}, where
			we construct a sequence of toric morphisms interpolating between 
			$ X_{\Sigma_{M, \emptyset}} $ and $ X_{\Sigma_{M}} $ using \textit{matroidal flips}. 
			We show that each map in this sequence is close to a blow-up, and we construct \textit{the long
			exact sequence of a matroidal flip} comparing cohomology under these maps.
			These exact sequences produce an isomorphism on the top-weight cohomology which we use
			to deduce Theorem \ref{thm:vanishingCohomology}. 

			In Sections \ref{section:Chow} and \ref{section:uniform}, we use the long exact 
			sequence of a matroidal flip in the context of Chow rings and uniform matroids, respectively.
			The long exact sequence of a matroidal flip induces a short exact sequence of Chow rings,
			which recovers \cite{AHK}*{Theorem 6.18}.
			In the context of uniform matroids, we show that long exact sequences of matroidal flips split,
			and we use this compute the Hodge--Poincar\'{e} polynomial of $ X_{\Sigma_{U_{r,n}}} $. 
			We describe how the Hodge numbers of $ X_{\Sigma_{M, \mathcal{G}}} $
			vary for different choices of building set $ \mathcal{G} $ in
			Section \ref{section:buildingSets}. 

			At the end of this paper,
			we have included Appendix \ref{section:blowups} on the cohomology of 
			toric blow-ups. Here, we use explicit computations 
			in Koszul homology to prove the classical result relating the
			cohomology of the blow-up in terms of the cohomology of the base 
			space, center, and exceptional divisor. We rely upon these results
			and computations throughout the paper.

			\subsection*{Acknowledgements}
			We thank Eric Katz, whose thoughts benefited from his discussions 
			with David Speyer, for his suggestion of this problem, and David Anderson 
			for discussions on the cohomology of 
			toric varieties. This work also benefited from helpful 
			conversations with Juliette Bruce, Alex Fink, and Matt Larson.
			This material is based upon work supported by the National Science 
			Foundation under Grants DMS-2231565 and DMS-2231492.

	\section{Background and notation}\label{section:background}
		In this section we introduce the background and notation we need 
		on matroid theory, Bergman fans, toric geometry, and Koszul homology.

	\subsection{Matroids}
		For details on matroids not defined here, see \cite{oxley}.

		Let $M$ be a matroid on the ground set $[n] = \left\{ 1, \dots, n \right\}$, 
		let $ \mathcal{B}(M) $ denote the set of 
		\emph{bases} of $M$, and let $ \mathcal{L}(M) $  be the \emph{lattice of 
		flats} of $M$. A \emph{proper, non-empty flat} is an element of 
		$ \overline{\mathcal{L}}(M) := \mathcal{L}(M) \setminus
		\left\{ \hat{0}, \hat{1} \right\} $, where $ \hat{0} $ and 
		$ \hat{1} $ are the minimal and maximal elements of $ \mathcal{L}(M) $,
		respectively.  		
		For a subset $ S \subseteq [n] $, we denote its closure by $ \closure(S) $
		and its rank by $ \rk(S) $.
		We write $ M^{*} $ for the dual matroid of $M$. By definition,
		$ \mathcal{B}(M^{*}) = \left\{ [n] \setminus B : B \in \mathcal{B}(M) \right\} $.

		Given a flat $ F$, $ M_{F} $ is the
		\emph{contraction} of $M$ to $ F $. The lattice $ \mathcal{L}(M_{F}) $
		is isomorphic to the interval $ [F, \hat{1}] $ in $ \mathcal{L}(M) $.
		Similarly, $ M^{F} $ is the \emph{restriction} of $M$ to $ F $.
		The lattice $ \mathcal{L}(M^{F}) $ is isomorphic to the interval
		$ [ \hat{0}, F] $. We can also define the \emph{restriction} of $M$ to
		$ W$ for any subset $ W \subseteq [n] $ of the ground set. We denote this by
		$ M\big|_{W} $.

		Our running convention is that $M$ is a loopless matroid of rank $ r \geq 1 $,
		defined on the ground set $ [n]$. Importantly, we allow $M$ to have parallel 
		edges, as the singular cohomology ring
		of $M$ is sensitive to these elements (see Proposition \ref{prop:parallelEdges}).

		\begin{example}
			An important class of matroids are the \textit{uniform matroids}. For 
			integers $ n $ and $ r $ with $ n \geq 1 $ and $ 1 \leq r \leq n $, the 
			uniform matroid of rank $ r $ on the ground set $ [n] $ is the matroid
			$U_{r,n} $ with bases
			\[ 
				\mathcal{B}(U_{r,n}) = \left\{ B \subseteq [n] : \left| B \right| = r \right\}.
			\]
			The lattice of flats $ \mathcal{L}(U_{r,n}) $ consists of the set
			$[n]$ along  with all subsets of
			$ [n] $ of cardinality at most $ r-1$. 		
		\end{example}

		\subsubsection{Basis activity}
			The natural linear order on $[n]$ defines \textit{basis 
			activity} for $M$.

		\begin{definition}
			Fix a basis $ B \in \mathcal{B}(M) $. 
			\begin{enumerate}
				\item An element $ x \in [n] \setminus B $ is \emph{externally
									active} with respect to $ B$ if $ x $ is the smallest element of 
									the fundamental circuit $ C_{B \cup x} $. Otherwise,
							$ x $ is \emph{externally passive} with respect to $B$.
			\item An element $ x \in B $ is \emph{internally active} with respect to $B$ if it is
						externally active with respect to the dual basis $[n]\setminus B $
						in the dual matroid $ M^{*} $.
			\end{enumerate}
						We define $ \EP(B) \subseteq [n] $ 
						to be the set of externally active elements with respect to $ B $ and
						$ \ea(B) $, $ \ep(B) $, and $ \ia(B) $ to be the 
						number of externally active, externally passive, and internally
						active elements with respect to $B$, respectively. 
					
		\end{definition}

		The  \emph{Tutte polynomial} is the generating function for bases 
		with certain internal and external activity
			\[ 
				T_{M}(x,y) := \sum_{B \in \mathcal{B}(M)} x^{\ia(B)} y^{\ea(B)}.
			\]
	  This  polynomial specializes to many classical invariants of graphs and
		matroids \cite{brylawskiOxley}.
		For example, the specialization $ T_{M}(1,0) $ counts the number of 
		bases $B$ with $ \ea(B) = 0 $.
		These are called the \emph{no-broken-circuit
		bases} of $M$, and we define $ \NBC(M) $ to be the set
		of no-broken-circuit bases. The cardinality of $ \NBC(M) $
		is the \emph{M\"{o}bius invariant} of $M$, which is written $ \mu(M) $.
		For loopless matroids, $ \mu(M) > 0 $.
					
		\subsection{Bergman fans of order filters}
			We define two families of rational fans that generalize
			the \textit{Bergman fan} of a matroid studied
			in \cite{ardilaKlivans}. Each fan
			in these families arises from a special choice of subset of 
			$ \mathcal{L}(M) \setminus \hat{0}  $,
			either an \textit{order filter} or \textit{building
			set}.
			
			For notation in defining rational fans, let $ e_{i} $ be the
			$i$-th basis element of $ \mathbb{Z}^{n} $. For
			$ S \subseteq [n] $, let $e_{S} := \sum_{i \in S} e_{i}$. We write
			$ \pos(E) $ for the cone of positive real linear combinations of a set
			of points
			$ E \subseteq \mathbb{Z}^{n} $. 
			For a (possibly empty) chain of flats $ \mathcal{F} \subseteq 
			\mathcal{L}(M) \setminus \hat{0}
			$ and 
			subset $ I \subseteq [n] $, we write $ I < \mathcal{F} $ if 
			$I$ is strictly contained in each flat of $\mathcal{F} $.
			To such a pair $ I < \mathcal{F} $, define the cone
			\[ 
				\sigma_{I < \mathcal{F}} := \pos \left(e_{i}, e_{F} : i \in I \textrm{ and }
				F \in \mathcal{F}\right).
			\]

			Recall that an order filter 
			on a poset is a subset which is upward-closed.

			\begin{definition}
			Let $ \mathcal{P} $ be an order filter on $ \mathcal{L}(M) \setminus
			\hat{0} $ which
			contains $ [n] $.
			The \emph{non-reduced Bergman fan} of $ \mathcal{P} $ is
			the unimodular fan
						\[ \widehat{\Sigma}_{M, \mathcal{P}} :=  
							\left\{ \sigma_{I < \mathcal{F}} : 
							\mathcal{F} \subseteq \mathcal{P} \textrm{ and }
							\closure(I) \notin \mathcal{P}
							\right\}.
						\]
			The rays of $ \widehat{\Sigma}_{M, \mathcal{P}} $ are 
			$ \rho_{F} := \pos(e_{F}) $ for $ F \in \mathcal{P} $ and 
			$ \rho_{i} := \pos(e_{i}) $ for $ i \in [n]  $ with $ \closure(i) \notin
			\mathcal{P} $.
		\end{definition}
		As we assume that $ [n] \in \mathcal{P} $,  the fan
		$ \widehat{\Sigma}_{M, \mathcal{P}} $ has lineality space
		$ \mathbb{Z}\cdot e_{[n]} $.

		\begin{definition}
			Let $ \mathcal{P} $ be an order filter on $ \overline{\mathcal{L}}(M) $. 
			The \emph{Bergman fan} of $ \mathcal{P} $, denoted
			$ \Sigma_{M, \mathcal{P}} $, is the quotient fan
			of $ \widehat{\Sigma}_{M, \mathcal{P} \cup [n] } $ in $ \mathbb{Z}^{n}/
			\mathbb{Z}\cdot e_{[n]} $.

			The rays of $ \Sigma_{M, \mathcal{P}} $ are the images of 
			$ \rho_{F} $ for $ F \in \mathcal{P} $ and $ \rho_{i} $ for $ i \in [n]$
			with $ \closure(i) \notin \mathcal{P}\cup [n] $.
		\end{definition}

		\begin{rmk}
		We choose to work with the Bergman fan rather than the non-reduced
		Bergman fan, although this choice is largely irrelevant. The difference amounts to working
		with the \textit{star} of $ \rho_{[n]} $ rather than the \textit{closed star}
		(Definition \ref{def:star}).
		By Lemma \ref{lem:starIsomorphism}, this difference does not matter when computing 
		cohomology.
		\end{rmk}

		The Bergman fans of the maximal and minimal order filters will be most important for us.

		\begin{example}
			For a loopless matroid $M$, take $ \mathcal{P} = \overline{\mathcal{L}}(M) $.
			Then, $ \Sigma_{M, \mathcal{P}} $ is the standard Bergman fan $ \Sigma_{M} $.
			The rays of $ \Sigma_{M} $ are in bijection with the proper,
			nonempty flats of $ M $, and the cones correspond to chains of these flats.
			In other words, $ \Sigma_{M} $ is a realization of the cone over the order
			complex of $ \overline{\mathcal{L}}(M) $.
		\end{example}

		\begin{figure}
        \centering
 \begin{tikzpicture}[scale=.9]
            \node (1234) at (0,0) {$1234$};
            \node (12) at (-3,-2) {$12$};
            \node (13) at (-1,-2) {$13$};
            \node (14) at (1,-2) {$14$};
            \node (234) at (3,-2) {$234$};
            \node (1) at (-3,-4) {$1$};
            \node (2) at (-1,-4) {$2$};
            \node (3) at (1,-4) {$3$};
            \node (4) at (3,-4) {$4$};
            \node (empty) at (0,-6) {$\emptyset$};
            \draw (1234) -- (12);
            \draw (1234) -- (13);
            \draw (1234) -- (14);
            \draw (1234) -- (234);
            \draw (12) -- (1);
            \draw (12) -- (2);
            \draw (13) -- (1); 
            \draw (13) -- (3);
            \draw (14) -- (1);
            \draw (14) -- (4);
            \draw (234) -- (2);
            \draw (234) -- (3);
            \draw (234) -- (4);
            \draw (1) -- (empty);
            \draw (2) -- (empty);
            \draw (3) -- (empty);
            \draw (4) -- (empty);

						\node at (0,-6.75) {$ \mathcal{L}(M^{\mathrm{br}})$};
        \end{tikzpicture}
            \tdplotsetmaincoords{86}{73}
            \begin{tikzpicture}[tdplot_main_coords, scale =1.7]
             \filldraw[darkgray] (0,0,0) -- (-1.2,-1.2,-1.2) -- (0,2,0) -- cycle;
             \filldraw[darkestgray] (-1.2,-1.2,-1.2) -- (2,0,0) -- (0,2,0) -- cycle;
             \filldraw[gray] (0,0,0) -- (0,2,0) -- (0,0,2) --cycle;          
             \filldraw[gray] (0,0,0) -- (0,0,2) -- (-1.2,-1.2,-1.2)--cycle;
             \filldraw[darkgray] (0,0,0) -- (-1.2,-1.2,-1.2) -- (2,0,0) -- cycle;
             \filldraw[darkgray] (0,0,0) -- (2,0,0) -- (0,2,0) -- cycle;
             \filldraw[lightgray] (0,0,0) -- (2,0,0) -- (0,0,2) --cycle;
             \draw[->] (0,0,0) -- (2.1,0,0) node [anchor = south east] {$\rho_{3}$};
             \draw[->] (0,.58,0) -- (0,2.1,0) node[anchor = north west] {$\rho_{2}$};
             \draw[->] (0,0,0) -- (0,0,2.1) node[anchor = south] {$\rho_{1}$};
             \draw[->] (0,0,0) -- (-1.25,-1.25,-1.25) 
											node[anchor= north] {$\rho_{4}$};                     
             \draw[->] (0,.46,.46) -- (0,1.07, 1.07) 
											node [anchor = south west] {$\rho_{12}$};   
             \draw[->] (0,0,0) -- (1.07,0,1.07) 
											node [anchor = south west] {$\rho_{13}$};
             \draw[->] (0,0,0) -- (-.9, -.9,0) node[anchor=east] {$\rho_{14}$};
                    \fill (0,0,0) circle[radius=1pt];
							\node at (0,0,-1.3) {$\Sigma_{M^{\mathrm{br}}, \emptyset}$};
						\end{tikzpicture} \\\vspace*{.5em}
					\begin{tikzpicture}[tdplot_main_coords, scale =1.7]
                    \filldraw[darkestgray] (0,0,0) -- (0,2,0) -- (0,0,2) --cycle;

                    \filldraw[darkestgray] (0,0,0) -- (0,2,0) -- (0,1,1) --cycle;
                    \filldraw[darkestgray] (0,0,0) -- (0,2,0) -- (0,0,-.555) --cycle;

                    \draw[->] (0,0,0) -- (0,1.1, 1.1) node [anchor = south west] {$\rho_{12}$};
                    \draw[->] (0,.37,0) -- (0,2.1,0) node[anchor = north west] {$\rho_{2}$};
                    \filldraw[darkgray] (0,0,0) -- (-1.2,-1.2,-1.2) -- (-.8,-.8,0)--(0,0,2) --cycle;
                    \filldraw[gray] (0,0,0) -- (2,0,0) -- (0,0,2)--cycle;
                    \filldraw[gray] (0,0,0) -- (2,0,0) -- (0,0,-.555)--cycle;
                    \filldraw[darkgray] (0,0,0) -- (-1.2,-1.2,-1.2) -- (0,0,-.555)--cycle;
                    cycle;
                    \draw[->] (0,0,0) -- (2.1,0,0) node [anchor = south east] {$\rho_{3}$};
                    \draw[->] (0,0,0) -- (0,0,2.1) node[anchor = south] {$\rho_{1}$};
                    \draw[->] (0,0,0) -- (-1.25,-1.25,-1.25) node[anchor= north] {$\rho_{4}$};
                     \draw[->] (0,0,0) -- (1.07,0,1.07) node [anchor = south west] {$\rho_{13}$};
                    \draw[->] (0,0,0) -- (-.9, -.9,0) node[anchor=east] {$\rho_{14}$};
                    \draw[->] (0,0,0) -- (0,0,-.8) node[anchor = north east] {$\rho_{234}$};
                    \fill (0,0,0) circle[radius=1pt];

										\node at (0,0,-1.6) {$\Sigma_{M^{\mathrm{br}}}$};
         \end{tikzpicture}%
				\begin{tikzpicture}[tdplot_main_coords, scale =1.7]
                    \filldraw[darkestgray] (0,0,0) -- (0,2,0) -- (0,0,2) --cycle;

                    \filldraw[darkestgray] (0,0,0) -- (0,2,0) -- (0,1,1) --cycle;
                    \filldraw[darkestgray] (0,0,0) -- (0,2,0) -- (0,0,-.555) --cycle;

                    \draw[->] (0,.37,0) -- (0,2.1,0) node[anchor = north west] {$\rho_{2}$};
                    \filldraw[darkgray] (0,0,0) -- (-1.2,-1.2,-1.2) -- (-.8,-.8,0)--(0,0,2) --cycle;
                    \filldraw[gray] (0,0,0) -- (2,0,0) -- (0,0,2)--cycle;
                    \filldraw[gray] (0,0,0) -- (2,0,0) -- (0,0,-.555)--cycle;
                    \filldraw[darkgray] (0,0,0) -- (-1.2,-1.2,-1.2) -- (0,0,-.555)--cycle;
                    cycle;
                    \draw[->] (0,0,0) -- (2.1,0,0) node [anchor = south east] {$\rho_{3}$};
                    \draw[->] (0,0,0) -- (0,0,2.1) node[anchor = south] {$\rho_{1}$};
                    \draw[->] (0,0,0) -- (-1.25,-1.25,-1.25) node[anchor= north] {$\rho_{4}$};
                    \draw[->] (0,0,0) -- (0,0,-.8) node[anchor = north east] {$\rho_{234}$};
                    \fill (0,0,0) circle[radius=1pt];

										\node at (0,0,-1.6) {$\Sigma_{M^{\mathrm{br}}, \mathcal{G}_{\min}}$};
         \end{tikzpicture}
				\caption{ The lattice of flats of the \textit{broom matroid} 
						$M^{\mathrm{br}}$, and various Bergman fans.
						The bottom-most simplex 
						of $ \Sigma_{M^{\mathrm{br}}, \emptyset} $
						is filled in, so the 
						fan is not pure.} \label{fig:nonPureExample}
		\end{figure}
		
		\begin{example}\label{ex:nsComplex}
			For a loopless matroid $M$ on the ground set $ [n] $ with rank at least $ 2 $, 
			take $ \mathcal{P} = \emptyset $. 
			The Bergman fan $ \Sigma_{M, \emptyset} $ has rays
			$ \left\{ \rho_{i} : i \in [n] \right\} $.
			The rays $ \rho_{i_{1}}, \dots, \rho_{i_{k}} $ span a cone in
			$ \Sigma_{M, \emptyset} $ if and only if $ \closure(i_{1}, \dots,
			i_{k}) \neq [n] $.

			Therefore, $ \Sigma_{M, \emptyset} $ is a realization of the cone
			over the \emph{non-spanning complex of $M$}, $\Delta^{ns}_{M} $.
			This is the simplicial complex on $ [n] $ whose faces are the subsets of $[n]$
			whose closure is not $ [n] $.
			These fans are not necessarily pure (see Figure \ref{fig:nonPureExample}).
		\end{example}

		An important property of the family of 
		Bergman fans of order filters is that  they interpolate between 
		$ \Sigma_{M, \emptyset} $ and $ \Sigma_{M} $.
		Specifically, one can construct a sequence of order filters 
		$ \emptyset = \mathcal{P}_{0} \subset \mathcal{P}_{1}
		\subset \cdots \subset \mathcal{P}_{k} = \overline{\mathcal{L}}
		(M) $ such that consecutive order filters only differ by a single flat.
		Each corresponding modification of Bergman fans,
		$ \Sigma_{M, \mathcal{P}_{i}} \rightsquigarrow 
		\Sigma_{M, \mathcal{P}_{i+1}} $, is a \textbf{matroidal flip},
		and the flat $ Z = \mathcal{P}_{i+1} \setminus 
		\mathcal{P}_{i} $ is called the \emph{center} of the matroidal flip. These 
		were introduced in \cite{AHK} to inductively prove 
		the K\"{a}hler package for the Chow ring of a matroid, and we
		use these flips in Section \ref{section:flips}.

		\begin{example}
			Over the course of this section, we will use
			the \textit{broom matroid} $ M^{\mathrm{br}} $
			and uniform matroid $ U_{3,4} $ as our running examples.
			Some of their Bergman 
			fans are pictured in 
			Figures \ref{fig:nonPureExample} and \ref{figure:U34Picture}, respectively.

		\begin{figure}
				\tdplotsetmaincoords{86}{74}
				\begin{tikzpicture}[tdplot_main_coords, scale =1.7]
				  \filldraw[gray] (0,0,0) -- (0,2,0) -- (0,0,2) --cycle;
           \filldraw[gray] 
										(0,0,0) -- (0,0,2) -- (-1.2,-1.2,-1.2)--cycle;
						\filldraw[darkestgray] 
									(0,0,0) -- (-1.2,-1.2,-1.2) -- (0,2,0) -- cycle;
					\filldraw[darkgray] (0,0,0) -- (2,0,0) -- (0,2,0) --
                    cycle;
           \filldraw[darkgray] (0,0,0) -- (-1.2,-1.2,-1.2) -- (2,0,0) -- cycle;
           \draw[->] (0,0,0) -- (0,2.1,0) node[anchor = west] {$2$};
           \filldraw[lightgray] (0,0,0) -- (2,0,0) -- (0,0,2) --cycle;
					\draw[->] (0,0,0) -- (2.1,0,0) node[below right = 1.5pt] {$3$};
           \draw[->] (0,0,0) -- (0,0,2.1) node[anchor = south] {$1$};
           \draw[->] (0,0,0) -- (-1.25,-1.25,-1.25) node[anchor= north east] {$4$};
               \fill (0,0,0) circle[radius=1pt];
          \end{tikzpicture}
				\begin{tikzpicture}[tdplot_main_coords, scale =1.7]
					 \filldraw[gray] (0,0,0) -- (0,2,0) -- (0,0,2) --cycle;
           \filldraw[gray] 
										(0,0,0) -- (0,0,2) -- (-1.2,-1.2,-1.2)--cycle;
						\filldraw[darkestgray] 
									(0,0,0) -- (-1.2,-1.2,-1.2) -- (0,2,0) -- cycle;
           \filldraw[darkgray] (0,0,0) -- (2,0,0) -- (0,2,0) --
                    cycle;
					 \draw[->] (0,0,0) -- (-.9,-.9, 0) node[anchor = east]{$14$};
					 \draw[->] (0,0,0) -- (-.9, 0, -.9) node[anchor = north  west] {$24$};
           \filldraw[darkgray] (0,0,0) -- (-1.2,-1.2,-1.2) -- (2,0,0) -- cycle;
									\draw[->] (0,0,0) -- (0,-.9, -.9) node[anchor = north]{$34$};
           \draw[->] (0,0,0) -- (0,2.1,0) node[anchor = west] {$2$};
						\draw[->] (0,0,0) -- (0,1.15,1.15) node[anchor = south west] {$12$};
						\draw[->] (0,0,0) -- (1.15, 1.15, 0) node[anchor = north west] {$23$};
           \filldraw[lightgray] (0,0,0) -- (2,0,0) -- (0,0,2) --cycle;
					\draw[->] (0,0,0) -- (1.15, 0, 1.15) node[anchor = south] {$13$};
					\draw[->] (0,0,0) -- (2.1,0,0) node[below right =2pt] {$3$};
           \draw[->] (0,0,0) -- (0,0,2.1) node[anchor = south] {$1$};
           \draw[->] (0,0,0) -- (-1.25,-1.25,-1.25) node[anchor= north east] {$4$};
               \fill (0,0,0) circle[radius=1pt];
         \end{tikzpicture}
							\caption{
								The Bergman fan of the empty order filter
								and the Bergman fan for $ U_{3,4} $.
							}
							\label{figure:U34Picture}
			\end{figure}
		\end{example}

		\subsection{Bergman fans of building sets}
			Our second family of generalized Bergman fans come from 
			a choice of \textit{building set} on $ \mathcal{L}(M) $. 
			We begin with a quick review of building sets and their associated
			nested set complexes.

			For a subset $ \mathcal{G} $ of a lattice $ \mathcal{L} $
			and an element
			$ X \in \mathcal{L} $, let 
			$ \mathcal{G}_{\leq X} := \mathcal{G} \cap [\hat{0}, X]  $.
			The \emph{factors} $ f(\mathcal{G}\big|_{X}) $ of $X$ in $ \mathcal{G} $
			are the maximal elements of $ \mathcal{G}_{\leq X} $.

			\begin{definition}
				Let $ \mathcal{L} $ be a lattice. A subset $ \mathcal{G} 
				\subseteq \mathcal{L} \setminus \hat{0}$ is a \emph{building set} 
				on $ \mathcal{L} $ if for all 
				$ X \in \mathcal{L} $, the map
				\[ \prod_{Y \in f (\mathcal{G}\mid_{X})} 
						[\hat{0}, Y] \longrightarrow 
					[0,X], \;\;\;  
							(Z_{Y})_{Y} \longmapsto \bigvee_{Y} Z_{Y}\]
				is an isomorphism of lattices.
			\end{definition}

			Building sets define a simplicial complex of \textit{nested sets}.

			\begin{definition}
				Let $ \mathcal{G} $ be a building set on the lattice $ \mathcal{L} $. 
			  A subset $ \mathcal{S} \subseteq \mathcal{G} $ is \emph{nested} 
				if the 
				following condition holds: for any set $\left\{ S_{1}, \dots, S_{k}
				\right\} $ of pairwise incomparable elements of $\mathcal{S} $ 
				with $ k \geq 2 $,
				the join  $ \bigvee_{i=1}^{k} N $ is not in $ \mathcal{G} $.

			The collection of nested sets forms a 
			simplicial complex denoted $ \mathcal{N}(\mathcal{L}, \mathcal{G}) $. 
			\end{definition}

			We now specialize	to the case of the lattice of flats of a matroid. 
			The \textit{Bergman fan of a building set} $ \mathcal{G} $ on
			$\mathcal{L}(M) $ is defined as a geometric realization of the cone over
			$ \mathcal{N}(\mathcal{L}(M), \mathcal{G}) $.

			\begin{definition}
				Let $ M $ be a loopless matroid on the ground set $[n]$, and
				let $ \mathcal{G} $ be a building set on $ \mathcal{L}(M) $ containing
				$[n]$. The \emph{non-reduced Bergman fan} of $ \mathcal{G} $ is the
				rational fan 
				\[ 
					\widehat{\Sigma}_{M, \mathcal{G}}  = \left\{ \pos(e_{F_{1}}, \dots, 
							e_{F_{k}}) :  \left\{F_{1}, \dots, F_{k} \right\} \in
							\mathcal{N}(\mathcal{L}(M), \mathcal{G}) \right\}.
				\]
				These fans are unimodular by \cite{FY}*{Proposition 2}.
			\end{definition}

			The fan $ \widehat{\Sigma}_{M, \mathcal{G}} $ again has lineality space
			$ \mathbb{Z}\cdot e_{[n]} $, as $ [n] \in \mathcal{G} $.

			\begin{definition}
				Let $M$ be a loopless matroid on the ground set $[n] $ and
				$ \mathcal{G} $ a building set on $ \mathcal{L}(M) $ 
				containing $[n]$. The \emph{Bergman fan} of $ \mathcal{G} $,
				written $ \Sigma_{M, \mathcal{G}} $, is
				the quotient fan of $ \widehat{\Sigma}_{M, \mathcal{G}} $ in
							$ \mathbb{Z}^{n} / \mathbb{Z} \cdot e_{[n]}$.
			\end{definition}

			As with order filters, we give special attention to 
			the Bergman fans associated to the maximal and minimal building sets.

			\begin{example}
				For any lattice $ \mathcal{L} $, the maximal building set is
				$ \mathcal{G}_{\max} = \mathcal{L} \setminus \hat{0} $. 
				Here, the nested sets are precisely chains in $ \mathcal{L} \setminus
				\hat{0} $. For matroids, this means that $ \Sigma_{M, \mathcal{G}_{\max}} $
				is the Bergman fan $ \Sigma_{M} $.
			\end{example}

			\begin{example}
				For any lattice $ \mathcal{L} $, there is a minimal building set
				$ \mathcal{G}_{\min} $ which contains $[n]$. 
				It consists of $[n]$ along with all \textit{irreducible}
				elements of $ \mathcal{L} \setminus \hat{0} $.
				These are the elements
				$X$ such that $[\hat{0}, X] $ cannot be written as the product of
				intervals of the form $ [\hat{0},Y] $ with $ Y < X $.
			 \end{example}

			 Our running examples of $ M^{\mathrm{br}} $ and
			 $ U_{3,4} $ show that order filters and building sets
			 define distinct classes of generalized Bergman fans.

			 \begin{example}
				For the broom matroid $M^{\mathrm{br}}$ of Figure \ref{fig:nonPureExample}, 
				the minimal building
				set on $ \mathcal{L}(M^{\mathrm{br}}) $ 
				is $ \left\{1,2, 3, 4, 234, 1234 \right\} $, which is not an order
				filter. On the other hand, $ \Sigma_{M^{\mathrm{br}}, \emptyset} $ is not 
				the Bergman fan of any building set on $ \mathcal{L}(M^{\mathrm{br}}) $,
				as it does not have a ray corresponding to the irreducible element $ 234 $.

				For the uniform matroid $U_{3,4} $, the minimal building set on
				$ \mathcal{L}(U_{3,4}) $ consists of the ground set along with
				all singletons. This shows that $ \Sigma_{U_{3,4}, \emptyset} 
				= \Sigma_{U_{3,4}, \mathcal{G}_{\min}} $. In fact,
				$ \Sigma_{M, \emptyset} = \Sigma_{M, \mathcal{G}_{\min}} $ if and only
				if $ M$ is a uniform matroid.
			 \end{example}

	\subsection{Toric varieties}
		See \cite{CLS} for a general
		reference in toric geometry.

		Let $N$ be a finite-dimensional integer lattice, and define the rational
		vector space
		$N_{\mathbb{Q}} := N \otimes_{\mathbb{Z}} 
		\mathbb{Q} $. We denote the dual of $N$ by
		$ N^{\vee} $ and define
		$ N^{\vee}_{\mathbb{Q}} := N^{\vee} \otimes_{\mathbb{Z}} \mathbb{Q} $.
		We write $ \langle -,- \rangle $ for the pairing
		$ N_{\mathbb{Q}}^{\vee} \otimes N_{\mathbb{Q}} \to \mathbb{Q} $.

		Each rational fan $ \Sigma \subseteq N_{\mathbb{Q}} $ defines
		a toric variety (over $ \mathbb{C} $) denoted by 
		$ X_{\Sigma} $. 
		We assume that all toric varieties are smooth, and equivalently, 
		all rational fans are unimodular.
		The lattices $N$ and $ N^{\vee} $ are the
		\emph{cocharacter} and \emph{character lattices} of $ X_{\Sigma}$, respectively.
		Let $ \Sigma(1) $ be the set of rays of $ \Sigma $. 
		For $ \rho \in \Sigma(1) $, we take $ u_{\rho} \in N $ to be the minimal
		generator of $ \rho \cap N $.
		Given a subset of cones
		$ \left\{ \sigma_{1}, \dots, \sigma_{k} \right\} \subseteq \Sigma  $,
		we write $ \pos(\sigma_{1}, \dots, \sigma_{k}) $ for the 
		(abstract) cone spanned by the $ \sigma_{i} $---it may or may not be a cone
		of $ \Sigma $.

		Each cone $ \tau \in \Sigma $ defines distinguished, torus-invariant open
		and closed subvarieties in $ X_{\Sigma} $.
		These are determined by the \emph{closed star} and \emph{star} of $ \tau $,
		respectively.

		\begin{definition}\label{def:star}
			The \emph{closed star} of $ \tau $ is the subfan
				\[ 
					\overline{\Star}(\tau) := \left\{ \sigma \in \Sigma :
						\pos(\tau, \sigma) \in \Sigma \right\}.
				\]
			The toric variety $ X_{\overline{\Star}(\tau)} $ is an open subset
			of $ X_{\Sigma} $. In general, it strictly contains the affine open
			$ U_{\tau} $. 

			The \emph{star} $ \Star(\tau) $ of $ \tau $	is the image of $ \overline{\Star}(\tau)
			$ in the quotient lattice $ N / \spann(\tau) $. It is 
			also a unimodular fan. The toric variety $ X_{\Star(\tau)} = V(\tau) $
			is a closed subvariety of $ X_{\Sigma} $. When $ \tau $ is a ray,
			$ X_{\Star(\tau)} $ is the torus-invariant divisor associated to $ \tau $.
		\end{definition}

		\subsection{Singular cohomology of smooth toric varieties}
		Throughout this paper, we will work with singular cohomology with
		rational coefficients, so $ H^{\bullet}(X_{\Sigma}) =
		H^{\bullet}(X_{\Sigma}, \mathbb{Q}) $.		

		Like many invariants of smooth toric varieties, the singular cohomology ring
		can be read off of the fan $ \Sigma $.
		As a word of caution,
		the singular cohomology ring  is determined not only by
		the combinatorics of $ \Sigma $---its structure as an abstract simplicial fan---but 
		also by its embedding
		in $ N_{\mathbb{Q}} $. Here, we recall a presentation for
		$H^{\bullet}(X_{\Sigma}) $
		in terms of the Stanley--Reisner ring of $ \Sigma $. 

		\begin{definition}
			For a simplicial fan $ \Sigma $, 
			the \textbf{Stanley--Reisner ring}
			is
				\[ 
					 \mathbb{Q}[\Sigma] := \frac{\mathbb{Q}[x_{\rho} :
							\rho \in \Sigma(1)]}{
											( x_{\rho_{1}} \cdots x_{\rho_{s}}  
											: \textrm{ $ \pos(\rho_{1}, \dots,  \rho_{s}) 
											\notin \Sigma $}).}
				\]
			The ideal $ (x_{\rho_{1}} \cdots x_{\rho_{s}} : \pos(\rho_{1}, \dots, \rho_{s}) 
			\notin \Sigma ) $ is called the \emph{Stanley--Reisner ideal}.
		\end{definition}

		When $ \Sigma \subseteq N_{\mathbb{Q}} $ is moreover a rational fan, the ring
		$ \mathbb{Q}[\Sigma] $ naturally has the structure of a
		$ \Sym(N_{\mathbb{Q}}^{\vee}) $-algebra, induced by
		the map
		\[ N_{\mathbb{Q}}^{\vee} \longrightarrow  \mathbb{Q}[\Sigma], \;\;\;  
				m \longmapsto \sum_{\rho \in \Sigma(1)}
						\langle m, u_{\rho} \rangle \cdot x_{\rho}.\]

		\begin{thm}[{\cite{franzRing}*{Theorem 1.2}}] \label{thm:franzRing}
			For a smooth toric variety $ X_{\Sigma} $, there is an
			isomorphism of rings
			\[
				 H^{\bullet}(X_{\Sigma}) \cong 
							\Tor_{\bullet}^{\Sym(N_{\mathbb{Q}}^{\vee})}
							(\mathbb{Q}[\Sigma], \mathbb{Q})_{\bullet},
			\]
			where $ \mathbb{Q} $ is the residue field of the maximal homogeneous ideal.
		\end{thm}

		The cohomology of $ X_{\Sigma} $ has additional structure coming  
		from the mixed Hodge structure \cite{deligne}. Specifically,
		there is an increasing \emph{weight filtration}
		\[ 
			0 = W_{i-1} H^{i}(X_{\Sigma}) \subseteq
			\cdots \subseteq W_{2i-1} H^{i}(X_{\Sigma}) 
			\subseteq W_{2i} H^{i}(X_{\Sigma}) = 
			H^{i}(X_{\Sigma}).
		\]
		Totaro \cite{totaro}*{Theorem 5} noted that this filtration splits in the case of
		toric varieties, and Weber \cite{weberWeights}*{Theorem 5.4} 
		showed that the splitting
		matches the bigrading of the Tor algebra:
			\[ 
					\Gr^{W}_{2j} H^{2j-i}(X_{\Sigma})
					\cong \Tor_{i}^{\Sym(N_{\mathbb{Q}}^{\vee})}(\mathbb{Q}[\Sigma],
					\mathbb{Q})_{j}.%
			\]
		Each $ \dim \Gr_{j}^{W} H^{k}(X_{\Sigma}) $ is called a \emph{Hodge number}.
	
		\begin{definition}\label{def:hpPolynomial}
		The \emph{Hodge--Poincar\'{e} polynomial} of $ X_{\Sigma} $
		is
		\[ 
			 \HPoin\left( X_{\Sigma} \right) := 
			 \sum_{i,j} \dim \Gr^{W}_{2j} H^{2j-i}(X_{\Sigma}, \mathbb{Q}) \cdot w^{i}x^{j}
			 = \sum_{i,j} \dim \Tor_{i}^{\Sym(N_{\mathbb{Q}}^{\vee})}
			 (\mathbb{Q}[\Sigma], \mathbb{Q})_{j} \cdot w^{i}x^{j}.
		\]
		\end{definition}

		\begin{rmk}
			Our chosen indexing is so that if we specialize $\HPoin\left( X_{\Sigma} \right) $ to $ w= 0 $, 
			we recover the Hilbert series of the Chow ring
			of $ X_{\Sigma} $ (see Example \ref{ex:chow} for a proof).
		\end{rmk}

		 Theorem \ref{thm:franzRing} allows us to use tools from commutative
		algebra to compute the cohomology of smooth toric varieties. For example,
		computer algebra software like Macaulay2 \cite{M2} can compute the Hodge 
		numbers of $ X_{\Sigma} $ by calculating the graded Betti numbers of 
		$ \mathbb{Q}[\Sigma] $ over $ \Sym(N_{\mathbb{Q}}^{\vee}) $.
		Another consequence of the theorem is that we can express the Hodge--Poincar\'{e} polynomial 
		of a toric subvariety of affine space  via Hochster's formula 
		as follows.

		Given an abstract simplicial complex $ \Delta $ on vertices 
		$ [n] $, there is a \emph{standard geometric realization}
		of $ \Delta $ in $ \mathbb{Q}^{n} $. It sends the vertex $ i $ to 
		$ e_{i} $. The cone over $ \Delta $ (based at 
		the origin) is a unimodular fan $ \Sigma $, and  the variety
		$ X_{\Sigma} $ is a toric subvariety of affine space. Any
		toric subvariety of affine space arises in this way.
		The Hodge numbers of $ X_{\Sigma} $ are determined by the reduced simplicial
		cohomology of $ \Delta $ and all its restrictions.

		\begin{lem}\label{lem:toricAffine}
			Suppose that $ \Sigma $ is the cone over the standard geometric realization
			of a simplicial complex $ \Delta $ on vertices $[n]$. Then,
			\[ 
					\HPoin\left(X_{\Sigma}\right) = \sum_{i,j} 
					\sum_{\substack{W \subseteq [n] \\ \left| W \right| = j}}
					\dim \widetilde{H}^{j-i-1}(\Delta\big|_{W}) \cdot  w^{i}x^{j}.
			\]
			Here, $ \Delta \big|_{W} $ denotes the restriction of $ \Delta $ to the
			vertices in $W$.
		\end{lem}

		\begin{proof}
			Let $ N^{\vee} $ be the character lattice of $ X_{\Sigma} $. Then,
			$ \Sym(N_{\mathbb{Q}}^{\vee}) \cong \mathbb{Q}[x_{1}, \dots, x_{n}] $,
			and the structure of $ \mathbb{Q}[\Sigma]\cong \mathbb{Q}[\Delta] $ as a 
			$ \Sym(N_{\mathbb{Q}}^{\vee}) $-algebra is given by the map 
			$ x_{i} \mapsto x_{i} $. Therefore, Hochster's formula applies
			(\cite{millerSturmfels}*{Theorem 5.12}), and
			\[ 
					\dim \Tor_{i}^{\Sym(N_{\mathbb{Q}}^{\vee})} (\mathbb{Q}[\Sigma],
						\mathbb{Q})_{j} = \dim \sum_{\substack{W \subseteq [n] \\ 
						\left| W \right| = j}} \widetilde{H}^{j-i-1}(\Delta \big|_{W}).  \qedhere
			\]
		\end{proof}

	\subsection{Koszul homology}\label{subsection:koszul}
		We now give a more explicit and combinatorial 
		presentation of $H^{\bullet}(X_{\Sigma}) $ in terms
		of the Tor algebra and Koszul homology. See \cite{brunsHerzog} for
		a standard reference on Koszul homology.

		For a graded ring $ R $, let $ R_{i} $ be the module of homogeneous elements
		of degree $i$.

		\begin{definition}
			Let $ \Sigma \subseteq N_{\mathbb{Q}} $ be a unimodular fan.
			The \textbf{Koszul complex}	
			$ K_{\bullet}(\mathbb{Q}[\Sigma]) $ is the bigraded differential algebra
			\[ 
						K_{i}(\mathbb{Q}[\Sigma])_{j} := \mathbb{Q}[\Sigma]_{j-i}
						\otimes_{\mathbb{Q}} \bigwedge^{i} N_{\mathbb{Q}}^{\vee}
			\]
			with differentials $d_{i}\colon K_{i}(\mathbb{Q}[\Sigma])_{j}\to  
			K_{i-1}(\mathbb{Q}[\Sigma])_{j} $   defined by
			\[ f \otimes m_{1} \wedge \cdots \wedge m_{i}
							\longmapsto \sum_{k=1}^{i} (-1)^{k+1} (m_{k}\cdot f) \otimes 
							m_{1} \wedge \cdots \wedge \widehat{m}_{k} \wedge 
							\cdots \wedge m_{i}.\]
			The multiplication map is given by
			\[ 
						(f \otimes \zeta) \cdot  (g \otimes \xi) \mapsto f \cdot g \otimes
						\zeta \wedge \xi,
			\]
			which is graded commutative in the index $ i $. 
		\end{definition}		

		\begin{rmk}
			The standard construction of the Koszul complex starts with
			a choice of regular sequence $ \mathbf{f} = (x_{1}, \dots, 
			x_{s}) $ of $ \Sym(N_{\mathbb{Q}}^{\vee}) $ to produce the Koszul 
			complex  $ K_{\bullet}(\mathbf{f}, \mathbb{Q}[\Sigma]) $. 
			Our construction makes no such choice, but
			it is isomorphic to the Koszul complex associated
			to  any maximal regular sequence of elements of 
			$ \Sym^{1}(N_{\mathbb{Q}}^{\vee}) $.
		\end{rmk} 

			Taking homology of the Koszul complex yields the 
			\textbf{Koszul homology} groups
			\[ 
				    H_{i}(K(\mathbb{Q}[\Sigma]))_{j} \cong
						\Tor_{i}^{\Sym(N_{\mathbb{Q}}^{\vee})}
						(\mathbb{Q}[\Sigma], \mathbb{Q})_{j} \cong
						\Gr_{2j} H^{2j-i}(X_{\Sigma}).
			\]
		This induces an isomorphisms of rings $ H_{\bullet}(K(\mathbb{Q}[\Sigma]))
		\cong H^{\bullet}(X_{\Sigma}) $, where the multiplication
		in $ H_{\bullet}(K(\mathbb{Q}[\Sigma])) $ is induced by the
		multiplication in the Koszul complex. Moreover, this isomorphism 
		is natural with respect to toric morphisms which we now
		describe.

		Suppose that $ \Sigma \subseteq N_{\mathbb{Q}} $ and $ \Sigma' \subseteq
		N'_{\mathbb{Q}} $ are rational unimodular fans, and suppose that 
		$ \phi\colon N \to N' $ is a map of integral lattices such that
		for every $ \sigma \in \Sigma $ there exists $ \sigma' \in \Sigma $ such 
		that $ \phi(\sigma) \subseteq \sigma' $. Then, $ \phi $ defines a 
		\emph{toric morphism} $ \phi\colon X_{\Sigma} \to X_{\Sigma'} $.

		There are two ways to get a map on cohomology via $ \phi $. 
		The first is the pullback map $ \phi^{*}\colon H^{\bullet}(X_{\Sigma'}) 
		\to H^{\bullet}(X_{\Sigma}) $. The second is by Koszul complexes.
		We can view $ \mathbb{Q}[\Sigma] $ as the algebra of piecewise polynomial
		functions on the support of $ \Sigma $
		(\cite{brion}*{\S 1.3}), and pulling back these functions along
		$ \phi\colon N_{\mathbb{Q}} \to N'_{\mathbb{Q}} $ induces a ring map
		$ \mathbb{Q}[\Sigma'] \to \mathbb{Q}[\Sigma] $. Dualizing $ \phi $ induces
		a map $ \phi^{*}\colon {N'}_{\mathbb{Q}}^{\vee} \to N_{\mathbb{Q}}^{\vee} $. 
		One verifies that the map $ \phi\otimes \phi^{*}\colon
		K_{\bullet}(\mathbb{Q}[\Sigma']) \to K_{\bullet}(\mathbb{Q}[\Sigma]) $ 
		respects differentials and we produces a map
		$ \phi \otimes \phi^{*}\colon H_{\bullet}(K(\mathbb{Q}[\Sigma']))
		\to H_{\bullet}(K(\mathbb{Q}[\Sigma])) $. 
		The two maps on cohomology yield a diagram
		\begin{center}
			\begin{tikzcd}
				H^{\bullet}(X_{\Sigma'}) \rar["\phi^{*}"] & H^{\bullet}(X_{\Sigma}) \\
				H_{\bullet}(K(\mathbb{Q}[\Sigma'])) \uar["\cong"] \rar["\phi \otimes
							\phi^{*}"] & H_{\bullet}(K(\mathbb{Q}[\Sigma])),  \uar["\cong"]
			\end{tikzcd}
		\end{center}
		and the following result of Franz--Fu shows that this square commutes.

		\begin{thm}[\cite{franzFu}*{Theorem 1.2}]
						\label{thm:franzFu}
						The isomorphism $ H_{\bullet}(K(\mathbb{Q}[\Sigma])) \cong
						H^{\bullet}(X_{\Sigma}) $ is natural with respect to 
						toric morphisms.
		\end{thm}

		\begin{example}\label{ex:chow}
			We use Koszul homology to describe the isomorphism (\cite{totaro}*{Theorem 6}) between 
			the Chow cohomology
			of $ X_{\Sigma} $ and the cohomology of Hodge--Tate type 
			$ A^{k}(X_{\Sigma}) \cong \Gr_{2k}^{W} H^{2k}(X_{\Sigma}) $.

			Note that $ \Gr_{2k}^{W} H^{2k}(X_{\Sigma}) \cong
			H_{0}(\mathbb{Q}[\Sigma])_{k} $. Every element of $ \mathbb{Q}[\Sigma] $ 
			is a cycle in the Koszul complex, and the boundaries are the ideal
			\[ J = 
				\left( \sum_{\rho \in \Sigma(1)} \langle m, u_{\rho} \rangle \cdot
						x_{\rho} : m \in N^{\vee}_{\mathbb{Q}} \right). \]
			Thus, $ \Gr_{2k}^{W}H^{2k}(X_{\Sigma}) \cong 
				H_{0}(\mathbb{Q}[\Sigma])_{k} \cong \mathbb{Q}[\Sigma]_{k} / J $.
			The term on the right is the well-known presentation for
			$ A^{k}(X_{\Sigma}) $ as in  \eqref{eq:FYPresentation}.

			In particular, the singular cohomology ring of a smooth toric variety contains
			the Chow ring, and the specialization of $ \HPoin(X_{\Sigma}) $ at $ w = 0 $
			is the Hilbert series of $A^{\bullet}(X_{\Sigma})$.
		\end{example}

		\section{The singular cohomology ring of a matroid}\label{section:singCohomology}

		In this section, we define the \textit{singular cohomology ring} of a matroid and
		compute some examples to illustrate Theorem \ref{thm:vanishingCohomology}.

		\begin{definition}
			For a loopless matroid $M$, the \emph{singular cohomology ring} of
			$M$ is the ring $ H^{\bullet}(X_{\Sigma_{M}}) $.
		\end{definition}

		By Example \ref{ex:chow}, the singular cohomology ring contains the Chow ring of the matroid,
		with $ A^{k}(M) \cong \Gr_{2k}^{W} H^{2k}(\Sigma_{M}) $ under the weight filtration.
		In fact, it is a consequence of our main result, Theorem \ref{thm:vanishingCohomology}, 
		that this containment is strict if and only if $  M$ is not a boolean matroid.

		\begin{example} \label{ex:exampleHP}
			We use Macaulay2  and the Matroids package
			\cite{matroidsM2} to find the Hodge numbers of the singular cohomology rings of the two
			examples matroids from Section \ref{section:background}.
			\[ 
				\HPoin\left(X_{\Sigma_{M^{\mathrm{br}}}}\right)
						= 1 + 5x + x^{2} + 5wx^{2} + 2 wx^{3}; \;\;\;
				\HPoin\left(X_{\Sigma_{U_{3,4}}}\right) = 
					1 + 7x +  x^{2} + 6wx^{2} + 3 wx^{3}.
			\]

			These Hodge--Poincar\'{e} polynomials illustrate Theorem
			\ref{thm:vanishingCohomology}. The highest-degree term of 
			$ \HPoin(X_{\Sigma_{M^{\mathrm{br}}}}) $ is $ 2 wx^{3} $. The exponent says 
			that the top-weight cohomology is $ \Gr_{6}^{W}H^{5}(X_{\Sigma_{M^{\mathrm{br}}}}) $,
			while the coefficient says the dimension is $ 2 $. This is the M\"{o}bius invariant
			of $ M^{\mathrm{br}} $, as $ \NBC(M^{\mathrm{br}}) = \left\{ 123, 124 \right\} $.
			Similarly, the top-weight cohomology for $ U_{3,4} $ is
			$\Gr_{6}^{W}H^{5}(X_{\Sigma_{U_{3,4}}}) $ with dimension $ 3 $. Note here that
			$ \NBC(U_{3,4}) = \left\{ 123, 124, 134 \right\} $.
		\end{example}

			We can also use Koszul homology to compute the singular cohomology ring of a matroid
			by hand.  To simplify notation, we will write $ x_{F} $ instead of 
			$ x_{\rho_{F}} \in \mathbb{Q}[\Sigma_{M}] $. We take generators
			$ \left\{ \ell_{i} - \ell_{j} := e^{*}_{i} - e^{*}_{j} : 
			i, j \in [n] \right\} $ for $ N_{\mathbb{Q}}^{\vee} $, so that the Koszul differential is
			described by
			\[ 
				d (\ell_{i} - \ell_{j}) = \sum_{i \in F} x_{F} - \sum_{j \in G} x_{G}.
			\]
			
		\begin{example}\label{ex:U23}
			We compute a basis for the singular cohomology of 
			$U_{2,3} $ (see Figure \ref{fig:U23} for a picture of its Bergman fan).
			Instead of using the infinite dimensional Koszul complex, we will pass to the quasi-isomorphic
			\textit{square-free Koszul complex} whose elements and differentials are given in
			Table \ref{table:koszul} (see Definition \ref{def:squareFreeKoszul} and 
			Lemma \ref{lem:squareFreeKoszul}).

			\begin{figure}
				\begin{tikzpicture}
								\foreach \x in {-2,-1,0,1,2} \foreach \y in {-2,-1,0,1,2}
								{\filldraw[color = gray] (\x,\y) circle (.3pt);}
								\draw[->] (0,0) -- (2,0) node[anchor = west]{$\rho_{1}$};
								\draw[->] (0,0) -- (0,2) node[anchor = south] {$\rho_{2}$};
								\draw[->] (0,0) -- (-1.4, -1.4) node[anchor = north east]
								{$\rho_{3}$};
				\end{tikzpicture}
							\caption{The Bergman fan $ \Sigma_{U_{2,3}} $.}
				\label{fig:U23}
			\end{figure}			

			\begin{table}[b]
			\begin{tabular}{c|c|c|c|c}
							$\alpha $ & $ 1$ & $ (\ell_{1} - \ell_{2}) $ & $ (\ell_{2} - \ell_{3}) $ 
							& $ (\ell_{1} - \ell_{2}) \wedge (\ell_{2} - \ell_{3}) $ \\ \hline
					$d \alpha $ & $ 0 $ & $ x_{1} - x_{2} $ 
							& $ x_{2} - x_{3} $ & $ x_{1} 
							\otimes (\ell_{2} - \ell_{3}) - x_{2} \otimes (\ell_{1} - \ell_{3}) 
							+ x_{3} \otimes (\ell_{1} - \ell_{2}) $
			\end{tabular}

			\vspace*{1em}

			\begin{tabular}{c|c|c|c|c|c|c}
							$ \alpha $ & $ x_{1} $ & $ x_{1} \otimes 
							(\ell_{2} - \ell_{3}) $ & $ x_{2} $
							& $ x_{2} \otimes (\ell_{1} - \ell_{3}) $ &
							$ x_{3} \otimes 1 $ & $ x_{3} \otimes (\ell_{1} - \ell_{2})  $ \\
				\hline
				$ d \alpha $ & $ 0 $ & $0 $ & $0$ & $ 0$ & $0$ & $0$ 
			\end{tabular}		
				\vspace*{1em}
							\caption{The square-free Koszul complex for $U_{2,3} $.} \label{table:koszul}
			\end{table}

		From the table, we see that $ H_{\bullet}(K(\mathbb{Q}[\Sigma_{U_{2,3}}])) $ has basis
			$ 1 $, $x_{3} $, $x_{2} \otimes (\ell_{1} - \ell_{3}) $,
			and $ x_{\rho_{3}} \otimes (\ell_{1} - \ell_{2}) $. Therefore,
			$ \HPoin(X_{\Sigma_{U_{2,3}}}) = 1 + x + 2wx^{2} $.
		\end{example}

		\begin{example}\label{ex:U34}
			A similar, but longer, computation shows that 
			$H_{\bullet}(K(\mathbb{Q}[\Sigma_{U_{3,4}}])) $ has basis
				\begin{gather*}	
						1, x_{4}, x_{12}, x_{13}, x_{14}, 
						x_{23}, x_{24}, x_{34},
						x_{4}x_{34}, \\
						x_{12} \otimes (\ell_{3} -\ell_{4}), 
						x_{13} \otimes (\ell_{2} - \ell_{4}), 
						x_{14} \otimes (\ell_{2} - \ell_{3}), \\
						x_{23} \otimes (\ell_{1} - \ell_{4}), 
						x_{24} \otimes (\ell_{1} - \ell_{3}),
						x_{34} \otimes (\ell_{1} -\ell_{2}), \\
						x_{3} x_{23} \otimes (\ell_{1} - \ell_{4}),
						x_{4} x_{24} \otimes (\ell_{1} - \ell_{3}),
						x_{4} x_{34} \otimes (\ell_{1} - \ell_{2}).
				\end{gather*} 

			As in Example \ref{ex:exampleHP}, we see that $ \HPoin(X_{\Sigma_{U_{3,4}}}) =
			1 + 7x + x^{2} + 6wx^{2} + 3wx^{3} $.
		\end{example}

		Finally, we show that	the singular cohomology ring is sensitive to parallel elements,
		in distinction to the Chow ring of a matroid.

		\begin{example}
		Define the matroid $M$ on the ground set $[5]$ such that $ M\big|_{[4]}=
		U_{2,4} $ and $ 1 $ and $ 5 $ are parallel elements. Then,
		$ U_{2,4} $ is the simplification of $ M$.
		A direct computation using Macaulay2 shows that the singular cohomology rings of
		these matroids are distinct:
		\[ 
			\HPoin\left( X_{\Sigma_{M}} \right) = 
			1  + wx + x + wx^{2} +  wx^{2} \left(2(1 +wx) + 3 (1+wx)^{2} \right), \]
		while
		\[ 
			\HPoin\left( X_{\Sigma_{U_{2,4}}} \right) =
			1 + x + wx^{2} \left(2 + 3(1+wx)\right).
		\] 
			\end{example}
		In this example, note that 
		$ \HPoin( X_{\Sigma_{M}} ) =  (1+ wx) \cdot \HPoin ( X_{\Sigma_{U_{2,4}}}) $.
		This illustrates a more general phenomenon. 

		\begin{prop}\label{prop:parallelEdges}
		Let $M$ be a loopless 
		matroid and $ M^{s} $ its simplification. Let $ p $ be the difference
		in the size of the ground sets of $M$ and $ M^{s} $. Then,
		\[ 
				\HPoin\left(X_{\Sigma_{M}}\right) = (1 + wx)^{p}\cdot \HPoin\left(X_{\Sigma_{M^{s}}}\right).
		\]
		\end{prop}

		\begin{proof}
			The fan  $ \widehat{\Sigma}_{M} \subseteq N_{\mathbb{Q}} $ 
			sits in the subspace spanned by $ e_{i} - e_{j} $ for parallel elements $ i,j $,
			which has dimension $p$.
			Quotienting out by this subspace yields a rational fan isomorphic 
						to $ \widehat{\Sigma}_{M^{s}} $.
			Therefore, $ X_{\widehat{\Sigma}_{M}} \cong X_{\widehat{\Sigma}_{M^{s}}} \times
					(\mathbb{C}^{*})^{p} $, and
			$ X_{\Sigma_{M}} \cong X_{\Sigma_{M^{s}}} \times (\mathbb{C}^{*})^{p} $.
			The result follows from the K\"{u}nneth formula, as
		  $ \HPoin((\mathbb{C}^{*})^{p}) = (1 + wx)^{p} $, 		
		\end{proof}

		The rest of this paper is devoted to understanding the singular cohomology ring of a 
		matroid in terms of its combinatorics. 
		Before working with the variety $ X_{\Sigma_{M}} $ directly, we introduce two simpler varieties,
		$ X_{\widehat{\Sigma}_{M, \emptyset} \setminus \rho_{[n]}} $ and $ X_{\Sigma_{M, \emptyset}} $,
		whose cohomology we compute in terms of commutative algebra. 

		\section{\texorpdfstring{$\HPoin\left(X_{\widehat{\Sigma}_{M, \emptyset}
		\setminus \rho_{[n]}}
		\right)$}{The Hodge--Poincar\'{e} polynomial of the lifted Bergman fan of the empty order filter}
	via Hochster's formula}\label{section:liftEmpty}

		In this section, we compute
		the Hodge--Poincar\'{e} polynomial of $ X_{\widehat{\Sigma}_{M, \emptyset}
		\setminus \rho_{[n]} } $ via Lemma \ref{lem:toricAffine}, which is Hochster's formula 
		applied to the cohomology of toric subvarieties of affine space.
		
		Let $ \widehat{\Sigma}_{M, \emptyset} \setminus \rho_{[n]} $ be the fan
		obtained by removing the ray $ \rho_{[n]} $ from $ \widehat{\Sigma}_{M, \emptyset} $.
		From Example \ref{ex:nsComplex}, the non-spanning complex $ \Delta^{ns}_{M} $ is the simplicial 
		complex on vertex set $[n]$ whose faces are subsets $ S \subseteq [n] $ such that 
		$ \closure(S) \neq [n] $.
		The fan $ \widehat{\Sigma}_{M, \emptyset} \setminus \rho_{[n]} $ is
		the cone over the standard geometric realization of $ \Delta^{ns}_{M} $,
		so $ X_{\widehat{\Sigma}_{M, \emptyset} \setminus \rho_{[n]}} $ is a toric 
		subvariety of affine space.
		Thus, Lemma \ref{lem:toricAffine} applies, and computing the Hodge numbers	
		of $ X_{\widehat{\Sigma}_{M, \emptyset}
		\setminus \rho_{[n]}} $ is equivalent to computing the reduced simplicial
		cohomology of all the restrictions of $ \Delta_{M}^{ns} $. 

		We begin this computation by finding the
		reduced Betti numbers of $ \Delta_{M}^{ns} $ itself.
		 
		\begin{lem}\label{lem:cohomologyNSComplex}
					Let $ M$ be a matroid of rank $ r $ on the ground set $[n]$. Then,
			\[ 
					\dim \widetilde{H}^{i}(\Delta_{M}^{ns})
					= \begin{cases}
							\left| \NBC(M) \right| & i = r-2 \\
							0 & i \neq r-2.
						\end{cases} 			\]
	\end{lem}

	\begin{proof}
		The Alexander dual of $ \Delta_{M}^{ns} $ is 
		the independence complex $ \IN(M^{*}) $
		of the dual matroid \cite{bettiNumbersFacetIdeal}*{Proposition 1}.
		Bj\"{o}rner
		\cite{bjorner}*{Theorem 7.8.1} computes the Betti numbers of this complex as		\[ 
			 \dim \widetilde{H}_{i}(\IN(M^{*})) =
				\begin{cases}
						\left| \NBC(M) \right| & i = n-r-1 \\
						0 & i \neq n-r-1.
				\end{cases}
		\]
		By  combinatorial Alexander duality \cite{alexanderDuality}*{Theorem 1.1},
		$	\widetilde{H}_{i}(\IN(M^{*})) = \widetilde{H}^{n-i-3}(\Delta_{M}^{ns}) $, and
		the statement follows.
	\end{proof}

	The reduced cohomology for restrictions of $ \Delta_{M}^{ns} $ is similar.

	\begin{lem}\label{lem:restrictionsNS}
		Let $M$ be a matroid of rank $ r $ on the ground set $[n]$, and
		let $ W \subseteq [n] $. 
					\begin{enumerate}
						\item If $ W  = \emptyset $, then
									\[ 
											\dim \widetilde{H}^{s}\left(\Delta_{M}^{ns}\big|_{W}
													\right) = \begin{cases} 1 & s = -1 \\ 0 & s \neq -1.
													\end{cases}
									\]
						\item If $W $ is non-empty and $ \rk(W) < r $, then 
									\[ 
													\widetilde{H}^{\bullet}\left(\Delta_{M}^{ns}\big|_{W}
													\right) = 0.
													\]

						\item If $W$ is non-empty and $ \rk(W) = r $, then
									\[ 
										\dim \widetilde{H}^{s}\left(\Delta_{M}^{ns}\big|_{W}
													\right) = \begin{cases}
														\left|\NBC\left(M\big|_{W}\right) \right| & s = r-2 \\
														0 & s \neq r-2.
													\end{cases}
									\]
							\end{enumerate}
	\end{lem}

	\begin{proof}The proof comes down to identifying each restriction
					$ \Delta_{M}^{ns}\big|_{W} $.
		\begin{enumerate}
			\item If $W = \emptyset $, then $ \Delta_{M}^{ns}\big|_{W} $
						is the empty simplicial complex, and its reduced cohomology is 
						as described (see \cite{millerSturmfels}*{Remark 1.19}).

		\item If $ W $ is non-empty but $ \rk(W) < r $, then
					$ \Delta_{M}^{ns}\big|_{W} $ is a non-empty simplex and the
					reduced cohomology vanishes.

	\item If $W $ is non-empty and $ \rk(W) = r $, then
					$ \left(\Delta_{M}^{ns}\right) \big|_{W} = \Delta_{(M|_{W})}^{ns} $, 
				and the computation follows from Lemma \ref{lem:cohomologyNSComplex}
				applied to $ M \big|_{W} $. \qedhere
		\end{enumerate}
	\end{proof}

	We arrive at the main result of this section.
	\begin{thm}\label{thm:HPoinLiftEmpty}
		Let $M$ be a matroid of rank $ r $ on the ground set $[n]$. Then,
		\begin{align*}
					\HPoin\left(X_{\widehat{\Sigma}_{M, \emptyset}
					\setminus \rho_{[n]}} \right)
					=  1  + wx^{r} \sum_{B \in \mathcal{B}(M)} (1 + wx)^{\ep(B)}
					= 1 + (1 +wx)^{n-r} wx^{r} T_{M}\left( 1, \, \frac{1}{1+wx} \right).
		\end{align*}
		\end{thm}

	\begin{proof}
					Directly applying Lemma \ref{lem:restrictionsNS} to Lemma \ref{lem:toricAffine},
					we obtain
					\[ 
						\HPoin\left( X_{\widehat{\Sigma}_{M, \emptyset}
						\setminus \rho_{[n]}} \right) =
						1 + \sum_{\substack{ W \subseteq [n] \\ \rk(W) = r}}
						\left| \NBC\left(M\big|_{W} \right) \right| \cdot w^{\left| W \right|
						-r + 1} x^{\left| W \right|}.
					\]
					By the following lemma, the contribution of a basis $ B \in \mathcal{B}(M) $ to this
					sum is $ wx^{r} (1 + wx)^{\ep(B)} $, and
					indexing the summation over $ \mathcal{B}(M) $  produces the first equality in the statement.
					As $ n-r-\ea(B) = \ep(B) $, the definition of the Tutte polynomial produces
					the second equality.
	\end{proof}

	\begin{lem}\label{lem:epBasis}
			Let $M$ be a loopless matroid of rank $ r $ on the ground set $[n]$, 
			and let $ B \in \mathcal{B}(M) $. Then,
			\[
							\left| \left\{ W \subseteq [n] \vphantom{\frac{}{}} : 
						\left| W \right| = r + i \textrm{ and }
						B \in \NBC\left(M\big|_{W}\right) \right\} \right|
						= \binom{\ep(B)}{i}.
			\]
		\end{lem}

		\begin{proof}
			The set 
			$ \left\{ W \subseteq [n] : \left| W \right| 
							= r + i \textrm{ and } B \in \NBC\left( M\big|_{W}
							\right) \right\} $
			is in bijection with the collection of $i$-element subsets of $ \EP(B) $
			by the map $ X \mapsto X \setminus B $.
			Indeed, a subset $ W \subseteq[n] $ satisfies $ B \in \NBC\left( M
			\big|_{W} \right) $ if and only if $ B \subseteq W $ and 
			every $ x \in W \setminus B $ is in $ \EP(B) $. 
		\end{proof}
	
\begin{rmk}\label{r:codingTheory}
    Specializations of the Tutte polynomial similar to that in
		Thereom \ref{thm:HPoinLiftEmpty} have appeared in the coding theory literature, 
		and it would be interesting to see if there is any underlying
		relation to the geometry of these toric varieties.
  	For example, see \cite{weightPolynomialsMatroids}*{Corollary 5.1}.
 \end{rmk}

	From the support of the Hilbert--Poincar\'{e} polynomial we read off the following 
	vanishing results for the cohomology of
	$ X_{\widehat{\Sigma}_{M, \emptyset} \setminus \rho_{[n]}} $. 	

	\begin{cor}\label{cor:weightLiftEmpty}
		Let $M$ be a matroid of rank $ r $ on the ground set $[n]$. 
		Then,
		\begin{enumerate}
			\item For $ k = 0 $, $ H^{k}\left( X_{\widehat{\Sigma}_{M, \emptyset}
										\setminus \rho_{[n]}}\right) \cong 
									\Gr_{k}^{W} H^{k}\left(X_{\widehat{\Sigma}_{M, \emptyset}
										\setminus \rho_{[n]}}\right) $.
			\item For $ 0 < k < 2r-1 $,
							$H^{k}\left(X_{\widehat{\Sigma}_{M, \emptyset}
										\setminus \rho_{[n]}}\right) = 0 $.
			\item For $ k > 0 $, 
						$H^{k}\left(X_{\widehat{\Sigma}_{M,\emptyset}
										\setminus \rho_{[n]}}\right) \cong 
								\Gr^{W}_{2k-2r+2}H^{k}\left(X_{\widehat{\Sigma}_{M,\emptyset}
										\setminus \rho_{[n]}}\right) $.
		\end{enumerate}
	\end{cor}

	\section{\texorpdfstring{$\HPoin\left(X_{\Sigma_{M, \emptyset}}\right) $}
	{The Hodge--Poincar\'{e} polynomial of the Bergman fan of the empty order filter.}
	via flat base change}\label{section:empty}

	In this section, we compute $ \HPoin(X_{\Sigma_{M,\emptyset}}) $
	from $ \HPoin(X_{\widehat{\Sigma}_{M, \emptyset}\setminus \rho_{[n]}}) $. 
	We compare the cohomology of $ X_{\Sigma_{M, \emptyset} \setminus \rho_{[n]}} $
	and $ X_{\Sigma_{M, \emptyset}} $ through a long exact sequence
	\begin{equation}\label{eq:lesWeight}
		\cdots \to \Gr^{W}_{\bullet}H^{\bullet}\left( X_{\Sigma_{M, \emptyset}}
		\right) \to \Gr^{W}_{\bullet + 2} H^{\bullet + 2} \left( X_{\Sigma_{M,
		\emptyset}} \right)  \to
		\Gr^{W}_{\bullet + 2} H^{\bullet + 2} \left( X_{\widehat{\Sigma}_{M,
					\emptyset}\setminus \rho_{[n]}}\right) \to \cdots.
	\end{equation}
	Crucially, each group $H^{k}(X_{\widehat{\Sigma}_{M, \emptyset} \setminus
	\rho_{[n]}}) $ is supported in a single associated graded 
	of the weight filtration by Corollary
	\ref{cor:weightLiftEmpty}. We will show that this is true for
	$ H^{k}(X_{\Sigma_{M, \emptyset}}) $ as well, and this forces
	\eqref{eq:lesWeight} to frequently split into short exact sequences.
	This splitting enables the computation of $ \HPoin(X_{\Sigma_{M, \emptyset}})$.

	We begin by constructing \eqref{eq:lesWeight}, where we make use
	of flat base change for Tor.
	\begin{prop}[\cite{weibel}*{Proposition 3.2.9}]
		\label{prop:flatBaseChange}
		If $ R \to S $ is a flat ring map, $ A $ is an $  R $-module, and
		$ B $ is an $ S $-module, then there is a natural isomorphism
		\[ 
			 \Tor_{\bullet}^{R}(A,B) \cong \Tor_{\bullet}^{S}(A
					\otimes_{R} S, B).
		\]
	\end{prop}

	For our construction, let $N $ and $\widehat{N} $ be the cocharacter
	lattices of  $ X_{\Sigma_{M, \emptyset}} $
	and $ X_{\widehat{\Sigma}_{M, \emptyset} \setminus \rho_{[n]}} $, respectively.
	By definition, $ N $ is the quotient of $ \widehat{N} $ by $ \mathbb{Z} \cdot e_{[n]} $. 
	Then, $ N^{\vee} \subseteq 
	\widehat{N}^{\vee} $, and we can choose $ \ell \in \widehat{N}^{\vee} $ 
	such that 
	$ N_{M}^{\vee} \oplus \mathbb{Z} \cdot \ell = \widehat{N}_{M}$.
	Therefore, $ \Sym(N_{\mathbb{Q}}^{\vee})[\ell] =
	\Sym(\widehat{N}_{\mathbb{Q}}^{\vee}) $, and the inclusion
	$ \Sym(N_{\mathbb{Q}}^{\vee}) \hookrightarrow 
		\Sym(\widehat{N}_{\mathbb{Q}}^{\vee}) $ is flat.

	Note that 
	$ \mathbb{Q}[\Sigma_{M, \emptyset}] $ and $ \mathbb{Q}[\widehat{\Sigma}_{M,
	\emptyset} \setminus \rho_{[n]}] $ 
	are isomorphic as $ \Sym(N_{\mathbb{Q}}^{\vee}) $-algebras, and let $y$
	be the image of $\ell $ in $ \mathbb{Q}[\widehat{\Sigma}_{M, \emptyset}
	\setminus \rho_{[n]} ]
	\cong \mathbb{Q}[\Sigma_{M, \emptyset}] $ under the structure map.
	We have the short exact sequence of 
	$ \Sym(\widehat{N}^{\vee}_{\mathbb{Q}}) $-algebras
	\begin{multline*}
			0 \to \mathbb{Q}[\Sigma_{M, \emptyset}] \otimes_{\Sym(N^{\vee}_{\mathbb{Q}})} 
			\Sym(\widehat{N}_{\mathbb{Q}}^{\vee}) 
			\xrightarrow{ \cdot (y \otimes 1 - 1 \otimes \ell)}
					\mathbb{Q}[\Sigma_{M, \emptyset}] \otimes_{\Sym(N_{\mathbb{Q}}^{\vee})} 
					\Sym(\widehat{N}_{\mathbb{Q}}^{\vee})
					\\  \longrightarrow \mathbb{Q}[\Sigma_{M, \emptyset}] 
					\otimes_{\Sym(\widehat{N}_{\mathbb{Q}}^{\vee})} 
  			\Sym(\widehat{N}_{\mathbb{Q}}^{\vee})\to 0.
	\end{multline*}
	Applying
	$ \Tor_{\bullet}^{\Sym(\widehat{N}_{\mathbb{Q}}^{\vee})}
		(-, \mathbb{Q})_{\bullet} $,
	flat base change, and Theorem \ref{thm:franzRing}, this is the desired long exact 
	sequence \eqref{eq:lesWeight}.

	The next two lemmas use \eqref{eq:lesWeight} to show that  each cohomology group 
	$ H^{k}(X_{\Sigma_{M, \emptyset}}) $ sits inside a single weight space.

	\begin{lem}\label{lem:weightEmptyChow}
		Let $M$ be a loopless matroid of rank $ r $. 
		For $ 0 \leq k \leq 2r-2 $ and $k$ even,
		\[ 
				H^{k}\left( X_{\Sigma_{M, \emptyset}}\right)
					\cong 
					\Gr_{k}^{W}H^{k}\left( X_{\Sigma_{M, \emptyset}}  \right)
					\cong \mathbb{Q}.
		\]
		For $ 1 \leq k \leq 2r-3 $ and $k$ odd,
		$H^{k}\left( X_{\Sigma_{M, \emptyset}} \right) = 0 $.
	\end{lem}

	\begin{proof}
		We induct on $k$. For $ k=0 $, $ X_{\Sigma_{M, \emptyset}} $ is non-empty and connected,
		so $ H^{0}(X_{\Sigma_{M, \emptyset}}) \cong \mathbb{Q} $.
		For $ k = 1 $, \eqref{eq:lesWeight} reads
		\[ 
			\Gr^{W}_{j-2}H^{-1}\left(X_{\Sigma_{M, \emptyset}}\right)  \to 
			\Gr^{W}_{j}H^{1}\left(X_{\Sigma_{M, \emptyset}}\right) \to
			\Gr^{W}_{j}H^{1}\left(X_{\widehat{\Sigma}_{M, \emptyset}
					\setminus \rho_{[n]}}\right)
		\]
		for all $j$. The left term clearly vanishes, and the right term 
		vanishes by Corollary \ref{cor:weightLiftEmpty}.

		For the inductive step, \eqref{eq:lesWeight}
		reads
		\[\Gr^{W}_{j}H^{k-1}\left(X_{\widehat{\Sigma}_{M, \emptyset}
					\setminus \rho_{[n]}}\right) \to
			\Gr^{W}_{j-2}H^{k-2}\left(X_{\Sigma_{M, \emptyset}}\right)  \to 
			\Gr^{W}_{j}H^{k}\left(X_{\Sigma_{M, \emptyset}}\right) \to
			\Gr^{W}_{j}H^{k}\left(X_{\widehat{\Sigma}_{M, \emptyset}
					\setminus \rho_{[n]}}\right).
		\]
		For all $j$ and $ 2 \leq k \leq 2r-2 $, the two outer terms vanish
		by Corollary \ref{cor:weightLiftEmpty}, so
		\[ 
			\Gr^{W}_{j-2}H^{k-2}\left(X_{\Sigma_{M, \emptyset}}\right)  \cong
			\Gr^{W}_{j}H^{k}\left(X_{\Sigma_{M, \emptyset}}\right). \qedhere
		\]
	\end{proof}

	\begin{lem}\label{lem:weightEmptyHigher}
		Let $M$ be a loopless matroid of rank $ r$. For $ k \geq 2r-2 $,
		\[ 
			 H^{k}\left( X_{\Sigma_{M, \emptyset}} \right) \cong
			\Gr_{2k-2r+2}^{W}H^{k} \left( X_{\Sigma_{M, \emptyset}} \right).
		\]
	\end{lem}

	\begin{proof}
		We start by showing that $ \Gr_{j}^{W}H^{k}( X_{\Sigma_{M, \emptyset}} )= 0 $ 
		for $ k \geq 2r-3 $ and $ j > 2k-2r + 2$.
		We induct on $k$. When $ k = 2r-3 $ or
		$k = 2r-2 $, this
		is Lemma \ref{lem:weightEmptyChow}. 

		For the inductive step,
		 \eqref{eq:lesWeight} reads
		\[ 
					\Gr_{j-2}^{W} H^{k-2}\left( X_{\Sigma_{M,\emptyset}} \right)
					\to \Gr_{j}^{W} H^{k} \left( X_{\Sigma_{M, \emptyset}} \right)
					\to \Gr_{j}^{W}H^{k} \left( X_{\widehat{\Sigma}_{M, \emptyset}
					\setminus \rho_{[n]}} \right).
		\]
		The left term vanishes by the inductive hypothesis, while the
		right term vanishes by Corollary \ref{cor:weightLiftEmpty}, so
		$\Gr_{j}^{W} H^{k} ( X_{\Sigma_{M, \emptyset}}) = 0 $.

		We now show that $ \Gr_{j}^{W}H^{k}(X_{\Sigma_{M, \emptyset}}) = 0 $ for 
  	$ k \geq 2r-2 $ and $ j < 2k -2r + 2 $. 
		This will be a downward induction on $k$. For $ k$ sufficiently large,
		$H^{k}(X_{\Sigma_{M, \emptyset}}) = 0 $.

		For the inductive step, \eqref{eq:lesWeight} reads
		\[ 
			\Gr^{W}_{j+2}H^{k+1}\left(X_{\widehat{\Sigma}_{M, \emptyset}
					\setminus \rho_{[n]}}\right) \to 
			\Gr^{W}_{j}H^{k}\left(X_{\Sigma_{M, \emptyset}}\right)  \to
			\Gr^{W}_{j+2}H^{k+2}\left(X_{\Sigma_{M, \emptyset}}\right).
		\]
		The left term vanishes by Corollary \ref{cor:weightLiftEmpty},
		and the right term vanishes by the inductive hypothesis,
		so $\Gr_{j}^{W} H^{k} ( X_{\Sigma_{M, \emptyset}}) = 0 $.
	\end{proof}

	The concentration of weights implies that  \eqref{eq:lesWeight}
	often splits into short exact sequences.

	\begin{lem}\label{lem:vanishingMultiplicationMap}
		Let $M$ be a loopless matroid of rank $r$. Then, the map
		\[ 
					\Gr_{j}^{W}H^{k}\left( X_{\Sigma_{M, \emptyset}}\right)
					\to \Gr_{j+2}^{W}H^{k+2} \left( X_{\Sigma_{M, \emptyset}}
					\right) 
		\]
		in \eqref{eq:lesWeight} vanishes except when
		$ j=k $, $ 0 \leq k <  2r-2 $, and $k$ is even.
	\end{lem}

	\begin{proof}
		Suppose it is not the case that $j = k $, $ 0 \leq k < 2r-2 $, and $k$ 
		is even. 

		If $ k < 2r-2 $, then $ \Gr_{j}^{W}H^{k}\left( X_{\Sigma_{M, \emptyset}} \right)= 0 $ by
		Lemma \ref{lem:weightEmptyChow}, and the map is zero.
		On the other hand, if $ k \geq 2r-2 $, then either
		$ \Gr_{j}^{W}H^{k}( X_{\Sigma_{M, \emptyset}}) $ 
		or $ \Gr_{j+2}^{W}H^{k+2} ( X_{\Sigma_{M, \emptyset}}) $
		vanishes by Lemma  \ref{lem:weightEmptyHigher}.
		Indeed, if $ \Gr_{j}^{W}H^{k}( X_{\Sigma_{M, \emptyset}}) \neq 0 $, then
		$ j = 2k -2r + 2 $. This implies that $ j+ 2 \neq 2(k+2)-2r + 2 $, and
		$ \Gr_{j+2}^{W}H^{k}(X_{\Sigma_{M, \emptyset}}) = 0 $. In either
		case, the map in question vanishes.
	\end{proof}

	We now arrive at our computation of $ \HPoin(X_{\Sigma_{M, \emptyset}}) $.

\begin{thm}\label{thm:HPoinEmpty}
		Let $M$ be a loopless matroid of rank $ r $. Then,
		\[ 
			\HPoin\left(X_{\Sigma_{M, \emptyset}}\right) = 
						\sum_{i=0}^{r-1} x^{r} + wx^{r} \sum_{B \in \mathcal{B}(M) \setminus
						B_{\max}} (1+ wx)^{\ep(B)-1}.
		\]
	\end{thm}

	\begin{proof}
		Shifting \eqref{eq:lesWeight} reads
		\[ 
				\cdots \xrightarrow{\delta} \Gr_{j}^{W}H^{k}\left(X_{\Sigma_{M, \emptyset}}
					\right) \to \Gr_{j}^{W}H^{k}\left(X_{\widehat{\Sigma}_{M, \emptyset},
					\setminus \rho_{[n]}}
					\right)
					\to \Gr_{j-2}^{W}H^{k-1}\left(X_{\Sigma_{M,\emptyset}}\right)
					\xrightarrow{\delta} \cdots,
		\] 
			where the maps $ \delta $ are as in Lemma \ref{lem:vanishingMultiplicationMap}.
			The vanishing of $ \delta $ implies that
			\[ \dim \Gr_{j}^{W}H^{k}(X_{\Sigma_{M, \emptyset}}) +
			\dim \Gr_{j-2}^{W} H^{k-1}(X_{\Sigma_{M, \emptyset}}) = 
			\dim \Gr_{j}^{W}H^{k} ( X_{\widehat{\Sigma}_{M, \emptyset}
					\setminus \rho_{[n]}}\]
		except in the following cases:
		\begin{enumerate} 
			\item $ j = k $, $2\leq k < 2r $, and $ k $ is even, or
		  \item $ j = k+1 $, $ 1 \leq k < 2r-1 $, and $k$ is odd.
		\end{enumerate}
		In  these exceptional cases, 
		Lemma \ref{lem:weightEmptyChow} shows that
		\begin{equation*}
				 \dim \Gr_{j}^{W}H^{k}\left(X_{\Sigma_{M, \emptyset}}\right) +
					\dim \Gr_{j-2}^{W} H^{k-1}\left(X_{\Sigma_{M, \emptyset}}\right)
				= \dim \Gr_{j}^{W}H^{k} \left( X_{\widehat{\Sigma}_{M, \emptyset}
						\setminus \rho_{[n]}}
				\right) + 1.
		\end{equation*}	
		Therefore,
		\begin{equation*}\label{eq:HPoinEmpty}
			\HPoin\left( X_{\Sigma_{M, \emptyset}} \right) + wx \cdot \HPoin
			\left( X_{\Sigma_{M, \emptyset}} \right) = 
			\HPoin\left(X_{\widehat{\Sigma}_{M, \emptyset} \setminus \rho_{[n]}}
				\right) +
			\sum_{i=1}^{r-1} x^{i} + w \sum_{i=1}^{r-1} x^{i}
		\end{equation*}
	By Theorem \ref{thm:HPoinLiftEmpty} and Lemma
	\ref{lem:epMaxBasis},
	\begin{equation*}
		\HPoin\left( X_{\Sigma_{M, \emptyset}} \right) + wx \cdot \HPoin
			\left( X_{\Sigma_{M, \emptyset}} \right) = \sum_{i=1}^{r-1} x^{i} + w \sum_{i=1}^{r-1} x^{i} 
		  + 1 + wx^{r} \sum_{B \in \mathcal{B}(M)} (1 + wx)^{\ep(B)}.
	\end{equation*}

	Ordering $ \mathcal{B}(M) $ lexicographically, the following lemma shows that $ B_{\max} $,
	the lexicographically maximal basis, is the unique basis $B$  with $ \ep(B) = 0 $.
	Therefore, 
	\begin{equation*}
		\HPoin\left( X_{\Sigma_{M, \emptyset}} \right) + wx \cdot \HPoin
		\left( X_{\Sigma_{M, \emptyset}} \right) = 
		\sum_{i=0}^{r-1} x^{i} + wx \sum_{i=0}^{r-1} x^{i} 
		+ wx^{r} \sum_{B \in \mathcal{B}(M) \setminus B_{\max}} (1+wx)^{\ep(B)}.
	\end{equation*}
					and we may divide everything by $ 1 + wx $ to solve for $ \HPoin(X_{\Sigma_{M, \emptyset}}) $.
	\end{proof}

		\begin{lem}\label{lem:epMaxBasis}
			Let $M$ be a loopless matroid. Then, $ \ep(B) = 0 $ if and only
			if $ B = B_{\max} $.
		\end{lem}

		\begin{proof}
			We prove the contrapostive. If $ \ep(B) \neq 0 $,
			choose $ x \in \EP(B) $ and $ y \in C_{B \cup x} $ with $ y < x $.
			Then, $ \left( B \setminus y \right) \cup x >_{\lex} B $, so
			$ B \neq B_{\max} $.

			Conversely, suppose that $ B \neq B_{\max} $. By a symmetric version of
			basis exchange, we can find a basis $ B' >_{\lex} B $ such that
			$ \left| B' \setminus B \right| = 1 $, and we claim that the
			element $ x = B' \setminus B $ is in $ \EP(B) $. Indeed, we take
			$ y = B  \setminus B' $ and find that $ x > y $ and $ y \in C_{B \cup
			x} $. Therefore, $ x \in \EP(B) $, and $ \ep(B) > 0 $.
		\end{proof}

	Theorem \ref{thm:HPoinEmpty} in fact describes the Hodge numbers of the singular cohomology
	ring of matroids of small rank. 

	\begin{cor}\label{cor:HPoinRank2}
		Let $M$ be a loopless matroid of rank $ 1 $ or a
					simple matroid of rank $2 $. Then,
			\[ 
					\HPoin\left(X_{\Sigma_{M}} \right) = 1 + x + wx^{2} \sum_{B \in
					\mathcal{B}(M) \setminus B_{\max}} (1+wx)^{\ep(B)-1}.
			\]
	\end{cor}

	\begin{proof}
		For such matroids, $ \Sigma_{M, \emptyset} = \Sigma_{M} $.
	\end{proof}

  From the support of the Hodge--Poincar\'{e} polynomial, we read off the following sharp 
	vanishing result.

	\begin{cor}\label{cor:vanishingEmpty}
		Let $M$ be a loopless matroid of rank $ r $ on the ground set $[n]$. Then,
		\begin{enumerate}
			\item $\Gr_{j}^{W}H^{k} (X_{\Sigma_{M,\emptyset}}) = 0 $ whenever
							$ j > n - r + k $ or $ j < 2k - 2r + 2 $.
			\item $H^{k}(X_{\Sigma_{M, \emptyset}}) = 0 $ whenever $ k > n+r- 2 $.
		\end{enumerate}
		Moreover, $ \Gr_{2n-2}^{W} H^{n+r-2}(X_{\Sigma_{M, \emptyset}}) $ is
		the top-weight cohomology, and its dimension is $ \mu(M) $.
	\end{cor}

	In Section \ref{section:flips}, we show this sharp vanishing result holds for
	$H^{\bullet}(X_{\Sigma_{M, \mathcal{P}}}) $ for an arbitrary order filter $\mathcal{P} $.

	\begin{example}\label{ex:uniformEmpty}
		Let $ U_{r,n} $ be the rank $ r $ uniform matroid on the ground set $[n]$ for $ r \geq 1 $.
		We will show that	
		\begin{equation}\label{eq:uniformEmpty} 
			\HPoin\left( X_{\Sigma_{U_{r,n}, \emptyset}} \right) =
						\sum_{i=0}^{r-1} x^{i} + \sum_{k=0}^{n-r-1} \binom{r+k}{r-1} \cdot 
				wx^{r}(1+wx)^{k}.
		\end{equation}	
		For $-1 \leq k \leq n-r-1 $, there are $ \binom{r+k}{r-1} $ bases of 
		$U_{r,n} $ with minimal element $ n-r-k $. These are precisely the bases $B$ 
		with $ \ep(B) =k+1 $. The maximal basis has minimal element $ n-r+1 $. 
		Equation \eqref{eq:uniformEmpty} then follows from Theorem
		\ref{thm:HPoinEmpty}.
	\end{example}
	
		\section{Vanishing results for the singular cohomology ring of a matroid}\label{section:flips}

		In this section, we prove Theorem \ref{thm:vanishingCohomology}, which extends the 
		sharp vanishing result of Corollary \ref{cor:vanishingEmpty}  
		to an arbitrary order filter $ \mathcal{P} $.
		We begin by studying how the cohomology $ H^{\bullet}(X_{M, \mathcal{P}}) $ changes under
		a single matroidal flip.

		\subsection{The long exact sequence of a matroidal flip}
				Fix a matroidal flip $ \Sigma_{M, \mathcal{P}_{-}} \rightsquigarrow \Sigma_{M, \mathcal{P}_{+}} $
		with center $ Z = \mathcal{P}_{+} \setminus \mathcal{P}_{-}  $.
		We construct \textit{the long exact sequence of a matroidal flip} comparing 
		$ H^{\bullet}(X_{\Sigma_{M, \mathcal{P}_{-}}}) $ and 
		$ H^{\bullet}(X_{\Sigma_{M, \mathcal{P}_{+}}}) $. Our construction is inspired by blow-ups.

		Combinatorially, $ \Sigma_{M, \mathcal{P}_{+}} $ can be described by taking
		the stellar subdivision (Definition \ref{def:stellarSubdivision})
		of $ \Sigma_{M, \mathcal{P}_{-}} $ along
		$ \sigma_{Z < \emptyset} $ and then removing all cones $ \sigma_{I < \mathcal{F}}$
		with $ \closure(I) = Z $. The induced map $ X_{\Sigma_{M, \mathcal{P_{+}}}} 
		\to X_{\Sigma_{M, \mathcal{P}_{-}}} $ is therefore an open inclusion composed with a
		blow-down map.

		In the case of an actual blow-down map $ X' \to X $, a comparison of Mayer--Vietoris sequences 
		produces a long exact sequence relating $H^{\bullet}(X) $ and $ H^{\bullet}(X') $
		(see Appendix \ref{section:blowups}). Matroidal flips are close enough to blow-down that
		we can make a similar comparison of Mayer--Vietoris sequences on 
		$ X_{\Sigma_{M, \mathcal{P}_{-}}} $ and $ X_{\Sigma_{M, \mathcal{P}_{+}}} $ to 
		produce the long exact sequence of a matroidal flip. 

		The following subfans induce the open covers.

		\begin{definition}
			For a matroidal flip $ \Sigma_{M, \mathcal{P}_{-}} \rightsquigarrow
			\Sigma_{M, \mathcal{P}_{+}} $ with center $ Z =
			\mathcal{P}_{+} \setminus \mathcal{P}_{-} $, define
			\[	H_{-} := \left\{ \sigma_{I < \mathcal{F}} : 
						\closure(I) \neq Z \right\} \subseteq
						\Sigma_{M, \mathcal{P}_{-}} \;\;\; \textrm{and} \;\;\; 
					H_{+} := \left\{ \sigma_{I < \mathcal{F}} : Z 
					\notin \mathcal{F} \right\}
							\subseteq \Sigma_{M, \mathcal{P}}.\]
		\end{definition}
		The fans $ \Sigma_{M, \mathcal{P}_{-}} $ and $ \Sigma_{M, \mathcal{P}_{+}} $
		are defined on the same vector space $ N_{\mathbb{Q}} $, so their cones 
		and subfans may	be compared.

		\begin{lem}\label{lem:equalityFans}
			As rational fans, $ H_{-}=  H_{+} $, and 
			$ \overline{\Star}(\sigma_{Z < \emptyset}) \cap H_{-} 
				= \overline{\Star}(\rho_{Z}) \cap H_{+} $.
		\end{lem}

		\begin{proof}
			For the first claim, take $ \sigma_{I < \mathcal{F}} \in H_{-} $.
			As $ \closure(I) $ is neither in $ \mathcal{P}_{-} $ nor is 
			$ Z $, we see that $ \closure(I) \notin \mathcal{P}_{+} $. 
			Therefore, $ \sigma_{I < \mathcal{F}} \in \Sigma_{M, \mathcal{P}_{+}} $. 
			Moreover, $ Z \notin \mathcal{F} $, so $ \sigma_{I < \mathcal{F}} \in H_{+} $.
			Conversely, for $ \sigma_{I < \mathcal{F}} \in H_{+} $, we have
			$ \closure(I) \notin \mathcal{P}_{+} = \mathcal{P}_{-} \cup Z $
			and $ \mathcal{F} $ a chain in $ \mathcal{P}_{-} $. Therefore 
			$ \sigma_{I < \mathcal{F}} \in H_{-} $.

			For the second claim, we will show that $ \overline{\Star}(\sigma_{Z <
						\emptyset}) \cap H_{-} $ and $ \overline{\Star}(\rho_{Z}) \cap H_{+} $ 
			consist of cones $ \sigma_{I < \mathcal{F}} \in H_{-} (= H_{+}) $ such that
			$ I \subseteq Z $ and $ Z < \mathcal{F} $. This is clear for
			$ \overline{\Star}(\rho_{Z}) \cap H_{+} $. For 
			$ \overline{\Star}(\sigma_{Z < \emptyset}) \cap H_{-} $ it is clear 
			that $ Z < \mathcal{F} $. If it were the case that $ I \not\subseteq Z $, then 
			we would have
			$ \sigma_{ I \cup Z < \mathcal{F}} \in \Sigma_{M, \mathcal{P}_{-}} $.
			But this is a contradiction, as
			$ \closure(I \cup Z) $ strictly contains $ Z $, so it is in
			$ \mathcal{P}_{-} $ because it is an order filter.
		\end{proof}

		\begin{lem}\label{lem:fanMV}
			As rational fans, $ \Sigma_{M, \mathcal{P}_{-}} = \overline{\Star}(\sigma_{Z<
			\emptyset}) \cup H_{-} $, and $ \Sigma_{M, \mathcal{P}_{+}} = 
			\overline{\Star}(\rho_{Z}) \cup H_{+} $.
		\end{lem}

		\begin{proof}
			If $ \sigma_{I < \mathcal{F}} \in \Sigma_{M, \mathcal{P}_{-}}  \setminus H_{-} $,
			then $ \closure(I) = Z $, and $ \sigma_{I< \mathcal{F}} \in 
			\overline{\Star}(\sigma_{Z < \emptyset})$.
			Similarly, if $ \sigma_{I < \mathcal{F}} \in \Sigma_{M, \mathcal{P}_{+}} \setminus H_{+} $,
			then $ Z \in \mathcal{F} $ and $ \sigma_{I < \mathcal{F}} \in \overline{\Star}(\rho_{Z}) $.
		\end{proof}

		The identity map $ \mathrm{id}\colon N \to N $ induces a map of 
		rational fans
		$ \Sigma_{M, \mathcal{P}_{+}} \to \Sigma_{M, \mathcal{P}_{-}} $.
		The subfans $ H_{-} $ and $ H_{+} $ were defined so that this map restricts
		to submaps
		$ \overline{\Star}(\rho_{Z}) \to \overline{\Star}(\sigma_{Z < \emptyset}) $,
		$ H_{+} \to H_{-} $, and $ \overline{\Star}(\rho_{Z}) \cap H_{+} 
		\to \overline{\Star}(\sigma_{Z < \emptyset}) \cap H_{-} $. 
		These last two maps are isomorphisms by Lemma \ref{lem:equalityFans}.
		By Lemma \ref{lem:fanMV}, we have 
		open covers $ X_{\Sigma_{M, \mathcal{P}_{-}}} = 
		X_{\overline{\Star}(\sigma_{Z < \emptyset})} \cup X_{H_{-}}  $ and 
		$ X_{\Sigma_{M, \mathcal{P}_{+}}} = X_{\overline{\Star}(\rho_{Z})}
		\cup X_{H_{+}} $ which produce Mayer--Vietoris exact sequences. 
		The above maps of fans then produce commutating maps between the
		two exact sequences.
		\begin{equation}			
						\begin{tikzcd}[column sep = small]\label{eq:MV}
							\cdots \rar & H^{\bullet}\left( X_{\Sigma_{M, \mathcal{P}_{+}}} \right)
							\rar & H^{\bullet} \left( X_{\overline{\Star}(\rho_{Z})} \right) 
							\oplus H^{\bullet} \left( X_{H_{+}} \right) \rar & 
							H^{\bullet} \left( X_{\overline{\Star}(\rho_{Z}) \cap H_{+}} \right) 
							\rar & \cdots \\
							\cdots \rar & H^{\bullet}\left( X_{\Sigma_{M, \mathcal{P}_{-}}} \right)
							\rar \uar["\phi^{*}"] & H^{\bullet} \left( X_{
											\overline{\Star}(\sigma_{Z < \emptyset})} \right) 
							\oplus H^{\bullet} \left( X_{H_{-}} \right) \rar 
							\uar["\phi^{*} 
							\oplus  \,
										\phi^{*} "]
							& 
							H^{\bullet} \left( X_{\Pi_{-} \cap H_{-}} \right) \rar 
										\uar["\phi^{*}"] & \cdots.
			\end{tikzcd}
		\end{equation}

		We wish to apply Lemma \ref{lem:LESBlowup} to \eqref{eq:MV} in order to get a long exact
		sequence relating $ H^{\bullet}(X_{\Sigma_{M, \mathcal{P}_{-}}}) $ and
		$H^{\bullet}(X_{\Sigma_{M, \mathcal{P}_{+}}}) $. The next lemma shows that
		the hypotheses of Lemma \ref{lem:LESBlowup} are satisfied.

	\begin{lem}
					The map $ \phi^{*}
					\colon H^{\bullet}(X_{\overline{\Star}(\sigma_{Z < \emptyset})})
					\to H^{\bullet}(X_{\overline{\Star}(\rho_{Z})})$ 
					is injective, and 
					\[ \coker\phi^{*} 
					\cong \frac{H^{\bullet}\left(X_{\Sigma_{M^{Z},
					\emptyset}} \right) \otimes H^{\bullet}\left( X_{\Sigma_{M_{Z}}}
					\right)}{1 \otimes H^{\bullet}\left( X_{\Sigma_{M_{Z}}}
					\right)}.\]
	\end{lem}

	\begin{proof}
		We first compute $H^{\bullet}(X_{\overline{\Star}(\rho_{Z})}) $.
		By Lemma \ref{lem:starIsomorphism}, this is isomorphic to
		$ H^{\bullet}(X_{\Star(\rho_{Z})}) $. 
		Because $ \Star(\rho_{Z}) \cong \Sigma_{M^{Z}, \emptyset}
		\times \Sigma_{M_{Z}} $ as rational fans, the K\"{u}nneth formula implies
		that 
		\[ 
			 H^{\bullet}\left( X_{\overline{\Star}(\rho_{Z})}\right)
					\cong H^{\bullet}\left( X_{\Sigma_{M^{Z}, \emptyset}} \right)
					\otimes H^{\bullet} \left( X_{\Sigma_{M_{Z}}} \right).
		\]
				
		To compute $H^{\bullet}(X_{\overline{\Star}(\sigma_{Z < \emptyset})}) $, we
		note that $ \Star(\sigma_{Z < \emptyset})  \cong \Sigma_{M_{Z}} $ as
		rational fans. Again using Lemma \ref{lem:starIsomorphism},
		\[ 
					H^{\bullet}\left( X_{\overline{\Star}(\sigma_{Z < \emptyset})}\right)
					\cong H^{\bullet} \left( X_{\Sigma_{M_{Z}}} \right).
		\]

		These isomorphisms are realized in  Koszul homology by
		quasi-isomorphic subcomplexes
		\[ 
			K_{\bullet}(\mathbb{Q}[\Sigma_{M^{Z}, \emptyset}]) \otimes
			K_{\bullet}(\mathbb{Q}[\Sigma_{M_{Z}}] ) \hookrightarrow
			K_{\bullet}(\mathbb{Q}[\overline{\Star}(\rho_{Z})]) \;\; \textrm{and}
			\;\; K_{\bullet}(\mathbb{Q}[\Sigma_{M_{Z}}] ) \hookrightarrow
					K_{\bullet}(\mathbb{Q}[\overline{\Star}(\sigma_{Z < \emptyset})]).
		\]

		We now show that the map 
		$ \phi^{*}|_{\overline{\Star}(\sigma_{Z< \emptyset})} $
		is injective and compute the cokernel via these subcomplexes.
		Let $ \Phi \colon \mathbb{Q}[\overline{\Star}(\sigma_{Z< \emptyset})] \to 
		\mathbb{Q}[\overline{\Star}(\rho_{Z})] $ be the map defined on indeterminates by
		\[ 
			\Phi(x) = \begin{cases}
							x+ x_{\rho_{Z}} & x = x_{\rho_{i}} \textrm{ for $i \in Z $}\\
							x & \textrm{any other indeterminate}.
			\end{cases}
		\]
		Then, $ \Phi $ extends to a map 
		$
		\Phi \colon K_{\bullet}(\mathbb{Q}[\overline{\Star}(\sigma_{Z < \emptyset})])
		\to K_{\bullet}(\mathbb{Q}[\overline{\Star}(\rho_{Z})]) $.
		By Theorem \ref{thm:franzFu}, 
		$ \phi^{*}|_{\overline{\Star}(\sigma_{Z< \emptyset})} $ is realized by $ \Phi $.
		Restricting to the quasi-isomorphic subcomplexes,
		the map on homology is the natural inclusion
		\[ 
			1 \otimes H^{\bullet} \left( X_{\Sigma_{M_{Z}}} \right)
			\hookrightarrow H^{\bullet}\left( X_{\Sigma_{M^{Z}, \emptyset}}
			\right) \otimes H^{\bullet} \left( X_{\Sigma_{M_{Z}}} \right). \qedhere
		\]
		\end{proof}

		We will use the following notation for $ \coker \phi^{*}\big|_{
		\overline{\Star}(\sigma_{Z < \emptyset})} $:
			\[ 
					E_{Z}^{k} := \frac{ \bigoplus_{k_{1} + k_{2} = k}
						H^{k_{1}} \left( X_{\Sigma_{M^{Z}, \emptyset}} \right)
						\otimes H^{k_{2}} \left( X_{\Sigma_{M_{Z}}} \right)}{
						1 \otimes H^{k} \left( X_{\Sigma_{M_{Z}}} \right)},
			\]
			and 
			\[ 
						\Gr_{j}^{W} E_{Z}^{k} := \frac{ \bigoplus_{\substack{j_{1} + j_{2} = j \\
						k_{1} + k_{2}= k}}
						\Gr_{j_{1}}^{W} H^{k_{1}} \left( X_{\Sigma_{M^{Z}, \emptyset}} \right)
						\otimes \Gr_{j_{2}}^{W} H^{k_{2}} \left( X_{\Sigma_{M_{Z}}} \right)}{
										1 \otimes \Gr_{j}^{W} H^{k} \left( X_{\Sigma_{M_{Z}}} \right).}
			\]

		\begin{lem}\label{lem:LESMatroidalFlip}
			For a matroidal flip $ \Sigma_{M, \mathcal{P}_{-}} \rightsquigarrow
			\Sigma_{M, \mathcal{P}_{+}} $ with center $ Z $, there is a long
			exact sequence
				\[ \cdots \to H^{\bullet}\left(X_{\Sigma_{M, \mathcal{P}_{-}}}
				\right) \to H^{\bullet} \left( X_{\Sigma_{M, \mathcal{P}_{+}}} \right)
						\to E_{Z}^{\bullet} \to \cdots.
				\]
				Moreover, this exact sequence preserves the weight filtration, and
				this induces the \emph{long exact sequence of a matroidal flip}
				\[ 
					\cdots \to  \Gr_{\bullet}^{W}E^{\bullet-1}_{Z} 
						\to \Gr_{\bullet}^{W}H^{\bullet}\left(X_{\Sigma_{M, \mathcal{P}_{-}}}
						\right) \to \Gr_{\bullet}^{W} H^{\bullet} 
						\left( X_{\Sigma_{M, \mathcal{P}_{+}}} \right)
						\to \Gr^{W}_{\bullet} E_{Z}^{\bullet} \to \cdots.
						\]
		\end{lem}

		\begin{proof}
			All but the weight filtration is immediate from Lemma 
			\ref{lem:LESBlowup}. For this last part, we note that Mayer--Vietoris exact sequences for
			toric varieties preserve the weight filtration. Indeed,
			if we have a covering of fans $ \Sigma = \Sigma_{1} \cup 
			\Sigma_{2} $, then
			the Mayer--Vietoris exact sequence is obtained by applying the Tor functors
			to the exact sequence
			\[ 
						0 \to \mathbb{Q}[\Sigma] \to \mathbb{Q}[\Sigma_{1}]\oplus
						\mathbb{Q}[\Sigma_{2}] \to \mathbb{Q}[\Sigma_{1} \cap 
						\Sigma_{2}] \to 0.
			\]
			The associated long exact sequence of Tor groups preserves the bigrading of Tor,
			which corresponds to the weight filtration. Finally,  the construction of
			Lemma \ref{lem:LESBlowup} also preserves the weight filtration.
		\end{proof}

	\subsection{Vanishing results}\label{section:vanishing}
		We now use the long exact sequence of a matroidal flip from
		Lemma \ref{lem:LESMatroidalFlip} to prove the following sharp vanishing result.
		
		\begin{thm}\label{thm:vanishingCohomology}
			Let $ M$ be a loopless matroid of rank $ r $ on the ground set $[n]$, and
			let $ \mathcal{P} $ be an order filter on $ \overline{\mathcal{L}}(M) $. Then,
			\begin{enumerate}
				\item $ \Gr_{j}^{W} H^{k}(X_{\Sigma_{M, \mathcal{P}}}) = 0 $ whenever
								$ j > n -r+k $ or $ j < 2k -2r + 2 $.
				\item $ H^{k}(X_{\Sigma_{M, \mathcal{P}}}) = 0 $ whenever $ k > n + r - 2 $.
			\end{enumerate}
			Moreover, $ \Gr^{W}_{2n-2} H^{n+r-2}(X_{\Sigma_{M, \mathcal{P}}}) $ is the top-weight
			cohomology, and its dimension is $ \mu(M) $.
		\end{thm}

		When $ \mathcal{P} = \overline{\mathcal{L}}(M) $, this is the sharp vanishing result 
		for the singular cohomology ring of a matroid announced in the introduction.

		\begin{proof}
			We make a nested induction, first inducting on the rank  of $M$,
			and then inducting on the cardinality of $ \mathcal{P} $. 
			The base case for the induction on rank is for loopless rank 1 matroids. 
			Here, $ \Sigma_{M, \mathcal{P}} = \Sigma_{M, \emptyset} $, so the result follows
			from Corollary \ref{cor:vanishingEmpty}.
			The base case for the induction on cardinality of order filters is also 
			Corollary \ref{cor:vanishingEmpty}.

			For the induction, we may find an order filter $ \mathcal{P}' $ such that
			$ Z = \mathcal{P} \setminus \mathcal{P}' $ is a single flat. Thus $ \Sigma_{M, \mathcal{P}'}
			\rightsquigarrow \Sigma_{M, \mathcal{P}} $ is a matroidal flip with center $ Z $.

			Suppose that $ j> n-r+k $. The long exact sequence of a matroidal flip reads
			\begin{equation}\label{eq:weightLESMatroidalFlip}
				\cdots \to \Gr_{j}^{W} H^{k}\left(X_{\Sigma_{M, \mathcal{P}'}}\right)
						\to \Gr_{j}^{W} H^{k}\left(X_{\Sigma_{M, \mathcal{P}}}\right)
						\to 
							\Gr_{j}^{W} E^{k}_{Z}  \to \cdots.
			\end{equation}
		By induction, $ \Gr_{j}^{W}H^{k} (X_{\Sigma_{M, \mathcal{P}'}})$
		vanishes. For any pairs $ (j_{1}, k_{1}) $ and $ (j_{2}, k_{2}) $ such
		that $ j_{1} +j_{2} = j $ and $k_{1} + k_{2} = k $, either 
		$ j_{1} > \left| Z \right| - \rk(Z) + k_{1} $ or 
		$ j_{2} > n - \left| Z \right| - r + \rk(Z) + k_{2} $. By the induction
		on rank and Corollary \ref{cor:vanishingEmpty}, 
		$ \Gr_{j}^{W}E^{k}_{Z} = 0 $.
		Therefore, the exactness of \eqref{eq:weightLESMatroidalFlip} shows
		that $ \Gr_{j}^{W}H^{k}(X_{\Sigma_{M}}) = 0 $. 
		The proofs that $ \Gr_{j}^{W}H^{k}(X_{\Sigma_{M, \mathcal{P}_{m+1}}}) = 0 $
		for $ j < 2k -2r + 2 $ and $H^{k}(X_{\Sigma_{M, \mathcal{P}_{m+1}}}) =0 $
		for $ k > n + r-2 $	are analogous.  This proves parts (1) and (2).

		For the top-weight cohomology,
		$ H^{n+r-2}(X_{\Sigma_{M, \mathcal{P}'}}) \cong
			H^{n+r-2}(X_{\Sigma_{M, \mathcal{P}}}) $.
		Indeed, the long exact sequence of a matroidal flip reads
			\begin{equation*} 
							\cdots \to E^{n+r-3}_{Z} 	
					\to H^{n+r-2}\left( X_{\Sigma_{M, \mathcal{P}'}}\right) 
					\to H^{n+r-2}\left( X_{\Sigma_{M, \mathcal{P}}} \right) 
			\to E^{n+r-2}_{Z} \to \cdots.
			\end{equation*} 
			As $ H^{k_{1}}(X_{\Sigma_{M^{Z}, \emptyset}}) =0 $
			for $ k_{1} > \left| Z \right| + \rk(Z) - 2 $,
			and $ H^{k_{2}}(X_{\Sigma_{M_{Z}}}) = 0 $ for 
			$ k_{2} > n - \left| Z \right| + r - \rk(Z) -2 $,
			both $ E^{n+r-3}_{Z}$ and $ E^{n+r-2}_{Z} $ vanish.

			The isomorphism $ H^{n+r-2}(X_{\Sigma_{M, \mathcal{P}'}}) \cong
			H^{n+r-2}(X_{\Sigma_{M, \mathcal{P}}}) $, combined 
			with parts (1) and (2), show  that 
			$ \Gr^{W}_{2n-2} H^{n+r-2}(X_{\Sigma_{M, \mathcal{P}}}) $ is the top-weight cohomology, 
			with dimension $ \mu(M) $.
		\end{proof}
		
			While the long exact sequence of a matroidal flip computes the top-weight Hodge 
			number, it remains difficult to compute the other Hodge numbers.
			The major obstruction  is that
			the boundary maps  $ \Gr_{j}^{W}E^{k-1}_{Z} \to \Gr_{j}^{W}H^{k}(X_{\Sigma_{M, \mathcal{P}_{-}}}) $
			do not vanish in general.
			Even in the case of rank 3 matroids, it is difficult to pin down
			exactly what the kernels are. 

			\begin{example}
			Take the matroid $M$ on the ground set $[6]$ with bases
			\[ 
				\mathcal{B}(M) = \left\{ 125, 126, 135, 136, 145, 146, 235, 245, 
						246, 345, 346, 356, 456\right\}.
			\]
			If all of the boundary maps for all the long exact sequences of matroidal flips
			vanished, then $ \HPoin(X_{\Sigma_{M}}) $ would be
			\[ 
				1 + 9x + x^{2} + 28wx^{2} + 14wx^{3} + 31 w^{2} x^{3} + 17 w^{2}x^{4} 
						+ 11 w^{3} x^{4} + 6w^{3}x^{5}.
			\]
			However, a direct computation on Macaulay2 shows that
			\[ 
				 \HPoin(X_{\Sigma_{M}}) = 1 + 9x + x^{2} + 28wx^{2} +7wx^{3}
						+ 24w^{2}x^{3} + 13w^{2}x^{4} + 7w^{3}x^{4} + 6w^{3}x^{5}.
			\]
			\end{example}

			There are a two cases, however, where the boundary map is known to vanish.
			These are the topics of the next two sections.
			In Section \ref{section:Chow}, we show that the
			boundary map vanishes when the target of the boundary map is in the Chow ring.
			This will imply that matroidal flips induce short exact sequence on Chow rings,
			and we use this to recover \cite{AHK}*{Theorem 6.18}.
			In Section \ref{section:uniform}, we will show that the boundary map also vanishes whenever
			the center of the matroidal flip is an independent flat. In particular,
			this is always the case for matroidal flips on uniform matroids, and from this we obtain a formula 
			for the Hodge numbers of $ X_{\Sigma_{U_{r,n}}} $.

			\section{Matroidal Flips and Chow Rings}\label{section:Chow}

			Recall that the Chow ring is given by
			\[ 
				A^{k}\left(M, \mathcal{P}\right) =
				\Gr_{2k}^{W} H^{2k}(X_{\Sigma_{M,\mathcal{P}}}).
			\]
			In this section, we show that each boundary map of the form
			\[ 
				\delta\colon \Gr_{2k}^{W}E_{Z}^{2k-1} \to \Gr_{2k}^{W}H^{2k}(X_{\Sigma_{M,\mathcal{P}_{-}}})
			\]
			in the long exact sequence of a matroidal flip vanishes.

			Our proof is a direct computation in Koszul homology. Here, 
			the domain of $ \delta $ is the sum
			\begin{multline}\label{eq:decomposition}
							\Gr_{2k}^{W}E_{Z}^{2k-1} \cong \bigoplus_{k_{1}\geq 0}
							H_{1}\left(K(\mathbb{Q}[\Sigma_{M^{Z},\emptyset}])\right)_{k_1}
							\otimes H_{0}\left(K(\mathbb{Q}[\Sigma_{M_{Z}}] \right))_{k - k_{1}} \\
							\oplus \frac{\bigoplus_{k_{2} \geq 0} 
							H_{0}\left(K(\mathbb{Q}[\Sigma_{M^{Z},\emptyset}])\right)_{k_1}
							\otimes H_{1}\left(K( \mathbb{Q}[\Sigma_{M_{Z}}]) \right)_{k - k_{2}}}
							{1 \otimes H_{1}\left(K(\mathbb{Q}[\Sigma_{M_{Z}}]) \right)_{k}}.
			\end{multline}
		We first use the results of Section \ref{section:empty} to compute an explicit
		basis for $H_{1}\left(K(\mathbb{Q}[\Sigma_{M^{Z},\emptyset}])\right)$. 

	\begin{lem}\label{lem:koszulH1}
		Let $M$ be a loopless matroid of rank $ r $ on the ground set $[n]$. Then, the set
		\[ 
			\left\{ \prod_{k \in B \setminus i_{B}} x_{\rho_{k}} \otimes 
					(\ell_{i_{B}} - \ell_{j_{B}}) : B \in \mathcal{B}(M) \setminus 
					B_{\max}, j_{B} = \max \EP(B), i_{B} = \min(C_{B \cup j_{B}}) \right\}
		\]
		is a basis for $ H_{1}(K(\mathbb{Q}[\Sigma_{M, \emptyset}])) $.
	\end{lem}

	\begin{proof}
		Let $ A $ denote the proposed basis, which is well-defined by Lemma \ref{lem:epMaxBasis}. Each element of $A$ is a cycle 
		in the Koszul complex because $B$ and $ (B \setminus i_{B}) \cup j_{B} $ are bases.
		Because $ \left| \mathcal{B}(M) \setminus 
		B_{\max} \right| =  \dim H_{1}(K(\mathbb{Q}
		[\Sigma_{M, \emptyset}])) $ by Theorem \ref{thm:HPoinEmpty},
		it suffices to show that the elements of $ A $ are linearly
		independent.

		We have the short exact sequence of $ \Sym(N_{\mathbb{Q}}^{\vee})$-algebras
		\[ 
				0 \to I_{\Sigma_{M, \emptyset}} \to \mathbb{Q}[x_{\rho} : 
					\rho \in {\Sigma_{M, \emptyset}}(1)] \to \mathbb{Q}[\Sigma_{M, \emptyset}]
					\to 0,
		\]
		where $ I_{\Sigma_{M, \emptyset}} $ is the Stanley--Reisner ideal.
		This induces a long exact sequence in Koszul homology with connecting homomorphism
		\[ 
					\delta \colon H_{1}\left(K\left(\mathbb{Q}[\Sigma_{M, \emptyset}]
					\right) \right) \to H_{0}\left(K\left( I_{\Sigma_{M, \emptyset}}
					\right)\right).
		\]
		The vector space $H_{0}(K(I_{\Sigma_{M, \emptyset}}))$
		has a basis $ \left\{ \prod_{k \in B} x_{\rho_{k}} : B \in \mathcal{B}(M) \right\} $,
		and 
		\[ 
					\delta\left( \prod_{k \in B \setminus i_{B}} x_{\rho_{k}} \otimes 
					(\ell_{i_{B}} - \ell_{j_{B}})  \right) = \prod_{k \in B} x_{\rho_{k}} 
					- \prod_{j \in (B \setminus i_{B}) \cup j_{B}} x_{\rho_{j}}.
		\]
		From this we see that the image of $ A $ under $\delta $ has dimension
		$ \left| \mathcal{B}(M) \setminus B_{\max} \right| $, so
		the elements of $ A $ are linearly independent.
	\end{proof}

	We now deduce our main lemma of the section.

	\begin{lem}\label{lem:vanishingBoundary}
		For every $ k \in \mathbb{Z} $, the boundary map
		\[ 
			\delta\colon \Gr_{2k}^{W}E_{Z}^{2k-1} \to \Gr_{2k}^{W}H^{2k}(X_{\Sigma_{M,\mathcal{P}_{-}}}).
		\]
		in the long exact sequence of a matroidal flip vanishes.
	\end{lem}

	\begin{proof}
		In the long exact sequence of a matroidal flip we have
		\[ 
			\Gr_{2k}^{W} H^{2k-1}(X_{\Sigma_{M, \mathcal{P}_{+}}}) 
			\to \Gr_{2k}^{W}E_{Z}^{2k-1} \xrightarrow{\delta} 
			\Gr_{2k}^{W}H^{2k}(X_{\Sigma_{M,\mathcal{P}_{-}}}).
		\]
		Thus, it is enough to show that $\Gr_{2k}^{W} H^{2k-1}(X_{\Sigma_{M, \mathcal{P}_{+}}}) 
		\to \Gr_{2k}^{W}E_{Z}^{2k-1} $ is surjective. 

		Consider $ \bigoplus_{k_{1}=0}
		H_{1}\left(K(\mathbb{Q}[\Sigma_{M^{Z},\emptyset}])\right)_{k_1}
					\otimes H_{0}\left(K(\mathbb{Q}[\Sigma_{M_{Z}}]) \right)_{k - k_{1}}  $ in the decomposition
		\eqref{eq:decomposition}. By Lemma \ref{lem:koszulH1}, this is generated by
		elements
		$ \prod_{k \in B \setminus i_{B}}
		x_{\rho_{k}} \otimes (\ell_{i_{B}} - \ell_{j_{B}}) \cdot \xi $,
		where $ B $ is a basis of $M^{Z} $, $ j_{B} \in Z \setminus B $,
		$ i_{B} \in C_{B \cup j_{B}} $, and $ \xi \in  
			H_{0}\left(K(\mathbb{Q}[\Sigma_{M_{Z}}]) \right) $.
		We view $ \prod_{k \in B \setminus i_{B}}
		x_{\rho_{k}} \otimes (\ell_{i_{B}} - \ell_{j_{B}}) \cdot \xi $
		as an element in $ K_{\bullet}(\mathbb{Q}[\Sigma_{M, \mathcal{P}_{+}}]) $ 
		where it is also a cycle. Thus, this piece of the decomposition
		is in the image of $ \Gr^{W}_{2k}H^{2k-1}(X_{\Sigma_{M, \mathcal{P}_{+}}}) $.

		Let $C := \bigoplus_{k_{2} \geq 0} 
		H_{0}\left(K(\mathbb{Q}[\Sigma_{M^{Z},\emptyset}])\right)_{k_1}
		\otimes H_{1}\left(K(\mathbb{Q}[\Sigma_{M_{Z}}]) \right)_{k - k_{2}} 
		/ 1 \otimes H_{1}(K(\mathbb{Q}[\Sigma_{M_{Z}}]))_{k} $ in the 
		decomposition \eqref{eq:decomposition}. 
		Because $ H_{0}(K(\mathbb{Q}[\Sigma_{M^{Z}, \emptyset}])) $ is generated  by elements
		$ x_{\rho_{i}}^{j} $ for $i \in Z $ and $ 0 \leq j \leq \rk(Z) -1 $ (cf. \cite{AHK}*{\S6.1}),
		$C$ is generated by the classes of elements
		$ x_{\rho_{i}}^{j} \cdot \xi $ where $ i \in Z $, $ 1\leq j \leq \rk(Z) -1 $, and $ \xi \in 
		H_{1}(K(\mathbb{Q}[\Sigma_{M_{Z}}])) $. The elements
		$ x_{\rho_{Z}}^{j} \cdot \xi $ are cycles in 
		$H_{1}(K(\mathbb{Q}[\Sigma_{M, \mathcal{P}_{+}}])) $, and we claim that 
		the span of such cycles surjects onto	$ C $. 
		To see this, note that the map $ H_{1}(K(\mathbb{Q}[\Sigma_{M, \mathcal{P}_{+}}])) 
		\to C $ factors through $ H_{1}(K(\mathbb{Q}[\overline{\Star}(\rho_{Z})]))  $.  
		We may find $ \ell \in
		\Sym(N_{\mathbb{Q}}^{\vee}) $ such that, in $K(\mathbb{Q}[\overline{\Star}(\rho_{Z})]) $, 
		$ d \ell = x_{\rho_{i}} + x_{\rho_{Z}} +
		\mathbf{y} $ for $ \mathbf{y} \in \mathbb{Q}[\Sigma_{M_{Z}}] $. 
		Then, $ d (\ell \cdot \xi) = x_{\rho_{i}} \cdot \xi + x_{\rho_{Z}} \cdot \xi + \mathbf{y} 
		\cdot \xi $, and this shows that the classes of $ x_{\rho_{Z}} \cdot \xi $
		and $ -x_{\rho_{i}} \cdot \xi $ are homologous in  $ C $. 
		Thus, $ x_{\rho_{i}} \cdot \xi $ is in the image of 
		$ H_{1}(K(\mathbb{Q}[\Sigma_{M, \mathcal{P}_{+}}])) $.
		Now we induct on $j$. 
		In $ H_{1}(K(\mathbb{Q}[\overline{\Star}(\rho_{Z})])) $, $ d \left( \sum_{m=1}^{j} (-1)^{m}
		x_{\rho_{Z}}^{j-m} (x_{\rho_{i}} + \mathbf{y})^{m-1} \cdot \ell \cdot \xi \right) $ 
		shows that the class of
		$ \pm x_{\rho_{i}}^{j} \cdot \xi $ is homologous to the class of 
		$ x_{\rho_{Z}}^{j} \cdot \xi $ up to the sum of classes of $ x_{\rho_{i}}^{m} \cdot \xi_{m} $
		for $ m < j $ and $ \xi_{m} $ cycles in $ H_{1}(K(\mathbb{Q}[\Sigma_{M_{Z}}])) $.
		By induction, these are all in the image of the subspace of
		$ H_{1}(K(\mathbb{Q}[\Sigma_{M, \mathcal{P}_{+}}])) $ spanned by the 
		$ x_{\rho_{Z}}^{j} \cdot \xi $. This proves our claim.

		Thus,
		$\Gr_{2k}^{W} H^{2k-1}(X_{\Sigma_{M, \mathcal{P}_{+}}}) 
		\to \Gr_{2k}^{W}E_{Z}^{2k-1} $ is surjective, and this completes the proof.
	\end{proof}

		As a consequence of the vanishing of these boundary maps, we see that 
		matroidal flips induce a short exact sequence on Chow rings.

		\begin{prop} \label{prop:SESChow}
			Let $M$ be a loopless matroid and $ \Sigma_{M, \mathcal{P}_{-}}
			\rightsquigarrow \Sigma_{M, \mathcal{P}_{+}} $ a matroidal 
			flip with center $ Z = \mathcal{P}_{+} \setminus \mathcal{P}_{-} $. 
			Then, there is a short exact sequence 
			\begin{equation} \label{eq:SESChow}
				0 \to A^{\bullet}\left(M, \mathcal{P}_{-}\right) \to
							A^{\bullet}\left(M, \mathcal{P}_{+}\right) \to 
							A^{>0}\left(M^{Z}, \emptyset\right)
							\otimes A^{\bullet}\left(M_{Z}\right) \to 0.
			\end{equation}
		\end{prop}

		\begin{proof}
			From Lemma \ref{lem:LESMatroidalFlip}, we have the exact sequence
			\begin{center}
							\begin{tikzcd}[cramped,column sep = tiny]
					\Gr_{2k}^{W}E^{2k-1}_{Z} \rar & 
					\Gr_{2k}^{W}H^{2k}\left(X_{\Sigma_{M, \mathcal{P}_{-}}}\right)  \rar
					& 
					\Gr_{2k}^{W}H^{2k}\left(X_{\Sigma_{M, \mathcal{P}_{+}}}\right) \rar
								& \Gr^{W}_{2k} E^{2k}_{Z}  \rar &
					\Gr_{2k}^{W}H^{2k+1}\left( X_{\Sigma_{M,\mathcal{P}_{-}}} \right).
				\end{tikzcd}
			\end{center}
			As $ X_{\Sigma_{M, \mathcal{P}_{-}}} $ is smooth,
			$ \Gr_{2k}^{W}H^{2k+1}( X_{\Sigma_{M,\mathcal{P}_{-}}})  $ vanishes. 
			The left-most map  vanishes
			by Lemma \ref{lem:vanishingBoundary}.
			Therefore, the  exact sequence above
			reduces to 
			\begin{center}
							\begin{tikzcd}[cramped, column sep = small, row sep = small]
					0 \rar & \Gr_{2k}^{W}H^{2k}\left(X_{\Sigma_{M, \mathcal{P}_{-}}}\right)  
							\rar  \dar[equals]& 
							\Gr_{2k}^{W}H^{2k}\left(X_{\Sigma_{M, \mathcal{P}_{+}}}\right) \rar 
									\dar[equals]
								& \Gr^{W}_{2k} E^{2k}_{Z}  \rar  \dar[equals]& 0  \\
							& A^{k}\left(M, \mathcal{P}_{-}\right) & A^{k}\left(M,\mathcal{P}_{+}
							\right) & \displaystyle 
							\bigoplus_{\substack{k_{1} + k_{2} = k\\ k_{1} > 0 }}
							A^{k_{1}}\left(M^{Z}, \emptyset \right) \otimes 
							A^{k_{2}} \left(M_{Z}\right) &
			\end{tikzcd} 
			\end{center}
			which completes the proof.
			\end{proof}

			As an corollary, we recover the decomposition of the
			Chow ring from \cite{AHK}*{Theorem 6.18}, which we state
			below after introducing some notation.

			\begin{definition}
				In the notation of \cite{AHK}, let $ \Phi_{Z}^{\bullet}\colon
				A^{\bullet}(M, \mathcal{P}_{-}) \to A^{\bullet}(M, \mathcal{P}_{+}) $
				be the map defined by $ \Phi_{Z}^{\bullet}(x_{\rho_{i}}) = 
				x_{\rho_{i}} + x_{\rho_{Z}} $ for $ i \in Z $ and 
				$ \Phi_{Z}^{\bullet}(x) =x $ for all other indeterminates.

				For $ p \geq  1 $, define
				$\Psi^{p, \bullet}_{Z}\colon A^{\bullet - p}(M_{Z}) \to  A^{\bullet}(M, \mathcal{P}_{+})$
				by $\xi \mapsto x_{\rho_{Z}}^{p} \cdot \xi $.
			\end{definition}

			\begin{cor}[\cite{AHK}*{Theorem 6.18}]
				Let $M$ be a loopless matroid, and let
				$ \Sigma_{M, \mathcal{P}_{-}} \rightsquigarrow \Sigma_{M, \mathcal{P}_{+}} $ 
				be a matroidal flip with center $ Z $. 
				Then,  for all $ q $, the map
				$\Phi^{q}_{Z} \oplus 
						\bigoplus_{p=1}^{\rk(Z) - 1} \Psi_{Z}^{p,q} $  is an isomorphism.
			\end{cor}

			\begin{proof}
				By Theorem \ref{thm:franzFu}, $ \Phi_{Z}^{q} $ is the inclusion
				$ A^{q}(M, \mathcal{P}_{-}) \hookrightarrow A^{q}(M, \mathcal{P}_{+}) $. 
				On the other hand, the image of $ \bigoplus_{p=1}^{\rk(Z)-1} \Psi^{p,\bullet}_{Z} $
				surjects onto $ \bigoplus_{m =1}^{q} A^{m}(M^{Z}, \emptyset) \otimes A^{q-m}(M_{Z}) $
				under the map in \eqref{eq:SESChow}. The proof of this second fact is verbatim the
				second half of the proof of Lemma \ref{lem:vanishingBoundary}, replacing
				$ \xi \in H_{1}(K(\mathbb{Q}[\Sigma_{M_{Z}}])) $ with 
				$ \xi \in \mathbb{Q}[\Sigma_{M_{Z}}] $.
				This shows that $\Phi^{q}_{Z} \oplus 
				\bigoplus_{p=1}^{\rk(Z) - 1} \Psi_{Z}^{p,q} $  is surjective. By dimension, it 
				must be an isomorphism.
			\end{proof}

		\section{The singular cohomology ring of a uniform matroid}\label{section:uniform}
		In this section, we show that the boundary maps in the long exact sequence of 
		a matroidal flip always vanish when the center of the matroidal flip is 
		an independent flat. This is always the case for uniform matroids, whence we determine
		a recursive formula for $ \HPoin(X_{\Sigma_{U_{r,n}, \emptyset}}) $.

		Let $ \Sigma_{M, \mathcal{P}_{-}} \rightsquigarrow
		\Sigma_{M, \mathcal{P}_{+}} $ be a matroidal flip with center
		$ Z = \mathcal{P}_{+} \setminus \mathcal{P}_{-} $.
		In Section \ref{section:flips} we described how $ \Sigma_{M, \mathcal{P}_{+}} $ 
		is constructed from $ \Sigma_{M, \mathcal{P}_{-}} $ by first taking the stellar
		subdivision along $ \sigma_{Z < \emptyset} $ and then removing all cones
		$ \sigma_{I < \mathcal{F}} $ such that $ \closure(I) = Z $. 
		When $ Z $ is an independent flat, the only such cones are
		$ \sigma_{Z< \mathcal{F}} $; these are already removed in the stellar subdivision.
		In other words, whenever a matroidal flip $ \Sigma_{M, \mathcal{P}_{-}} \rightsquigarrow
		\Sigma_{M, \mathcal{P}_{+}} $ has an independent flat as a center,
		the matroidal flip is a stellar subdivision, and the induced map
		$ X_{\Sigma_{M, \mathcal{P}_{+}}} \to X_{\Sigma_{M, \mathcal{P}_{-}}} $
		is a blow-down. This means that the induced map on cohomology,
		$ H^{\bullet}(X_{\Sigma_{M, \mathcal{P}_{-}}}) \to H^{\bullet}
		(X_{\Sigma_{M, \mathcal{P}_{+}}}), $ is injective (Theorem \ref{thm:blowup}).

		\begin{cor}
			Let $M$ be a uniform matroid and let $ \Sigma_{M, \mathcal{P}_{-}} 
			\rightsquigarrow	\Sigma_{M, \mathcal{P}_{+}} $ be a matroidal flip
			with center $ Z = \mathcal{P}_{+} \setminus \mathcal{P}_{-} $.
			Then, for all $j,k 
			\in \mathbb{Z} $, we have the short exact sequences
			\begin{equation*}
				0 \to H^{k}\left( X_{\Sigma_{M, \mathcal{P}_{-}}} \right)
							\to H^{k}\left( X_{\Sigma_{M, \mathcal{P}_{+}}} \right)
							\to E^{k}_{Z} \to 0
			\end{equation*}
			and
			\begin{equation}\label{eq:SESUniformMatroidalFlip}
							0 \to \Gr_{j}^{W} H^{k}\left( X_{\Sigma_{M, \mathcal{P}_{-}}} \right)
							\to \Gr_{j}^{W} H^{k}\left( X_{\Sigma_{M, \mathcal{P}_{+}}} \right)
							\to \Gr_{j}^{W} E^{k}_{Z} \to 0.
			\end{equation}
		\end{cor}

		\begin{proof}
			In a uniform matroid, any proper flat is independent, so the
			center of any matroidal flip is independent. By the 
			discussion above, this implies that the long exact sequence of a 
			matroidal flip splits into short exact sequences.
		\end{proof}

		These short exact sequences produce the following computation of $ \HPoin(X_{\Sigma_{U_{r,n}}}) $.

		\begin{thm}\label{thm:uniform}
			Let $ U_{r,n} $ be the uniform matroid of rank $ r $ on the ground set $[n]$. Then,
			\begin{align*}
			\HPoin\left( X_{\Sigma_{U_{r,n}}} \right) &= \HPoin
			\left( X_{\Sigma_{U_{r,n}, \emptyset}} \right) 
			+ \sum_{i=1}^{r-1} \binom{n}{i} \frac{x-x^{i}}{1-x}
			\HPoin\left( X_{\Sigma_{U_{r-i,n-i}}} \right) \\
							&= 	\sum_{i=0}^{r-1} x^{i} + \sum_{k=0}^{n-r-1} \binom{r+k}{r-1} \cdot 
				wx^{r}(1+wx)^{k} + \sum_{i=1}^{r-1} \binom{n}{i} \frac{x-x^{i}}{1-x}
			\HPoin\left( X_{\Sigma_{U_{r-i,n-i}}} \right)
			\end{align*}
		\end{thm}

		\begin{proof}
			Take a sequence of matroidal flips interpolating between
			$ \Sigma_{U_{r,n},\emptyset} $ and $ \Sigma_{U_{r,n}} $.
			No matter how we choose this, each  $ Z \in \widehat{\mathcal{L}}(U_{r,n} $
			occurs as the center of exactly one of the matroidal flips.
			We then iteratively apply the short exact sequences 
			\eqref{eq:SESUniformMatroidalFlip} to get
			\[ 
					\HPoin\left(X_{\Sigma_{U_{r,n}}} \right) =
					\HPoin\left( X_{\Sigma_{U_{r,n}, \emptyset}} \right) 
						+ \sum_{Z \in \widehat{\mathcal{L}}(U_{r,n})} \HPoin\left( E_{Z}^{\bullet}
						\right).
			\]
			From the definition of $ E_{Z}^{\bullet} $ and the fact that
			$ \dim H^{0}(X_{\Sigma_{U_{r,n}^{Z}, \emptyset}}) = 1 $,
			\[ 
					\HPoin\left(E^{\bullet}_{Z} \right) = \left[
									\HPoin\left(X_{\Sigma_{U_{r,n}^{Z}, \emptyset}}\right)-1
									\right] \cdot \HPoin\left( X_{\Sigma_{(U_{r,n})_{Z}}} \right).
			\]
			When $ \left| Z \right| = i $, $ U_{r,n}^{Z} \cong U_{i,i} $, and
						$ (U_{r,n})_{Z} \cong U_{r-i,n-i} $. In this case,
			\[ 
					\HPoin\left(X_{\Sigma_{U_{r,n}^{Z}, \emptyset}} \right) -1 
						= \frac{x-x^{i}}{1-x.}
			\]
			For $ 1 \leq i \leq r-1 $, there are $ \binom{n}{i} $ flats $ Z \in 
			\widehat{\mathcal{L}}(U_{r,n}) $ of cardinality $ i $.
			Thus,
			\[ 
					 \sum_{Z \in \widehat{\mathcal{L}}(U_{r,n})} \HPoin\left( E_{Z}^{\bullet}
						\right) = \sum_{i=1}^{r-1} \binom{n}{i} \frac{x-x^{i}}{1-x}
						\HPoin\left( X_{\Sigma_{U_{r-i,n-i}}} \right).
			\]
			This proves the first equality. The second follows from the computation of
			$ \HPoin(X_{\Sigma_{U_{r,n}, \emptyset}}) $ in Example	\ref{ex:uniformEmpty}.
		\end{proof}
		
	\section{Singular cohomology for arbitrary building sets}\label{section:buildingSets}
			We now generalize our results for the singular cohomology
			ring of a matroid to arbitrary building sets on $ \mathcal{L}(M) $. Specifically, given two 
			building sets $ \mathcal{F} $ and $ \mathcal{G} $ on $ \mathcal{L}(M) $ containing
			$[n]$, we give an explicit 
			formula relating $ \HPoin(X_{\Sigma_{M, \mathcal{F}}}) $ and $ 
			\HPoin(X_{\Sigma_{M, \mathcal{G}}}) $. From this we derive vanishing 
			results for  $ H^{\bullet}(X_{\Sigma_{M, \mathcal{F}}}) $ from those
			for the cohomology of $ H^{\bullet}(X_{\Sigma_{M}}) $.

			Let $ \mathcal{F} $ and $ \mathcal{G} $ be two building sets on $ \mathcal{L}(M) $ 
			containing the ground set $ [n] $. 
			By \cite{FM}*{Theorem 4.2}, there is a sequence of building sets
			\[ \mathcal{F} =\mathcal{F}_{0} \subseteq  \cdots  \subseteq
			\mathcal{F}_{s} = \mathcal{G}_{\max} = \mathcal{H}_{0} \supseteq \cdots
			\supseteq \mathcal{H}_{m} = \mathcal{G}\]
			such that consecutive building sets differ by a single flat. This 
			defines modifications of Bergman fans
			\[ 
				\Sigma_{M, \mathcal{F}} = \Sigma_{M, \mathcal{F}_{0}}  \rightsquigarrow  \cdots 
			\rightsquigarrow \Sigma_{M, \mathcal{F}_{s}} = 
			\Sigma_{M,\mathcal{G}_{\max}} = \Sigma_{M,\mathcal{H}_{0}}\leftsquigarrow \cdots
			\leftsquigarrow  \Sigma_{M, \mathcal{H}_{m}} = \Sigma_{M,\mathcal{G}}
			\]
			and morphisms of toric varieties
			\begin{equation}\label{eq:buildingFactorization}
				{\Sigma_{M, \mathcal{F}}} = X_{\Sigma_{M, \mathcal{F}_{0}}}  \leftarrow  \cdots 
				\leftarrow
				X_{\Sigma_{M, \mathcal{F}_{s}}} = 
				X_{\Sigma_{M,\mathcal{G}_{\max}}} = X_{\Sigma_{M,\mathcal{H}_{0}}}\rightarrow \cdots
				\rightarrow  X_{\Sigma_{M, \mathcal{H}_{m}}} = X_{\Sigma_{M,\mathcal{G}}}.
			\end{equation}
			In contrast to the matroidal flips in Section \ref{section:flips}, 
			each of these maps is truly a blow-down.

			To compute how cohomology changes along these maps, we recall that building sets 
			behave well with respect to the operations of restriction and 
			contraction in matroids.

			\begin{definition}
				Let $M$ be a matroid,
				and let $ \mathcal{G}$ be a building set on $ \mathcal{L}(M) $.
				For a flat $ F \in \mathcal{L}(M) $, we define sets
				\[ 
					 \mathcal{G}^{F} = \left\{ G \in \mathcal{G} : G \leq F \right\},	
				\]
				and
				\[ 
							\mathcal{G}_{F} = \left\{ G \vee F : G \in \mathcal{G} 
							\textrm{ and } G \not\leq F \right\}.
				\]
				By \cite{wondertopes}*{Proposition A.8}, $ \mathcal{G}^{F} $ 
				and $ \mathcal{G}_{F} $ are building sets on
				$ \mathcal{L}(M^{F}) $ and $ \mathcal{L}(M_{F}) $, respectively.
			\end{definition}

			By the factorization in \eqref{eq:buildingFactorization}, the following theorem
			suffices to relate $ \HPoin(X_{\Sigma_{M, \mathcal{F}}}) $ and 
			$ \HPoin(X_{\Sigma_{M, \mathcal{G}}}) $. 

			\begin{thm}\label{thm:changeBuildingSet}
				Let $M$ be a loopless matroid on the ground set $[n]$, and let 
				$ \mathcal{H} $ and $ \mathcal{G} $ be two building sets on
				$ \mathcal{L}(M) $ which both contain $[n]$ 
				and whose difference is a single flat $ Z =
				\mathcal{H} \setminus \mathcal{G} $. 
							Write $ \left\{ F_{1}, \dots, F_{s} \right\} = f(\mathcal{G}{\big|}_{Z}) $ for 
				the factors of $ Z $ in $ \mathcal{G} $.
				Then,
				\[ 
						\HPoin\left(X_{\Sigma_{M, \mathcal{H}}}\right) 
							= \HPoin\left(X_{\Sigma_{M, \mathcal{G}}} \right)
								+ \frac{1 - x^{s}}{1-x}
								\cdot \prod_{i=1}^{s} 
								\HPoin\left(X_{\Sigma_{M^{F_{i}}, \mathcal{G}^{F_{i}}}}
								\right) \cdot \HPoin\left( X_{\Sigma_{M_{Z}},
									\mathcal{G}_{Z}} \right). 				\]
			\end{thm}

			\begin{proof}
				By \cite{FM}*{Theorem 4.2}, $ X_{\Sigma_{M, \mathcal{H}}} $
				is the blow-up of $ X_{\Sigma_{M, \mathcal{G}}} $ along
				$ V(\tau) $, where $ \tau = \pos(F_{1}, \dots, F_{s}) $. 
				By \cite{EFMPV}*{Lemma 5.2}, we have the combinatorial isomorphism
				\begin{equation}\label{eq:isomorphismSimplicial}
						\Star(\tau) \cong 
						\prod_{i=1}^{s} 
						\Sigma_{M^{F_{i}}, \mathcal{G}^{F_{i}}} \times 
						\Sigma_{M_{Z}, \mathcal{G}_{Z}},
				\end{equation}
				and we must extend this to an isomorphism of rational fans, so that
				\[ 
							X_{\Star(\tau)} \cong
							\prod_{i=1}^{s}  X_{\Sigma_{M^{F_{i}}, \mathcal{G}^{F_{i}}}} \times
							X_{\Sigma_{M_{Z}, \mathcal{G}_{Z}}}.
				\]
			  The result then follows from Corollary \ref{cor:blowup} and the 
				K\"{u}nneth formula.

				For the isomorphism of rational fans, note that rays in $ \Star(\tau) $ correspond to 
				flats $ G \in \mathcal{G} $ satisfying either
					\begin{enumerate}
						\item $ G \leq F_{i} $ for some $ i$, or
						\item $ \left\{ G, Z \right\} $ is a nested set for $ \mathcal{H} $.
					\end{enumerate}
				The isomorphism in \eqref{eq:isomorphismSimplicial} 
				sends a ray $ \rho_{G} $ to $ \rho_{G} \in \Sigma_{M^{F_{i}},
							\mathcal{G}^{F_{i}}} $ for $ G$ satisfying condition (1) or to 
							$ \rho_{G \vee Z \setminus Z} \in \Sigma_{M_{Z}, \mathcal{G}_{Z}} $ 			   
							for $ G $ satisfying condition (2).
				In this last case, basic properties of nested sets imply that 
					\begin{equation}\label{eq:joinVersusUnion}
						G  = \begin{cases}
										G \vee Z \setminus Z & \textrm{if $ G \cap Z = \emptyset $} \\
										Z \cup (G \vee Z \setminus Z) & \textrm{otherwise}
									\end{cases}
					\end{equation}
				as subsets of $[n]$. 

				By construction, $ \Star(\tau) $ lives in 
				$ \mathbb{Z}^{n}/\left(e_{[n]}, e_{F_{1}}, \dots, e_{F_{s}} \right) 
				= \mathbb{Z}^{n}/ \left( e_{[n] \setminus Z}, e_{F_{1}}, \dots,
					e_{F_{s}} \right)$, while 
					$\prod_{i=1}^{s} \Sigma_{M^{F_{i}}, \mathcal{G}^{F_{i}}} 
						\times \Sigma_{M_{Z}, \mathcal{G}_{Z}} 
							$ lives in
				$ \prod_{i=1}^{s} \mathbb{Z}^{\left| F_{i} \right|} / (e_{F_{i}})  \times
					\mathbb{Z}^{n -  \left| Z \right|}/(e_{[n] \setminus Z }) $.
				(For a subset $ F \subseteq [n] $, we view $ \mathbb{Z}^{\left| F \right|}
					\subseteq \mathbb{Z}^{n} $ as the sublattice spanned by basis vectors
					$ e_{i} $ with $ i \in F $.)
				As $ F_{1}, \dots, F_{s} $ are the factors of $ Z $ in a building set,
				it is well-known that  
				$ Z = F_{1} \sqcup \cdots \sqcup F_{s} $ as subsets of $[n]$, 
				and from here it is clear  
				that the two lattices are naturally isomorphic. Under this map,
				$ \rho_{G} \in \Star(\tau) $ is sent to $ \rho_{G} \in 
				\Sigma_{M^{F_{i}}, \mathcal{G}^{F_{i}}} $ for $G$ satisfying
				condition (1). Using \eqref{eq:joinVersusUnion}, 
				$ \rho_{G} \in \Star(\tau) $ is sent to $ \rho_{G \vee Z \setminus Z}
				\in \Sigma_{M_{Z}, \mathcal{G}_{Z}} $
				for $ G $ satisfying condition (2). Therefore, the combinatorial
				isomorphism in \eqref{eq:isomorphismSimplicial} is an isomorphism
				of rational fans.
			\end{proof}

			\begin{rmk}
				If we specialize the Hodge--Poincar\'{e} polynomials in
				Theorem \ref{thm:changeBuildingSet} to $ w = 0 $, we
				recover \cite{EFMPV}*{Theorem 5.1}.
			\end{rmk}

			Finally, we see that the vanishing results for the singular cohomology 
			of a matroid transfer to the singular cohomology ring of
			$ X_{\Sigma_{M, \mathcal{F}}} $ for arbitrary building sets.

			\begin{cor}
					Let $ M$ be a loopless matroid of rank $ r $ on the ground set $[n]$,
					and let $ \mathcal{G} $ be a building set on $ \mathcal{L}(M) $
					which contains $[n]$.
					Then,
					\begin{enumerate}
						\item $ \Gr_{j}^{W} H^{k}(X_{\Sigma_{M,\mathcal{G}}}) = 0 $ whenever
								$ j > n -r+k $ or $ j < 2k -2r + 2 $.
						\item $ H^{k}(X_{\Sigma_{M,\mathcal{G}}}) = 0 $ 
										whenever $ k > n + r - 2 $.
					\end{enumerate}
							Moreover, $ \Gr^{W}_{2n-2} H^{n+r-2}(X_{\Sigma_{M,\mathcal{G}}}) $ is 
							the top-weight cohomology, and its dimension is
							$ \mu(M) $.
			\end{cor}

			\begin{proof}
				By \cite{FM}*{Theorem 4.2} there is a sequence of 
				building sets $ \mathcal{G} = \mathcal{G}_{s} \subseteq
				\cdots \subseteq \mathcal{G}_{0} = \mathcal{G}_{\max} $ such that
				consecutive pairs differ by a single flat. The proof is a nested induction
				on the rank of $M$ and the index of building sets. 
				The base case for rank is when $ \rk(M) = 1 $, in which case the
				only building set is the maximal one, so the statement follows from
				Theorem \ref{thm:vanishingCohomology}. The base case for building sets is 
				$ \mathcal{G}_{\max} $ which is also Theorem 
				\ref{thm:vanishingCohomology}.

				The proof of the inductive step is analogous to that in Theorem 
				\ref{thm:vanishingCohomology}, and we omit it.
			\end{proof}

		\appendix

		\section{Cohomology of toric blow-ups}\label{section:blowups}

			In this appendix, we work out the cohomology of smooth toric
			blow-ups through explicit computations in Koszul homology. Our approach follows
			\cite{griffithsHarris}*{Chapter 4.6}. 

			Fix $ \Sigma \subseteq N_{\mathbb{Q}} $ a rational unimodular fan and a cone $ \tau \in
			\Sigma $. 
			Combinatorially, the blow-up of $ X_{\Sigma} $ along $ V(\tau) $ 
			corresponds to the \textit{stellar subdivision} of $ \Sigma $ 
			along $ \tau $.

		\begin{definition}\label{def:stellarSubdivision}
			Let $ \tau(1) = \left\{ \rho_{1}, \dots, \rho_{k} \right\}  $ be the rays of $ \tau $ and
			$ \left\{ u_{1}, \dots, u_{k} \right\} $ their primitive generators.
			We define the ray $ \rho_{\tau} := \pos(u_{1} + \cdots + u_{k}) $.

			For a cone $ \sigma  \in \Sigma $, define the subdivision of $ \sigma $
			\[ 
					\Sigma_{\sigma}^{*}(\tau) := \begin{cases}
					\left\{ \sigma \right\} & \textrm{if $\tau \not\leq \sigma $} \\
									\left\{ \pos\left(\sigma(1) \setminus \rho_{i}, \rho_{\tau}
									\right) : 1 \leq i \leq k \right\} & \textrm{if 
									$ \tau \leq \sigma. $}
					\end{cases}
			\]
			The \emph{stellar subdivision} of $ \Sigma $ along  $ \tau $
			is the rational fan
						\[ \Sigma^{*}(\tau)  := \bigcup_{\sigma \in \Sigma}
						\Sigma_{\sigma}^{*}(\tau).\]
			See Figure \ref{fig:fanSubdivision} for an example.
		\end{definition}

		\begin{figure}
			\centering
    \begin{tikzpicture}[scale=2]
        \filldraw[gray] (0,0) -- (.9,0) -- (0,.9) -- cycle;
        \filldraw[gray] (0,0) -- (.9,0) -- (-.56,-.56) -- cycle;
        \filldraw[gray] (0,0) -- (0,.9) -- (-.56,-.56) -- cycle;
						\draw[->] (0,0) -- (1,0) node[anchor = west] {$ \rho_{1}$};
						\draw[->] (0,0) -- (0,1) node[anchor = south] {$\rho_{2} $};
						\draw[->] (0,0) -- (-.66,-.66) node[anchor = north east]
						{$ \rho_{3}$};
    \end{tikzpicture}\hspace{4em}%
     \begin{tikzpicture}[scale=2]
        \filldraw[gray] (0,0) -- (.56,.56) -- (0,.9) -- cycle;
        \filldraw[gray] (0,0) -- (.9,0) -- (.56,.56) -- cycle;
        \filldraw[gray] (0,0) -- (.9,0) -- (-.56,-.56) -- cycle;
        \filldraw[gray] (0,0) -- (0,.9) -- (-.56,-.56) -- cycle;
						 \draw[->] (0,0) -- (1,0) node[anchor =west] {$\rho_{1}$};
						 \draw[->] (0,0) -- (0,1) node[anchor = south] {$\rho_{2}$};
						 \draw[->] (0,0) -- (.66,.66) node [anchor = south west]
						 {$ \rho_{\tau} $};
						 \draw[->] (0,0) -- (-.66,-.66) node [anchor = north east]
						 {$ \rho_{3}$};
    \end{tikzpicture}
						\caption{$ \Sigma_{U_{3,3}, \emptyset} $ and the stellar subdivision
						along $ \tau = \pos(\rho_{1}, \rho_{2}) $.}
    \label{fig:fanSubdivision}
		\end{figure}

			As $ \Sigma^{*}(\tau) $ subdivides $ \Sigma $,
			the identity $N \to N $ defines a 
			toric morphism $ \pi_{\tau}\colon X_{\Sigma^{*}(\tau)}
			\to X_{\Sigma} $ which is the \emph{blow-down map}.
			The \emph{exceptional divisor} of the blow-up is 
			$ V(\rho_{\tau}) = X_{\Star(\rho_{\tau})} $. The
			\emph{center} of the blow-up is
			$ V(\tau) = X_{\Star(\tau)} $. 

			The cohomology of the blow-up can be expressed in terms of the cohomology of the base space,
			center, and exceptional divisor.

		\begin{thm}\label{thm:blowup}
			Let $ \Sigma $ be a unimodular fan, and fix a cone
			$ \tau \in \Sigma $. For each $i \in \mathbb{Z} $, 
			there is an exact sequence
			of vector spaces
			\[ 
						0 \to H^{i}\left( X_{\Sigma}\right) 
						\to H^{i}\left(X_{\Sigma^{*}(\tau)}\right) 
						\to \bigoplus_{i=1}^{\dim \tau-1} 
						H^{\bullet -2i}(X_{\Star(\tau)})
						\to 0.
			\]
		\end{thm}

			To prove this, we begin by computing the cohomology of the exceptional divisor and 
			comparing it to the cohomology of the center. Then, we produce the desired exact
			sequences by comparing Mayer--Vietoris exact sequences of the blow-up and base space.

			The following two lemmas are general facts about Koszul homology and the cohomology 
			of smooth toric varieties.
			The first lemma compares the cohomology of the star
			and closed star of a cone $ \sigma \in \Sigma $ (Definition \ref{def:star}).
			Note that 
			$ \mathbb{Q}[\overline{\Star}(\sigma)] \cong \mathbb{Q}[\Star(\sigma)][x_{\rho} : \rho \in
			\sigma(1)] $, so we may view $ \mathbb{Q}[\Star(\sigma)] \subseteq \mathbb{Q}[\overline{\Star}
			(\sigma)] $.
			The projection $ N \to N/ \spann(\sigma) $ induces
			a toric morphism $ \pi \colon X_{\overline{\Star}(\sigma)} \to X_{\Star(\sigma)} $.

		\begin{lem}\label{lem:starIsomorphism}
			The map $ \pi^{*}\colon H^{\bullet}( X_{\Star(\sigma)})
					\to H^{\bullet} ( X_{\overline{\Star}(\sigma)}) $  is an isomorphism.
		\end{lem}

		\begin{proof}
			Let $ \sigma(1) = \left\{ \rho_{1}, \dots, \rho_{k} \right\} $. By unimodularity,
			we can choose  $ \left\{ \ell_{1}, \dots, \ell_{k} \right\} 
			\subseteq N^{\vee} $ such that $ \langle \ell_{i}, u_{\rho_{j}}
			\rangle = \delta_{i,j} $, the Kronecker delta function.
			This induces a splitting
			$ N^{\vee} = ( N / \spann(\sigma))^{\vee} \oplus
			\mathbb{Z} \ell_{1} \oplus \cdots \oplus \mathbb{Z} \ell_{k} $. 
			By Theorem \ref{thm:franzFu}, $ \pi^{*} $ is represented in Koszul homology
			by the inclusion of complexes
			\[ 
				K_{\bullet}(\mathbb{Q}[\Star(\sigma)]) =
					\mathbb{Q}[\Star(\sigma)] \otimes \bigwedge^{\bullet} 
					(N / \spann(\sigma))^{\vee}_{\mathbb{Q}} \hookrightarrow
						\mathbb{Q}[\overline{\Star}(\sigma)] 
						\otimes \bigwedge^{\bullet} N_{\mathbb{Q}}^{\vee}
				= K_{\bullet}(\mathbb{Q}[\overline{\Star}(\sigma)]).
			\]
			We prove this is a quasi-isomorphism by showing
			that the quotient complex $ Q_{\bullet} $ is acyclic.

			For notation, write
			$ x^{I} = \prod_{i=1}^{k} x_{\rho_{i}}^{a_{i}} $ for
			$ I = (a_{1}, \dots, a_{k}) $. Elements
			$ \alpha \in K_{\bullet}(\mathbb{Q}[\overline{\Star}(\sigma)]) $
			are written uniquely as
				\begin{equation}\label{eq:decompositionElement}
					\alpha = \sum_{I \in \mathbb{Z}_{\geq 0}^{k}} 
						x^{I} \cdot \xi_{I}, \;\;\;
					\textrm{with $ \xi  \in \mathbb{Q}[\Star(\sigma)] \otimes 
								\bigwedge^{\bullet} N_{\mathbb{Q}}^{\vee} $}.
		  \end{equation}
			Define
			$ \deg_{\sigma}(\alpha) = \max \left\{ I \in \mathbb{Z}_{\geq 0}^{k} : 
			\xi_{I} \neq 0 \right\} $, ordering $ \mathbb{Z}_{\geq 0}^{k} $ lexicographically.

			Fix a cycle $ \beta
			\in Q_{\bullet} $ which we will show is $0$. Choose a representative 
			$ \alpha \in K_{\bullet}(\mathbb{Q}[\overline{\Star}(\sigma)] $ such
			that $ \deg_{\sigma}(\alpha) $ is minimized. 
			For contradiction, assume that $ \deg_{\sigma}(\alpha) $ is nonzero, 
			and let $ i $ be the first index such that $ x_{\rho_{i}} $ divides $ x^{\deg_{\sigma}(\alpha)} $. 
			As $ \beta $ is a cycle, $ d \alpha \in K_{\bullet}(\mathbb{Q}[\Star(\sigma)])$. By the
			minimality of $ i$,
						\[ \deg_{\sigma}  \left(d \frac{x^{\deg_{\sigma}(\alpha)}}{x_{\rho_{i}}}
														\cdot \xi_{I} \right) < \deg_{\sigma}(\alpha), \] 
			where we write $ \alpha $ as in \eqref{eq:decompositionElement}, 
			Therefore, 
						\[ \deg_{\sigma} \left( \alpha - d 
						\frac{x^{\deg_{\sigma}(\alpha)}}{x_{\rho_{i}}} \cdot
						\ell_{i}\cdot \xi_{I} \right) < \deg_{\sigma}(\alpha), \]
			which is a contradiction.

			As $ \deg_{\sigma}(\alpha) = 0 $, $ \alpha $ is in the subspace
			$ \mathbb{Q}[\Star(\sigma)] \otimes \bigwedge^{\bullet}
			N_{\mathbb{Q}}^{\vee} $. If $ \alpha \notin 
			K_{\bullet}(\mathbb{Q}[\Star(\sigma)]) $, then $ d \alpha $ would
			contain terms involving $ x_{\rho_{i}} $ due to the presence of 
			some $ \ell_{i} $ in $ \alpha $. This is impossible
			as $ d \alpha \in K_{\bullet}(\mathbb{Q}[\Star(\sigma)]) $. 			
			We conclude that  $\alpha \in K_{\bullet}(\mathbb{Q}[\Star(\sigma)]) $
			and $ \beta = 0 $ in $ Q_{\bullet} $.
		\end{proof}

		The next general lemma asserts that we may use the following
		finite-dimensional \textit{square-free} Koszul complex to compute
		Koszul homology.

		\begin{definition}\label{def:squareFreeKoszul}
			Let $ \mathbb{Q}[\Sigma]^{sq} $ be the subspace of $ \mathbb{Q}[\Sigma] $ generated by 
			square-free monomials. Define
			\[ K^{sq}_{\bullet}(\mathbb{Q}[\Sigma]) := 
				 \left\{ \alpha \in \mathbb{Q}[\Sigma]^{sq} \otimes 
						\bigwedge^{\bullet} N_{\mathbb{Q}}^{\vee} :
						d \alpha \in \mathbb{Q}[\Sigma]^{sq} \otimes 
						\bigwedge^{\bullet} N_{\mathbb{Q}}^{\vee} \right\}. \]
			This is a finite-dimensional subcomplex of the Koszul complex.
		\end{definition}

		\begin{lem}\label{lem:squareFreeKoszul}
			The inclusion $ K^{sq}_{\bullet}(\mathbb{Q}[\Sigma]) \hookrightarrow
			K_{\bullet}(\mathbb{Q}[\Sigma]) $ is a quasi-isomorphism.
		\end{lem}

		\begin{proof}
			We will show that the quotient complex $ Q_{\bullet} =
			K_{\bullet}(\mathbb{Q}[\Sigma]) / K^{sq}_{\bullet}(\mathbb{Q}[\Sigma]) $  is acyclic.

			Order the rays of $\Sigma $ as $ \rho_{1}, \dots, \rho_{n} $. We write $ x^{I} = 
						\prod_{i=1}^{n} x_{\rho_{i}}^{a_{i}} $ for $I = (a_{1}, \dots, a_{n}) $. 
			The \textit{support} of $I$ is	$ \supp(I) := \pos(\rho_{i} : a_{i} \neq 0) $, 
			which may or may not be a cone in $ \Sigma $.
			Any element of $ K_{\bullet}(\mathbb{Q}[\Sigma]) $ can be written uniquely as
			\begin{equation}\label{eq:basisAlpha}
			\alpha = \sum_{\substack{I \in \mathbb{Z}_{\geq 0}^{n} \\ \supp(I) \in \Sigma}}
			x^{I} \otimes \xi_{I}, \;\;\; 
			\textrm{where $ \xi_{I} \in \bigwedge^{\bullet} N_{\mathbb{Q}}^{\vee} $}.
			\end{equation}
			Define $ \sqdeg(\alpha) = 
		 	\max	\left\{ (\max \{a_{1} -1, 0 \}, \dots, \max \{ a_{k}-1, 0\}) :
			\xi_{(a_{1}, \dots, a_{k})}\neq 0 \right\} $.
			
			Fix a cycle $ \beta \in Q_{\bullet} $ which we must show is $ 0 $.
			Choose a representative $ \alpha \in K_{\bullet}(\mathbb{Q}[\Sigma]) $
			which minimizes $ \sqdeg(\alpha) $.
			Suppose for contradiction that $ \sqdeg(\alpha) > 0 $.
			Define $ A  = \left\{ I : \sqdeg(x^{I}) = \sqdeg(\alpha) \textrm{ and } 
			\supp(I) \in \Sigma \right\} $, and
			let $ i $ be the first index at which $ \sqdeg(\alpha) $ is not zero. 
			For each $ I \in A $, we can find $ \ell_{I} \in N_{\mathbb{Q}}^{\vee} $ such 
			that $ d \ell_{I} = x_{\rho_{i}} + \mathbf{y} $  where $ \mathbf{y} $
			is a linear form not involving $ x_{\rho} $ for $ \rho \in \supp(I)(1) $. 
			This is possible because $ \Sigma $ is unimodular.
			Then, $ \alpha' := \alpha - d \left(\sum_{I \in A} \frac{x^{I}}{x_{\rho_{i}}}
			\cdot \ell_{I} \cdot \xi_{I} \right) $ is a representative of 
			 $ \beta $ with $ \sqdeg(\alpha') < \sqdeg(\alpha) $, which is a contradiction.

			Thus, $ \beta $ has a representative $ \alpha $ with $ \sqdeg(\alpha) = 0 $.
			This means that $ \alpha \in \mathbb{Q}[\Sigma]^{sq} \otimes 
			\bigwedge^{\bullet} N_{\mathbb{Q}}^{\vee} $. As $ \beta $ is a cycle,
			$ d \alpha \in \mathbb{Q}[\Sigma]^{sq} \otimes \bigwedge^{\bullet} N_{\mathbb{Q}}^{\vee} $
			as well. Therefore, $ \alpha \in K^{sq}_{\bullet}(\mathbb{Q}[\Sigma]) $ and
			$ \beta = 0 $.
		\end{proof}

			We now compare the cohomology of the exceptional divisor with that of the
			center.
			
			\begin{lem}\label{lem:cohomologyExceptionalDivisor}
			Let $ X_{\Star(\rho_{\tau})} $ be the exceptional divisor 
			of the blow-up of $ X_{\Sigma} $ along $ V(\tau) $. As 
			vector spaces, 
			\[ 
					H^{\bullet}\left( X_{\Star(\rho_{\tau})}\right) 
					\cong \bigoplus_{i=0}^{\dim \tau -1}  H^{\bullet - 2i}
						\left( X_{\Star(\tau)} \right).
			\]
			More explicitly, for any monomials $ \mathbf{x}_{i} 
			\in \mathbb{Q}[x_{\rho} : \rho \in \tau(1)] $ with $ \deg \mathbf{x}_{i} = i $,
			the inclusion
			\[ 
					\bigoplus_{i=0}^{\dim \tau -1} \mathbf{x}_{i} \cdot
					K_{\bullet}(\mathbb{Q}[\Star(\tau)]) \hookrightarrow
					K_{\bullet}(\mathbb{Q}[\Star(\rho_{\tau})])
			\]
			is a quasi-isomorphism.
		\end{lem}

		\begin{proof}
			Write $ \tau(1)=  \left\{ \rho_{1}, \dots, \rho_{k} \right\}$.
			Choose a splitting $N^{\vee} = (N/ \spann(\tau))^{\vee}
			\oplus \mathbb{Z} \ell_{1} \oplus \cdots \oplus \mathbb{Z} \ell_{k} $
			such that $ \langle \ell_{i}, u_{\rho_{j}} \rangle = \delta_{i,j} $.
			Then, $ (N/ \spann(\rho_{\tau}))^{\vee} =
			(N /\spann(\tau))^{\vee} \oplus \mathbb{Z} (\ell_{1}-\ell_{k})
			\oplus \cdots {\oplus \mathbb{Z}(\ell_{k-1}- \ell_{k})} $.

			We first reduce to $ \left\{ \mathbf{x}_{i} \right\}_{i=1}^{\dim \tau -1} 
			= \prod_{j=1}^{i} x_{\rho_{j}} $.
			Let $ \left\{ \mathbf{x}_{i} \right\}_{i} $ and $ \left\{ \mathbf{z}_{i} \right\}_{i} $
			be two choices for monomials. 
			As $ \dim \mathbf{x}_{i} \cdot H_{\bullet}(K(\mathbb{Q}[\Star(\tau)])) =
				\dim \mathbf{z}_{i} \cdot H_{\bullet}(K(\mathbb{Q}[\Star(\tau)])) $, it suffices
			to show the images of 
			$ \bigoplus_{i=0}^{\dim \tau -1} \mathbf{x}_{i} \cdot H_{\bullet}(K(\mathbb{Q}[\Star(\tau)])) $
			and 
			$ \bigoplus_{i=0}^{\dim \tau -1} \mathbf{z}_{i} \cdot H_{\bullet}(K(\mathbb{Q}[\Star(\tau)])) $
			are equal in
			$ H_{\bullet}(K(\mathbb{Q}[\Star(\rho_{\tau})])) $. Furthermore, it suffices to 
			show this when $ \mathbf{x}_{i}  = \mathbf{z}_{i} $ except at index $i_{0}$.
			Using the boundaries of $ (\ell_{j} - \ell_{k}) $, one verifies that
			$ \mathbf{x}_{i_{0}} $ is homologous to $ \mathbf{z}_{i_{0}}$ up to an element of 
			the form $ \sum_{i=1}^{i_{0}} \alpha_{i} \cdot \mathbf{x}_{i} + \mathbf{y} $ with 
			$ \alpha_{i} \in \mathbb{Q} $ and 
			$ \mathbf{y} \in \mathbb{Q}[\Star(\tau)] $.
			This shows
			$ \mathbf{x}_{i_{0}}\cdot H_{\bullet}(K(\mathbb{Q}[\Star(\tau)])) \subseteq
			\bigoplus_{i=0}^{i_{0}} \mathbf{z}_{i} \cdot H_{\bullet}(K(\mathbb{Q}[\Star(\tau)])) $
			as subspaces of $ H_{\bullet}(K(\mathbb{Q}[\Star(\rho_{\tau})])) $. As all other monomials
			in $ \left\{ \mathbf{x}_{i} \right\}_{i} $ and $ \left\{ \mathbf{z}_{i} \right\}_{i} $
			are equal,
			$\bigoplus_{i=0}^{\dim \tau -1} \mathbf{x}_{i}\cdot H_{\bullet}(K(\mathbb{Q}[\Star(\tau)])) 
			\subseteq
			\bigoplus_{i=0}^{\dim \tau - 1} \mathbf{z}_{i} \cdot H_{\bullet}(K(\mathbb{Q}[\Star(\tau)])) $.
			By symmetry, this implies equality.

			We now prove the statement for 	$ \left\{ \mathbf{x}_{i} \right\}_{i=1}^{\dim \tau -1} 
			= \prod_{j=1}^{i} x_{\rho_{j}} $. As these monomials are square-free, 
			we restrict to the square-free complexes by	Lemma \ref{lem:squareFreeKoszul}. 
			We will show that the cokernel complex $ Q_{\bullet} :=
			K^{sq}_{\bullet}(\mathbb{Q}[\Star(\rho_{\tau})])/ 
			\bigoplus_{i=0}^{\dim \tau -1} \mathbf{x}_{i} \cdot 	H_{\bullet}(K(\mathbb{Q}[\Star(\tau)])) $ 
			is acyclic.
						
			For a subset $ W \subsetneq [k] $, write $ x_{W} = \prod_{i \in W} x_{\rho_{i}} $.
			Any element of $ K^{sq}_{\bullet}(\mathbb{Q}[\Star(\rho_{\tau})]) $ can
			be written uniquely as
			\begin{equation}\label{eq:sumExceptional}
				\alpha = \sum_{W \subsetneq [k]}
				x_{W}\cdot \xi_{W}, \;\;\; \textrm{where $ \xi_{W} \in \mathbb{Q}[\Star(\tau)]^{sq}
				\otimes \bigwedge^{\bullet} N^{\vee} $.}
			\end{equation}
			Define $ \max(\alpha)  = \max\left\{ W : \xi_{W} \neq 0\right\} $, ordering
			subsets of $[k]$ lexicographically.
			We note $ \alpha \in \bigoplus_{i=0}^{\dim \tau -1} \mathbf{x}_{i} \cdot 
			H_{\bullet}(K^{sq}(\mathbb{Q}[\Star(\tau)])) $ if and only if $ d \alpha \in 
						\bigoplus_{i=0}^{\dim \tau -1} 
			\mathbf{x}_{i} \cdot H_{\bullet}(K^{sq}(\mathbb{Q}[\Star(\tau)])) $
			and $ \max(\alpha) \leq [k-1] $.
					
			Take a cycle $ \beta \in Q_{\bullet} $ which we must show is $0$. 			
			Choose a representative $ \alpha $ such that 
			$ \max(\alpha) $ is minimized.
			For contradiction, assume that $ \max(\alpha) > [k-1] $.
			Let $ W = \max(\alpha) $, $ i_{0} = \min([k] \setminus W) $, and $ j_{0} = \max(W) $.
			Note $ i_{0} $ exists because $ W \neq [k] $, and $ i_{0} < j_{0} $ because
			$ \max(\alpha) > [k-1] $. By the maximality of $ W $, $ d(\xi_{W}) $ cannot contain
			$ x_{\rho_{j}} $ for $ j > j_{0} $. Therefore,
			$ \alpha' := \alpha + d\left(x_{W \setminus j_{0}} \cdot (\ell_{i_{0}} - \ell_{j_{0}} )
			\cdot \xi_{W} \right) $ is a representative of $ \beta $ with
			$ \max(\alpha') < \max(\alpha) $.
			This contradiction implies that $ \max(\alpha) \leq [k-1] $. 

			As $ \beta $ is a cycle, we also find that $ d \alpha \in 
			\bigoplus_{i=0}^{\dim \tau -1} \mathbf{x}_{i} \cdot H_{\bullet}(K^{sq}(\mathbb{Q}[\Star(\tau)])) $.
			Therefore, $ \alpha  \in 
			\bigoplus_{i=0}^{\dim \tau -1} \mathbf{x}_{i} \cdot H_{\bullet}(K^{sq}(\mathbb{Q}[\Star(\tau)])) $,
			and $ \beta = 0 $.
		\end{proof}

		Combined with Lemma \ref{lem:starIsomorphism}, this shows that
		\[ \coker\left[ \pi_{\tau}^{*}\colon
		H^{\bullet}\left(X_{\overline{\Star}(\tau)}\right)
		\to H^{\bullet}\left(X_{\overline{\Star}(\rho_{\tau})}\right) \right]   \cong
		\frac{H^{\bullet}(X_{\Star(\rho_{\tau})})}{H^{\bullet}(X_{\Star(\tau)})}
		\cong \bigoplus_{i=1}^{\dim \tau -1} H^{\bullet - 2i}(X_{\Star(\tau)}).\]

		The last lemma we need before the proof of Theorem \ref{thm:blowup} is a piece of homological 
		algebra that we will use to compare Mayer--Vietoris exact sequences.

		\begin{lem}\label{lem:LESBlowup}
			Suppose that $ A_{-}^{\bullet} $, $ B_{-}^{\bullet} $, 
			$ C_{-}^{\bullet} $, $ A_{+}^{\bullet} $, $ B_{+}^{\bullet} $,
			$ C_{+}^{\bullet} $ are graded modules such that $ \epsilon\colon
			B_{-}^{\bullet} \to B_{+}^{\bullet} $ is an injection,
			$ \phi\colon C_{-}^{\bullet} \to C_{+}^{\bullet} $ is an isomorphism,
			and these modules fit into a commutative diagram of long exact
			sequences
			\begin{center}
			\begin{tikzcd}
				\cdots \rar & C_{+}^{\bullet-1} \rar["\gamma_{+}"] & 
							A_{+}^{\bullet} \rar["\alpha_{+}"] & B_{+}^{\bullet} \rar["\beta_{+}"]
							& C_{+}^{\bullet} \rar["\gamma_{+}"] & \cdots \\
				\cdots \rar & C_{-}^{\bullet-1} \uar["\phi"] \rar["\gamma_{-}"]
							& A_{-}^{\bullet} \rar["\alpha_{-}"] \uar["\iota"] &
							B_{-}^{\bullet} \rar["\beta_{-}"] \uar["\epsilon"] &
							C_{-}^{\bullet} \rar["\gamma_{-}"] \uar["\phi"] & \cdots.
			\end{tikzcd}
			\end{center}
			Then, there is a long exact sequence
			\[ 
					\cdots \longrightarrow 
						\coker \phi \xrightarrow{\; \gamma_{-} \circ \phi^{-1} \circ
						\beta_{+} \; } A_{-}^{\bullet} \xrightarrow{\; \iota \;} 
						A_{+}^{\bullet}
						\xrightarrow{\; \alpha_{+} \; } \coker \phi \longrightarrow \cdots.
			\]
		\end{lem}

		\begin{proof}
			We note that the map 
			$ \gamma_{-} \circ \phi^{-1} \circ \beta_{+}\colon \coker \epsilon 
			\to A_{-}^{\bullet} $ is well-defined on the cokernel, as 
		  $ \gamma_{-} \circ \beta_{-} $ is the zero map on $ B_{-}^{\bullet}$.

			The proof of the exactness of the sequence is a diagram chase which we
			include for completeness.

			\textbf{Exactness at $ A_{-}^{\bullet} $:}
        The composition $ \iota \circ \gamma_{-} \circ \phi^{-1} \circ \beta_{+}\colon
        \coker \epsilon \to A^{\bullet}_{+} $
        is equal to $ \gamma_{+} \circ \beta_{+} = 0 $. Now suppose that $ a_{-} \in 
        \ker \iota $. Then, $ \epsilon \circ \alpha_{-} (a_{-}) = 0 $; this shows
        $ \alpha_{-}(a_{-}) =0 $, as $ \epsilon $ is injective. Choose $ c_{-} 
				\in C^{\bullet +1}_{-} $ such
        that $ \gamma_{-}(c_{-}) = a_{-} $. Then, $ \gamma_{+} \circ \phi(c_{-}) = 0 $,
        so we find $ b_{+} \in B_{+}^{\bullet +1} $ such that $ \beta_{+}(b_{+}) 
				= \phi(c_{-}) $.
        The coset $ b_{+} + \epsilon(B^{\bullet+1}_{-}) $ maps to $ a_{-} $ under
        $\gamma_{-} \circ \phi^{-1} \circ \beta_{+} $.

				\textbf{Exactness at $ A_{+}^{\bullet} $:}
        The composition $ \alpha_{+} \circ \iota \colon A_{+}^{\bullet} \to 
				\coker \epsilon $ factors through $ \epsilon\colon B^{\bullet}_{-} \to
        B^{\bullet}_{+} $ and therefore vanishes in the cokernel.
        Now suppose that $ a_{+} \in \ker (\alpha_{+} \colon A^{\bullet}_{+} \to 
				\coker \epsilon )$. Choose $ b_{-} \in 
				B^{\bullet}_{-} $
        such that $ \alpha_{+}(a_{+}) = \epsilon(b_{-}) $. 
				Note that $ \beta_{-}(b_{-}) = 0 $,
        as $ \beta_{+} \circ \alpha_{+}(a_{+}) = 0 $ and $ \phi $ is an isomorphism.
        Choose $ a_{-} \in A_{-}^{\bullet} $ such that $ \alpha_{-}(a_{-}) = b_{-} $.
        We find that $ a_{+} - \iota(a_{-}) \in \ker \alpha_{+} $, and we 
				choose $ c_{+} \in
        C^{\bullet+1}_{+} $ with $ \gamma_{+}(c_{+}) = a_{+} - \iota(a_{-}) $.
        Then, $ \iota(\gamma_{-} \circ \phi^{-1}(c_{+}) + a_{-}) = a_{+} $.

				\textbf{Exactness at $ \coker \epsilon $:} 
        The composition $ \gamma_{-} \circ \phi^{-1} \circ \beta_{+} \circ \alpha_{+} $
				factors through $ \beta_{+} \circ \alpha_{+} $
        and therefore vanishes. Now suppose that $ b_{+} + \epsilon(B^{\bullet}_{-}) \in
        \ker \gamma_{-} \circ \phi^{-1} \circ \beta_{+} $. Choose $ b_{-} \in 
				B^{\bullet}_{-} $
        such that $ \beta_{-}(b_{-}) = \phi^{-1}\circ \beta_{+}(b_{+}) $. We find 
        that $ b_{+} - \epsilon(b_{-}) \in \ker \beta_{+} $, so there is $ a_{+} \in 
				A^{\bullet}_{+}$
        such that $ \alpha_{+} (a_{+}) = b_{+} - \epsilon(b_{-}) $.  
        Thus, $ b_{+} + \epsilon(B_{\bullet}^{-}) $ is in the image of
        $ \alpha_{+} \colon A_{+}^{\bullet} \to \coker \epsilon $.
		\end{proof}

		We are now ready to prove Theorem \ref{thm:blowup}.

		\begin{proof}[Proof of Theorem \ref{thm:blowup}]
			Let $ \Sigma' = 
				\Sigma \setminus \tau $ be the subfan of $ \Sigma $ consisting of 
		  all cones which do not contain $ \tau $. Similarly, define
			$ \Sigma^{*} \setminus \rho_{\tau} $ to be the subfan of all cones which
			do not contain $ \rho_{\tau} $. By the construction 
			of the stellar subdivison, $ \Sigma' = 
			\Sigma^{*} \setminus \rho_{\tau} $, and this produces coverings
			\[ \Sigma = \overline{\Star}(\tau) \cup \Sigma', \;\;\;  
						\Sigma^{*}(\tau) 				= \overline{\Star}(\rho_{\tau})
						\cup \Sigma', \;\;\; \textrm{and} \;\;\; 
				\overline{\Star}(\tau) \cap \Sigma' = \overline{\Star}(\rho_{\tau})
						\cap \Sigma'.
			\]
			We view $ \overline{\Star}(\tau) $ as a tubular neighborhood of the 
			exceptional divisor and $ \Sigma' $ as the complement of the exceptional divisor.

			The blow-down map $ \pi_{\tau}\colon X_{\Sigma^{*}(\tau)} 
			\to X_{\Sigma} $ restricts to maps
			$ X_{\Sigma'} \to X_{\Sigma'} $,
			$ X_{\overline{\Star}(\rho_{\tau})} \to
			X_{\overline{\Star}(\tau)} $, and 
			$ X_{\Sigma' \cap \overline{\Star}(\rho_{\tau})}
			\to X_{\Sigma' \cap \overline{\Star}(\tau)} $.
			On cohomology, $ \pi_{\tau} $ induces commuting
			maps between the Mayer--Vietoris
			exact sequences 
			\begin{center}
							\begin{tikzcd}[cramped]
							\cdots \rar &  
							H^{\bullet}\left( X_{\Sigma^{*}(\tau)} \right) \rar 
							& H^{\bullet}\left(X_{\Sigma'} \right) \oplus
								H^{\bullet} \left( X_{\overline{\Star}(\rho_{\tau})} \right)
								\rar
							& H^{\bullet}\left(X_{\Sigma' \cap 
							\overline{\Star}(\rho_{\tau})}\right) \rar & \cdots \\
				\cdots \rar 
											&  H^{\bullet}\left( X_{\Sigma} \right)\rar \uar["\pi_{\tau}^{*}"] &
					H^{\bullet}\left(X_{\Sigma'} \right) \oplus
								H^{\bullet} \left( X_{\overline{\Star}(\tau)} \right)		 
											\rar \uar["\pi_{\tau}^{*}"] &
							 H^{\bullet}\left(X_{\Sigma' \cap 
											\overline{\Star}(\tau)}\right)\rar \uar["\pi_{\tau}^{*}"] & \cdots.
			\end{tikzcd}
			\end{center}
			From Lemmas \ref{lem:starIsomorphism} and 
			\ref{lem:cohomologyExceptionalDivisor}, the  map
			$H^{\bullet}\left(X_{\Sigma'} \right) \oplus
								H^{\bullet} \left( X_{\overline{\Star}(\tau)} \right)
				\to H^{\bullet}\left(X_{\Sigma'} \right) \oplus
								H^{\bullet} \left( X_{\overline{\Star}(\rho_{\tau})} \right) $
			is injective and the cokernel is isomorphic to 
			$\bigoplus_{i=1}^{\dim \tau -1} H^{\bullet - 2i}
			\left( X_{\Star(\tau)} \right) $.
			The hypotheses of Lemma \ref{lem:LESBlowup} are satisfied, and this
			produces a long exact sequence
			\begin{equation}\label{eq:unsplitLES}
						\cdots \to H^{i}\left( X_{\Sigma}\right) 
						\to H^{i}\left(X_{\Sigma^{*}(\tau)}\right) 
						\to \bigoplus_{i=1}^{\dim \tau -1} H^{\bullet - 2i}
						\left( X_{\Star(\tau)} \right) \to \cdots.
			\end{equation}
			The following lemma shows that the map 
			$H^{i}(X_{\Sigma^{*}(\tau)}) \to \bigoplus_{i=1}^{\dim \tau -1}
			H^{\bullet -2i}(X_{\Star(\tau)}) $ is surjective and the long exact sequence splits.		
		\end{proof}

		\begin{definition}
			For $ 1 \leq i < \dim \tau  $, let 
			$ \Psi^{i}\colon H^{\bullet -2i}(X_{\Star(\tau)}) 
			\to H^{\bullet}(X_{\Sigma^{*}(\tau)}) $ be given on Koszul homology by 
			$ \xi \mapsto x_{\rho_{\tau}}^{i}	\cdot \xi $. 
		\end{definition}

		\begin{lem}\label{lem:LESSplit}
			The image of $ \bigoplus_{i=1}^{\dim \tau -1} \Psi^{i} $
			surjects onto
			$ \bigoplus_{i=1}^{\dim \tau -1}
			H^{\bullet -2i}(X_{\Star(\tau)}) $.
		\end{lem}

		\begin{proof}
			The map $ H^{\bullet}(X_{\Sigma^{*}(\tau)}) \to \bigoplus_{i=1}^{\dim \tau -1}
			H^{\bullet - 2i}(X_{\Star}(\tau)) $ in \eqref{eq:unsplitLES} factors as
			\[
					H^{\bullet}\left( X_{\Sigma^{*}(\tau)} \right) \to \frac{H^{\bullet}(X_{\overline{\Star}
						(\rho_{\tau})})}{H^{\bullet}(X_{\overline{\Star}(\tau)})} \cong
					\bigoplus_{i=1}^{\dim \tau -1}
					H^{\bullet -2i}(X_{\Star(\tau)}).
			\]
			By Lemmas \ref{lem:starIsomorphism} and \ref{lem:cohomologyExceptionalDivisor}, 
			$ H^{\bullet}(X_{\overline{\Star} (\rho_{\tau})}) / H^{\bullet}(X_{\overline{\Star}(\tau)}) $
			is generated in Koszul homology by elements of the form $ x_{\rho}^{j} \cdot \xi $
			where $ \rho \in \tau(1) $, $ 1 \leq j \leq \dim \tau -1 $, and  $ \xi $ a cycle in
			$ K_{\bullet}(\mathbb{Q}[\Star(\tau)]) $. By induction on the exponent
			$j$, we will show that $ x_{\rho}^{j} \cdot \xi $ is in the image of 
			$ \bigoplus_{i=1}^{\dim \tau-1} \Psi^{i} $.

			For $ j = 1 $, take $ x_{\rho_{\tau}}  \cdot
			\xi  \in K_{\bullet}(\mathbb{Q}[\Sigma^{*}(\tau)]) $. This is 
			a cycle, and its image in 
			$ K_{\bullet}(\mathbb{Q}[\overline{\Star}(\rho_{\tau})]) $
			is $ x_{\rho_{\tau}} \cdot \xi $. Choose $ \ell \in N^{\vee} $
			such that, in $ K_{\bullet}(\mathbb{Q}[\overline{\Star}(\rho_{\tau})]) $,
			$ d \ell = x_{\rho} + x_{\rho_{\tau}} + \mathbf{y} $ with $ \mathbf{y}
						\in \mathbb{Q}[\Star(\tau)] $. Thus, $ d (\xi \cdot \ell) = 
			 x_{\rho} \cdot \xi + x_{\rho_{\tau}} \cdot \xi + \mathbf{y} \cdot 
			\xi $. This shows that
			$ - x_{\rho_{\tau}} \cdot \xi $ is homologous to $ x_{\rho} \cdot \xi $
			in the cokernel
			$ H_{\bullet}(K(\mathbb{Q}[\overline{\Star}(\rho_{\tau})]))/
			H_{\bullet}(K(\mathbb{Q}[\overline{\Star}(\tau)])) $.

			For induction, take $ x_{\rho}^{j} \cdot \xi $
			with $ j > 1 $. Then,  $ x_{\rho_{\tau}}^{j} \cdot \xi $
			as a cycle in $K_{\bullet}(\mathbb{Q}[\Sigma^{*}(\tau)]) $ whose image
			in $ K_{\bullet}(\mathbb{Q}[\overline{\Star}(\rho_{\tau})]) $ is 
			$ x_{\rho_{\tau}}^{j} \cdot \xi $. The boundary
			$ d \left( \sum_{m=1}^{j} (-1)^{m} x_{\rho_{\tau}}^{j-m} (x_{\rho}
			+ \mathbf{y})^{m-1} \cdot \xi \right)$
			shows the class of $ \pm x_{\rho_{\tau}}^{j} \cdot \xi $
			is homologous to the class of $ x_{\rho}^{j} \cdot \xi $ in 
			$ H^{\bullet}(X_{\overline{\Star} (\rho_{\tau})}) / H^{\bullet}(X_{\overline{\Star}(\tau)}) $ 
			up to the sum of classes
			of the form $ x_{\rho}^{k} \cdot \xi_{k} $ for $ k < j $, with 
			$ \xi_{k} $ a cycle in $ K_{\bullet}(\mathbb{Q}[\Star(\tau)]) $.
			By the inductive hypothesis, the class of
			$ x_{\rho}^{j} \cdot \xi  $
			is therefore in the image of $ \bigoplus_{i=1}^{\dim \tau -1} \Psi^{i} $.
		\end{proof}

		We can strengthen the exact sequence in Theorem \ref{thm:blowup} by representing the cohomology of 
		$ X_{\Sigma^{*}(\tau)} $ in terms
		of the Koszul homology of $ K_{\bullet}(\mathbb{Q}[\Sigma]) $
		and $ K_{\bullet}(\mathbb{Q}[\Star(\tau)]) $. 		

		\begin{cor}\label{cor:blowup}
			Let $ X_{\Sigma^{*}(\tau)} $ be the blow-up of 
			$ X_{\Sigma} $ along $ V(\tau) $. Then, the map
			of complexes
			\[ 
					K_{\bullet}(\mathbb{Q}[\Sigma]) \oplus \bigoplus_{i=1}^{\dim \tau -1}
					x_{\rho_{\tau}}^{i} \cdot  K_{\bullet}(\mathbb{Q}[\Star(\tau)])
					\xrightarrow{} K_{\bullet}(\mathbb{Q}[\Sigma^{*}(\tau)])
						\]
						is a quasi-isomorphism.
		\end{cor}

		\begin{proof}
			By Theorem \ref{thm:blowup}, the homologies of the two complexes have the
			same dimension. The map $ K_{\bullet}(\mathbb{Q}[\Sigma]) \to
			K_{\bullet}(\mathbb{Q}[\Sigma^{*}(\tau)]) $ represents the 
			map $ H^{\bullet}(X_{\Sigma}) \to H^{\bullet}(X_{\Sigma^{*}(\tau)}) $.
			By the proof of Lemma \ref{lem:LESSplit}, the cycles of 
			$ \bigoplus_{i=1}^{\dim \tau -1}
			x_{\rho_{\tau}}^{i} \cdot  K_{\bullet}(\mathbb{Q}[\Star(\tau)])$
			surject onto $  \bigoplus_{i=1}^{\dim \tau -1}
			H^{\bullet -2i}(X_{\Star(\tau)})$. 
			Theorem \ref{thm:blowup} implies 
			\[ 
				K_{\bullet}(\mathbb{Q}[\Sigma]) \oplus \bigoplus_{i=1}^{\dim \tau -1}
				x_{\rho_{\tau}}^{i} \cdot  K_{\bullet}(\mathbb{Q}[\Star(\tau)])
				\xrightarrow{} K_{\bullet}(\mathbb{Q}[\Sigma^{*}(\tau)])
			\]
			induces an surjection on homology and, by dimension, an 
			isomorphism.
		\end{proof}

	\bibliography{bibliography}

@article {wondertopes,
    AUTHOR = {Brauner, Sarah and Eur, Christopher and Pratt, Elizabeth and
              Vlad, Raluca},
     TITLE = {Wondertopes},
   JOURNAL = {Adv. Math.},
  FJOURNAL = {Advances in Mathematics},
    VOLUME = {480},
      YEAR = {2025},
     PAGES = {Paper No. 110516, 41},
      ISSN = {0001-8708,1090-2082},
   MRCLASS = {14H10 (05E14 14H81 14N20)},
  MRNUMBER = {4957545},
       DOI = {10.1016/j.aim.2025.110516},
       URL = {https://doi.org/10.1016/j.aim.2025.110516},
}

@book {brunsHerzog,
    AUTHOR = {Bruns, Winfried and Herzog, J\"urgen},
     TITLE = {Cohen-{M}acaulay rings},
    SERIES = {Cambridge Studies in Advanced Mathematics},
    VOLUME = {39},
 PUBLISHER = {Cambridge University Press, Cambridge},
      YEAR = {1993},
     PAGES = {xii+403},
      ISBN = {0-521-41068-1},
   MRCLASS = {13H10 (13-02)},
  MRNUMBER = {1251956},
MRREVIEWER = {Matthew\ Miller},
}

@article {deligne,
    AUTHOR = {Deligne, Pierre},
     TITLE = {Th\'eorie de {H}odge. {III}},
   JOURNAL = {Inst. Hautes \'Etudes Sci. Publ. Math.},
  FJOURNAL = {Institut des Hautes \'Etudes Scientifiques. Publications
              Math\'ematiques},
    NUMBER = {44},
      YEAR = {1974},
     PAGES = {5--77},
      ISSN = {0073-8301,1618-1913},
   MRCLASS = {14C30 (14F15)},
  MRNUMBER = {498552},
MRREVIEWER = {J.\ H. M. Steenbrink},
       URL = {http://www.numdam.org/item?id=PMIHES_1974__44__5_0},
}

@article {totaro,
    AUTHOR = {Totaro, Burt},
     TITLE = {Chow groups, {C}how cohomology, and linear varieties},
   JOURNAL = {Forum Math. Sigma},
  FJOURNAL = {Forum of Mathematics. Sigma},
    VOLUME = {2},
      YEAR = {2014},
     PAGES = {Paper No. e17, 25},
      ISSN = {2050-5094},
   MRCLASS = {14C15 (14F42 14M20 14M25)},
  MRNUMBER = {3264256},
MRREVIEWER = {Jinhyun\ Park},
       DOI = {10.1017/fms.2014.15},
       URL = {https://doi.org/10.1017/fms.2014.15},
}

@book {CLS,
    AUTHOR = {Cox, David A. and Little, John B. and Schenck, Henry K.},
     TITLE = {Toric varieties},
    SERIES = {Graduate Studies in Mathematics},
    VOLUME = {124},
 PUBLISHER = {American Mathematical Society, Providence, RI},
      YEAR = {2011},
     PAGES = {xxiv+841},
      ISBN = {978-0-8218-4819-7},
   MRCLASS = {14M25 (05A15 05E45 52B12)},
  MRNUMBER = {2810322},
MRREVIEWER = {Ivan\ Arzhantsev},
       DOI = {10.1090/gsm/124},
       URL = {https://doi.org/10.1090/gsm/124},
}

@article {FM,
    AUTHOR = {Feichtner, Eva Maria and M\"uller, Irene},
     TITLE = {On the topology of nested set complexes},
   JOURNAL = {Proc. Amer. Math. Soc.},
  FJOURNAL = {Proceedings of the American Mathematical Society},
    VOLUME = {133},
      YEAR = {2005},
    NUMBER = {4},
     PAGES = {999--1006},
      ISSN = {0002-9939,1088-6826},
   MRCLASS = {06A11 (05E25 57N80)},
  MRNUMBER = {2117200},
MRREVIEWER = {Michael\ J.\ Falk},
       DOI = {10.1090/S0002-9939-04-07731-7},
       URL = {https://doi.org/10.1090/S0002-9939-04-07731-7},
}

@article {franzFu,
    AUTHOR = {Franz, Matthias and Fu, Xin},
     TITLE = {Cohomology of smooth toric varieties: naturality},
   JOURNAL = {J. Pure Appl. Algebra},
  FJOURNAL = {Journal of Pure and Applied Algebra},
    VOLUME = {228},
      YEAR = {2024},
    NUMBER = {6},
     PAGES = {Paper No. 107590, 38},
      ISSN = {0022-4049,1873-1376},
   MRCLASS = {14M25 (14F45 55N91 57S12)},
  MRNUMBER = {4687383},
MRREVIEWER = {Andreas\ Hochenegger},
       DOI = {10.1016/j.jpaa.2023.107590},
       URL = {https://doi.org/10.1016/j.jpaa.2023.107590},
}

@misc{EFMPV,
      title={Building sets, Chow rings, and their Hilbert series}, 
      author={Christopher Eur and Luis Ferroni and Jacob P. Matherne and Roberto Pagaria and Lorenzo Vecchi},
      year={2025},
      eprint={2504.16776},
      archivePrefix={arXiv},
      primaryClass={math.CO},
      url={https://arxiv.org/abs/2504.16776}, 
}

@book {millerSturmfels,
    AUTHOR = {Miller, Ezra and Sturmfels, Bernd},
     TITLE = {Combinatorial commutative algebra},
    SERIES = {Graduate Texts in Mathematics},
    VOLUME = {227},
 PUBLISHER = {Springer-Verlag, New York},
      YEAR = {2005},
     PAGES = {xiv+417},
      ISBN = {0-387-22356-8},
   MRCLASS = {13-01 (05-01 05E99 13D02 14M15 14M25)},
  MRNUMBER = {2110098},
MRREVIEWER = {Joseph\ Gubeladze},
}

@article {AHK,
    AUTHOR = {Adiprasito, Karim and Huh, June and Katz, Eric},
     TITLE = {Hodge theory for combinatorial geometries},
   JOURNAL = {Ann. of Math. (2)},
  FJOURNAL = {Annals of Mathematics. Second Series},
    VOLUME = {188},
      YEAR = {2018},
    NUMBER = {2},
     PAGES = {381--452},
      ISSN = {0003-486X,1939-8980},
   MRCLASS = {05B35 (05E99 14C25 14T05)},
  MRNUMBER = {3862944},
MRREVIEWER = {Zvi\ Rosen},
       DOI = {10.4007/annals.2018.188.2.1},
       URL = {https://doi.org/10.4007/annals.2018.188.2.1},
}

@book {oxley,
    AUTHOR = {Oxley, James G.},
     TITLE = {Matroid theory},
    SERIES = {Oxford Science Publications},
 PUBLISHER = {The Clarendon Press, Oxford University Press, New York},
      YEAR = {1992},
     PAGES = {xii+532},
      ISBN = {0-19-853563-5},
   MRCLASS = {05B35 (90C27)},
  MRNUMBER = {1207587},
MRREVIEWER = {Talmage\ J.\ Reid},
}

@article {bettiNumbersFacetIdeal,
    AUTHOR = {Johnsen, Trygve and Roksvold, Jan and Verdure, Hugues},
     TITLE = {Betti numbers associated to the facet ideal of a matroid},
   JOURNAL = {Bull. Braz. Math. Soc. (N.S.)},
  FJOURNAL = {Bulletin of the Brazilian Mathematical Society. New Series.
              Boletim da Sociedade Brasileira de Matem\'atica},
    VOLUME = {45},
      YEAR = {2014},
    NUMBER = {4},
     PAGES = {727--744},
      ISSN = {1678-7544,1678-7714},
   MRCLASS = {13F20 (05B35 05C25 05E40 13D02)},
  MRNUMBER = {3296189},
MRREVIEWER = {Christopher\ A.\ Francisco},
       DOI = {10.1007/s00574-014-0071-9},
       URL = {https://doi.org/10.1007/s00574-014-0071-9},
}

@incollection {bjorner,
    AUTHOR = {Bj\"orner, Anders},
     TITLE = {The homology and shellability of matroids and geometric
              lattices},
 BOOKTITLE = {Matroid applications},
    SERIES = {Encyclopedia Math. Appl.},
    VOLUME = {40},
     PAGES = {226--283},
 PUBLISHER = {Cambridge Univ. Press, Cambridge},
      YEAR = {1992},
      ISBN = {0-521-38165-7},
   MRCLASS = {52B40 (05B35 55N99)},
  MRNUMBER = {1165544},
MRREVIEWER = {Michel\ Yves\ Jambu},
       DOI = {10.1017/CBO9780511662041.008},
       URL = {https://doi.org/10.1017/CBO9780511662041.008},
}

@article {weberWeights,
    AUTHOR = {Weber, Andrzej},
     TITLE = {Weights in the cohomology of toric varieties},
   JOURNAL = {Cent. Eur. J. Math.},
  FJOURNAL = {Central European Journal of Mathematics},
    VOLUME = {2},
      YEAR = {2004},
    NUMBER = {3},
     PAGES = {478--492},
      ISSN = {1895-1074,1644-3616},
   MRCLASS = {14M25 (14F43)},
  MRNUMBER = {2113544},
MRREVIEWER = {Nicolas\ Pouyanne},
       DOI = {10.2478/BF02475240},
       URL = {https://doi.org/10.2478/BF02475240},
}

@article {alexanderDuality,
    AUTHOR = {Bj\"orner, Anders and Tancer, Martin},
     TITLE = {Note: {C}ombinatorial {A}lexander duality---a short and
              elementary proof},
   JOURNAL = {Discrete Comput. Geom.},
  FJOURNAL = {Discrete \& Computational Geometry. An International Journal
              of Mathematics and Computer Science},
    VOLUME = {42},
      YEAR = {2009},
    NUMBER = {4},
     PAGES = {586--593},
      ISSN = {0179-5376,1432-0444},
   MRCLASS = {55U30 (05E45)},
  MRNUMBER = {2556456},
       DOI = {10.1007/s00454-008-9102-x},
       URL = {https://doi.org/10.1007/s00454-008-9102-x},
}

@article {ardilaKlivans,
    AUTHOR = {Ardila, Federico and Klivans, Caroline J.},
     TITLE = {The {B}ergman complex of a matroid and phylogenetic trees},
   JOURNAL = {J. Combin. Theory Ser. B},
  FJOURNAL = {Journal of Combinatorial Theory. Series B},
    VOLUME = {96},
      YEAR = {2006},
    NUMBER = {1},
     PAGES = {38--49},
      ISSN = {0095-8956,1096-0902},
   MRCLASS = {05B35},
  MRNUMBER = {2185977},
MRREVIEWER = {Neil\ L.\ White},
       DOI = {10.1016/j.jctb.2005.06.004},
       URL = {https://doi.org/10.1016/j.jctb.2005.06.004},
}

@incollection {brylawskiOxley,
    AUTHOR = {Brylawski, Thomas and Oxley, James},
     TITLE = {The {T}utte polynomial and its applications},
 BOOKTITLE = {Matroid applications},
    SERIES = {Encyclopedia Math. Appl.},
    VOLUME = {40},
     PAGES = {123--225},
 PUBLISHER = {Cambridge Univ. Press, Cambridge},
      YEAR = {1992},
      ISBN = {0-521-38165-7},
   MRCLASS = {05C15 (05B35)},
  MRNUMBER = {1165543},
MRREVIEWER = {W.\ T.\ Tutte},
       DOI = {10.1017/CBO9780511662041.007},
       URL = {https://doi.org/10.1017/CBO9780511662041.007},
}

@article {weightPolynomialsMatroids,
    AUTHOR = {Johnsen, Trygve and Roksvold, Jan and Verdure, Hugues},
     TITLE = {A generalization of weight polynomials to matroids},
   JOURNAL = {Discrete Math.},
  FJOURNAL = {Discrete Mathematics},
    VOLUME = {339},
      YEAR = {2016},
    NUMBER = {2},
     PAGES = {632--645},
      ISSN = {0012-365X,1872-681X},
   MRCLASS = {05B35 (13D02 13F55)},
  MRNUMBER = {3431376},
MRREVIEWER = {Xiangqian\ Zhou},
       DOI = {10.1016/j.disc.2015.10.005},
       URL = {https://doi.org/10.1016/j.disc.2015.10.005},
}

@article {FY,
    AUTHOR = {Feichtner, Eva Maria and Yuzvinsky, Sergey},
     TITLE = {Chow rings of toric varieties defined by atomic lattices},
   JOURNAL = {Invent. Math.},
  FJOURNAL = {Inventiones Mathematicae},
    VOLUME = {155},
      YEAR = {2004},
    NUMBER = {3},
     PAGES = {515--536},
      ISSN = {0020-9910,1432-1297},
   MRCLASS = {14C15 (14M25)},
  MRNUMBER = {2038195},
MRREVIEWER = {G.\ K.\ Sankaran},
       DOI = {10.1007/s00222-003-0327-2},
       URL = {https://doi.org/10.1007/s00222-003-0327-2},
}

@incollection {brion,
    AUTHOR = {Brion, Michel},
     TITLE = {Piecewise polynomial functions, convex polytopes and
              enumerative geometry},
 BOOKTITLE = {Parameter spaces ({W}arsaw, 1994)},
    SERIES = {Banach Center Publ.},
    VOLUME = {36},
     PAGES = {25--44},
 PUBLISHER = {Polish Acad. Sci. Inst. Math., Warsaw},
      YEAR = {1996},
   MRCLASS = {14M25 (14N10 52A39 52B20)},
  MRNUMBER = {1481477},
}

@book {griffithsHarris,
    AUTHOR = {Griffiths, Phillip and Harris, Joseph},
     TITLE = {Principles of algebraic geometry},
    SERIES = {Pure and Applied Mathematics},
 PUBLISHER = {Wiley-Interscience [John Wiley \& Sons], New York},
      YEAR = {1978},
     PAGES = {xii+813},
      ISBN = {0-471-32792-1},
   MRCLASS = {14-01},
  MRNUMBER = {507725},
MRREVIEWER = {Gerhard\ Pfister},
}

@book {weibel,
    AUTHOR = {Weibel, Charles A.},
     TITLE = {An introduction to homological algebra},
    SERIES = {Cambridge Studies in Advanced Mathematics},
    VOLUME = {38},
 PUBLISHER = {Cambridge University Press, Cambridge},
      YEAR = {1994},
     PAGES = {xiv+450},
      ISBN = {0-521-43500-5; 0-521-55987-1},
   MRCLASS = {18-01 (16-01 17-01 20-01 55Uxx)},
  MRNUMBER = {1269324},
MRREVIEWER = {Kenneth\ A.\ Brown},
       DOI = {10.1017/CBO9781139644136},
       URL = {https://doi.org/10.1017/CBO9781139644136},
}

@article {matroidsM2,
    AUTHOR = {Chen, Justin},
     TITLE = {Matroids: a {M}acaulay2 package},
   JOURNAL = {J. Softw. Algebra Geom.},
  FJOURNAL = {Journal of Software for Algebra and Geometry},
    VOLUME = {9},
      YEAR = {2019},
    NUMBER = {1},
     PAGES = {19--27},
      ISSN = {1948-7916},
   MRCLASS = {05B35 (05C31 52B40)},
  MRNUMBER = {3919085},
       DOI = {10.2140/jsag.2019.9.19},
       URL = {https://doi.org/10.2140/jsag.2019.9.19},
}

@Misc{M2,
          author = {Grayson, Daniel R. and Stillman, Michael E.},
          title = {Macaulay2, a software system for research in algebraic geometry},
          howpublished = {Available at {http://www2.macaulay2.com}}
        }

@article {franzRing,
    AUTHOR = {Franz, Matthias},
     TITLE = {The cohomology rings of smooth toric varieties and quotients
              of moment-angle complexes},
   JOURNAL = {Geom. Topol.},
  FJOURNAL = {Geometry \& Topology},
    VOLUME = {25},
      YEAR = {2021},
    NUMBER = {4},
     PAGES = {2109--2144},
      ISSN = {1465-3060,1364-0380},
   MRCLASS = {14M25 (14F45 55N91)},
  MRNUMBER = {4286370},
MRREVIEWER = {Jie\ Wu},
       DOI = {10.2140/gt.2021.25.2109},
       URL = {https://doi.org/10.2140/gt.2021.25.2109},
}
\end{document}